\documentclass[a4paper, twoside, 11pt, english]{article}
 
\usepackage{amsmath, amsthm, amsfonts, amssymb, mathabx}
\usepackage[T1]{fontenc}
\usepackage[utf8]{inputenc}
\usepackage{newtxtext, courier, euler}
\usepackage[scaled=0.95]{helvet}
\usepackage[cal=euler, scr=boondoxo, scrscaled=1.05]{mathalfa}
\usepackage{fancyhdr, titlesec, url, enumerate, setspace}
\usepackage{tikz}
\usepackage[all,knot,poly]{xy}
\usepackage[ruled]{algorithm2e}

\usepackage{babel}
\usepackage{hyperref}

\titleformat{\section}[block]{\scshape\filcenter\LARGE}{\thesection.}{.5em}{}
\titleformat{\subsection}[block]{\bfseries\filcenter\large}{\thesubsection.}{.5em}{\medskip}
\titleformat{\subsubsection}[runin]{\bfseries}{\thesubsubsection.}{.5em}{}[.]
\titlespacing{\subsubsection}{0pt}{\topsep}{.5em}

\newtheoremstyle{ntheorem}%
	{\topsep}{\topsep}{\itshape}{0pt}{\bfseries}{.}{.5em}%
	{\thmnumber{#2.\hspace{.5em}}\thmname{#1}\thmnote{ (#3)}}
	
\newtheoremstyle{ndefinition}%
	{\topsep}{\topsep}{\normalfont}{0pt}{\bfseries}{.}{.5em}%
	{\thmnumber{#2.\hspace{.5em}}\thmname{#1}\thmnote{ (#3)}}
	
\newtheoremstyle{nremark}%
	{\topsep}{\topsep}{\normalfont}{0pt}{\itshape}{.}{.5em}%
	{\thmnumber{}\thmname{#1}\thmnote{ (#3)}}

\theoremstyle{ntheorem}
  	\newtheorem{theorem}[subsubsection]{Theorem}

  	\newtheorem{proposition}[subsubsection]{Proposition}
	\newtheorem{lemma}[subsubsection]{Lemma}

\theoremstyle{ndefinition}

\fancypagestyle{plain}{
\fancyhf{}\fancyfoot[LC,RC]{}
\fancyfoot[LE,RO]{$\thepage$}

}

\newcommand{\fll}{\xymatrix@1@C=22pt{\ar [r] &}}
\newcommand{\ofl}[1]{\xymatrix@1@C=22pt{\ar [r] ^-*+{#1} &}}

\renewcommand{\phi}{\varphi}
\renewcommand{\epsilon}{\varepsilon}

\newcommand{\K}{\mathbb{K}}
\newcommand{\Nb}{\mathbb{N}}

\newcommand{\Mr}{\mathcal{M}}

\renewcommand{\Pr}{\mathcal{P}}

\newcommand{\Ur}{\mathcal{U}}

\newcount\hh
\newcount\mm
\mm=\time
\hh=\time
\divide\hh by 60
\divide\mm by 60
\multiply\mm by 60
\mm=-\mm
\advance\mm by \time
\def\hhmm{\number\hh:\ifnum\mm<10{}0\fi\number\mm}

\def\cone{\mathrm{cone}}
\renewcommand{\div}[1]{\mathcal{#1}}

\def\Q{\mathbb{Q}}
\def\Rb{\mathbb{R}}

\newcommand{\comp}[1]{#1^{\complement}} 
\newcommand{\compp}[2]{#2^{\complement(#1)}}

\def\mult{\mathrm{Mult}}
\def\cmult{\mathrm{^{\complement}Mult}}
\def\cnonmult{\mathrm{^{\complement}NMult}}
\def\nonmult{\mathrm{NMult}}
\def\normf{\mathrm{nf}}

\DeclareMathOperator{\lm}{lm}
\DeclareMathOperator{\lc}{lc}
\DeclareMathOperator{\lt}{lt}
\newcommand{\ideal}[1]{\mathrm{Id}(#1)}

\SetAlgorithmName{Procedure}{List of procedures}

\DeclareMathOperator{\autoreduce}{{\bf Autoreduce}}
\DeclareMathOperator{\integralcond}{{\bf IntCond}}
\DeclareMathOperator{\complete}{{\bf Complete}}
\DeclareMathOperator{\janet}{{\bf Janet}}

\renewcommand{\leq}{\leqslant}
\renewcommand{\geq}{\geqslant}

\def\cwo{\preccurlyeq_{cwo}}
\def\cwostrict{\prec_{cwo}}
\def\wo{\preccurlyeq_{wo}}
\def\wostrict{\prec_{wo}}
\def\jo{\preccurlyeq_{J}}

\newcommand{\auteur}[3]{
\noindent
\begin{minipage}[t]{.45\textwidth}
\begin{flushright}
\textsc{#1} \\
{\footnotesize\textsf{#2}}
\end{flushright} 
\end{minipage}
\qquad
\begin{minipage}[t]{.45\textwidth}
#3
\end{minipage}
}

\begin{document}
\thispagestyle{empty}

\begin{center}

\begin{doublespace}
\begin{huge}
{\scshape From analytical mechanical problems}

\vskip+2pt

{\scshape to rewriting theory through M. Janet}
\end{huge}

\bigskip
\hrule height 1.5pt 
\bigskip

\begin{Large}
{\scshape Kenji Iohara \qquad Philippe Malbos}
\end{Large}
\end{doublespace}

\vspace{2cm}

\begin{small}\begin{minipage}{14cm}
\noindent\textbf{Abstract --}
This note surveys the historical background of the Gr\"{o}bner basis theory for $D$-modules and linear rewriting theory. The objective is to present a deep interaction of these two fields largely developed in algebra throughout the twentieth century. We recall the work of M. Janet on the algebraic analysis on linear partial differential systems that leads to the notion of involutive division. We present some generalizations of the division introduced by M. Janet and their relations with Gr\"{o}bner basis theory.

\smallskip\noindent\textbf{M.S.C. 2010 -- 01-08, 13P10, 12H05, 35A25, 58A15, 68Q42.} 
\end{minipage}\end{small}

\vspace{1cm}

\begin{small}\begin{minipage}{12cm}
\renewcommand{\contentsname}{}
\setcounter{tocdepth}{2}
\tableofcontents
\end{minipage}
\end{small}
\end{center}

\clearpage

\section*{Introduction}

Several lectures of the Kobe-Lyon summer school\footnote{Summer School \emph{On Quivers : Computational Aspects and Geometric Applications}, July 21-31, 2015, Kobe, Japan.} recalled a deep interaction between Gr\"{o}bner bases for $D$-modules and linear rewriting theory. The objective of this note is to survey the historical background of these two fields largely developed in algebra throughout the twentieth century and to present their deep relations. 
Completion methods are the main streams for these computational theories.
In Gr\"{o}bner bases theory, they were motivated by algorithmic problems in elimination theory such as computations in quotient polynomial rings modulo an ideal, manipulating algebraic equations and computing Hilbert series. In rewriting theory, they were motivated by computation of normal forms and linear basis for algebras and computational problems in homological algebra.

\bigskip

In this note we present the precursory ideas of the french mathematician M. Janet on the algebraic formulation of completion methods for polynomial systems. Indeed, the problem of completion already appear in the seminal work of M. Janet in 1920 in his thesis~\cite{Janet20}, that proposed a very original approach by formal methods in the study of linear partial differential equations systems, PDE systems for short. Its constructions were formulated in terms of polynomial systems, but without the notion of ideal and of Noetherian induction. These two notions were introduced by E. Noether in 1921~\cite{Noether1921} for commutative rings.

\bigskip

The work of M. Janet was forgotten for about a half-century. It was rediscovered by F. Schwarz in 1992 in \cite{Schwarz92}. Our exposition in this note does not follow the historical order. The first section deals with the problems that motivate the questions on PDE undertaken by M. Janet. In Section~\ref{Subsection:JanetWork}, we present completion for monomial PDE systems as introduced by Janet in his monograph~\cite{Janet29}. This completion used an original division procedure on monomials. In Section~\ref{Section:InvolutiveDivision}, we present axiomatisation of this Janet's division, called \emph{involutive division}, and due to V. P. Gerdt. The last two sections concern the case of polynomial PDE systems, with the Janet's completion method used to reduce a linear PDE system to a canonical form and the axiomatisation of the reductions involved in terms of rewriting theory.

\subsection*{From analytical mechanical problems to involutive division}

\subsubsection*{From Lagrange to Janet} 
The analysis on linear PDE systems was mainly motivated in 18th century by resolution of analytical mechanical problems. The seminal work of J.-L. Lagrange gave the first systematic study of PDE systems launched by such problems. The case of PDE of one unknown function of several variables has been treated by J. F. Pfaff. The Pfaff problem will be recalled in~\ref{SS:PfaffProblem}.  This theory was developed in two different directions: toward the general theory of differential invariants and the existence of solutions under given initial conditions. The differential invariants approachs will be discussed in~\ref{SS:PfaffProblem} and~\ref{Subsubsection:InvolutionFrobenius}.  The question of the existence of solution satisfying some initial conditions  was formulated in the Cauchy-Kowalevsky theorem recalled in~\ref{SSS:CauchyKowalevskyTheorem}. 

\subsubsection*{Exterior differential systems} 
Following the work of H. Grassmann in 1844 exhibiting the rules of the exterior algebra computation, \'E. Cartan introduced exterior differential calculus in 1899. This algebraic calculus allowed him to describe a PDE system by an exterior differential system that is independent of the choice of coordinates. This leaded to the so called Cartan-Kähler theory, that we will review in Section \ref{SS:CartanKahlerTheory}. We will present a geometrical property of involutivity on exterior differential systems in~\ref{Subsubsection:InvolutionCartan}, that motivates  the formal methods introduced by M. Janet for analysis on linear PDE systems.

\subsubsection*{Generalizations of Cauchy-Kowalevsky's theorem}
Another origin of the work of M. Janet is the Cauchy-Kowalevsky's theorem that gives the initial conditions of solvability of a family of PDE systems that we describe in~\ref{SSS:CauchyKowalevskyTheorem}. \'E. Delassus, C. Riquier and M. Janet attempted to generalize this result to a wider class of linear PDE systems which in turn leaded them to introduce the computation of a notion of normal form for such systems.

\subsubsection*{The Janet monograph}
Section~\ref{Subsection:JanetWork} presents the historical motivations that leaded M. Janet to introduce an algebraic algorithm in order to compute normal form of linear PDE systems. In particular, we recall the problem of computation of \emph{inverse of derivation} introduced in the monograph of M. Janet, $\ll$~\emph{Leçons sur les systèmes d'équations aux dérivées partielles} $\gg$ on the analysis on linear PDE systems, published in 1929,~\cite{Janet29}. 
In this monograph M. Janet introduced formal methods based on polynomial computations for analysis on linear PDE systems. He developed an algorithmic approach for analyzing ideals in the polynomial ring $\K[\frac{\partial}{\partial x_1}, \ldots, \frac{\partial}{\partial x_n}]$ of differential operators with constant coefficients. Having the ring isomorphism between this ring and the ring $\K[x_1,\ldots,x_n]$ of polynomials with $n$ variables in mind, M. Janet gave its algorithmic construction in this latter ring.
He began by introducing some remarkable properties of monomial ideals. In particular, he recovered the Dickson's Lemma, \cite{Dickson13}, on the finiteness generation of monomial ideal. This property is essential for Noetherian properties on the set of monomials. Note that, M. Janet wasn't familiar with the axiomatisation of the algebraic structure of ideal and the property of Noetherianity already introduced by E. Noether in \cite{Noether1921} and \cite{Noether23}. Note also that the Dickson lemma was published in 1913 in a paper on numbers theory in an American journal. Due to the first world war, it would take long times before these works were accessible to french mathematical community. The Janet's algebraic constructions given in his monograph will be recalled in Section~\ref{Subsection:JanetWork} for monomial systems and in Section~\ref{Section:PolynomialPartialDifferentialEquationsSystems} for polynomial systems.

\subsubsection*{Janet's multiplicative variables}
The computation on monomial and polynomial ideals performed by M. Janet are founded on the notion of \emph{multiplicative variable} that he introduced in his thesis, \cite{Janet20}. Given an ideal generated by a set of monomials, he distinguished the family of monomials contained in the ideal and those contained in the complement of the ideal. The notion of multiplicative and non-multiplicative variables appear in order to stratify these two families of monomials. We will recall this notion of multiplicativity on variables in~\ref{Subsection:JanetMultiplicativeVariables}. This leads to a refinement of the classical division on monomials, nowadays called \emph{Janet's division}.

\subsubsection*{Involutive division and Janet's completion procedure}
The notion of multiplicative variable is local in the sense that it is defined with respect to a subset $\Ur$ of the set of all monomials. A monomial $u$ in $\Ur$ is said to be a Janet divisor of a monomial $w$ with respect to~$\Ur$, if $w=uv$ and all variables occurring in $v$ are multiplicative with respect to $\Ur$. In this way, we distinguish the set $\cone_\div{J}(\Ur)$ of monomials having a Janet divisor in~$\Ur$, called \emph{multiplicative} or \emph{involutive cone} of $\Ur$, to the set $\cone(\Ur)$ of multiple of monomials in $\Ur$ for the classical division. The Janet division being a refinement of the classical division, the set $\cone_\div{J}(\Ur)$ is a subset of $\cone(\Ur)$. M. Janet called a set of monomials $\Ur$ \emph{complete} when this inclusion is an equality.

For a monomial PDE system $(\Sigma)$ of the form 
\[
\frac{\partial^{\alpha_1+\ldots+\alpha_n}\quad}{\partial x_1^{\alpha_1}\ldots \partial x_n^{\alpha_n}} \varphi = f_{\alpha}(x_1,x_2,\ldots, x_n),
\]
where $(\alpha_1,\ldots,\alpha_n)$ belongs to a subset $I$ of $\Nb^n$, M. Janet associated the set of monomials \linebreak $\lm(\Sigma) =\{x_1^{\alpha_1}\ldots x_n^{\alpha_n}\;\;|\;\;(\alpha_1,\ldots,\alpha_n)\in I\}$. The compatibility conditions of the system $(\Sigma)$ corresponds to the factorizations of the monomials $ux$ in $\cone_\div{J}(\lm(\Sigma))$, where $u$ is in $\lm(\Sigma)$ and $x$ is a non-multiplicative variable of $u$ with respect to $\lm(\Sigma)$, in the sense given in~\ref{Subsubsection:CalculInverseDerivation}. By definition, for any monomial $u$ in $\lm(\Sigma)$ and $x$ non-multiplicative variable of $u$ with respect to $\lm(\Sigma)$, the monomial $ux$ admits such a factorization if and only if $\lm(\Sigma)$ is complete, see Proposition~\ref{Proposition:Completeness}.

The main procedure presented in Janet's monograph \cite{Janet29} completes in  finite number of operations a finite set of monomials $\Ur$ into a complete set of monomials $\widetilde{\Ur}$ that contains $\Ur$. This procedure consists in analyzing all the local default of completeness, by adding all the monomials $ux$ where $u$ belongs to $\Ur$ and $x$ is a non-multiplicative variable for $u$ with respect to $\Ur$. This procedure will be recalled in~\ref{Subsubsection:CompletionProcedure}. A generalization of this procedure to any involutive division was given by V. P. Gerdt in~\cite{Gerdt97}, and recalled in~\ref{Procedure:CompletionInvolutive}. 

Extending this procedure to a set of polynomials, M. Janet applied it to linear PDE systems, giving a procedure that transforms a linear PDE system into a complete PDE system having the same set of solutions. This construction is presented in Section~\ref{Subsection:ReductionPDEsystemToCanonicalForm}. In Section~\ref{Section:PolynomialInvolutiveBases}, we present such a procedure for an arbitrary involutive division given by V. P. Gerdt and Y. A. Blinkov in~\cite{GerdtBlinkovYuri98} and its relation to the Buchberger completion procedure in commutative polynomial rings, \cite{Buchberger65}.

\subsubsection*{The space of initial conditions}
In order to stratify the complement of the involutive cone $\cone_\div{J}(\Ur)$ M. Janet introduced the notion of \emph{complementary monomial}, see~\ref{Subsubsection:ComplementaryMonomials}, as the monomials that generate this complement in a such a way that the involutive cone of $\Ur$ and the involutive cone of the set $\comp{\Ur}$ of complementary monomials form a partition of the set of all monomials, see Proposition~\ref{Proposition:PartitionMr}.

For each complementary monomial $v$ in $\comp{\lm(\Sigma)}$ each analytic function in the multiplicative variables of~$v$ with respect to $\comp{\lm(\Sigma)}$ provides an initial condition of the PDE system $(\Sigma)$ as stated by Theorem~\ref{Theorem:BoundaryConditions}.

\subsubsection*{Polynomial partial differential equations systems}
In Section~\ref{Section:PolynomialPartialDifferentialEquationsSystems}, we present the analysis on polynomial PDE systems as M. Janet described in his monograph,~\cite{Janet29}. To deal with polynomials he defined some total orders on the set of derivatives, corresponding to total orders on the set of monomials. We recall them in Section~\ref{Subsection:ParametricPrincipalDerivative}.  The definitions on monomial orders given by M. Janet clarified the same notion previously introduced by C. Riquier in~\cite{Riquier93}. In particular, he made more explicit the notion of parametric and principal derivatives in order to distinguish the leading derivative in a polynomial PDE. In this way, he extended its algorithms on monomial PDE systems to the case of polynomial PDE systems. In particular, using these notions, he defined the property for a polynomial PDE system to be complete. Namely, a polynomial PDE system is complete if the associated set of monomials corresponding to leading derivatives of the system is complete. Moreover, he extended also the notion of complementary monomials to define the notion of \emph{initial condition} for a polynomial PDE system as in the monomial case. 

\subsubsection*{Initial conditions}
In this way, the notion of completeness is a suitable framework to discuss the existence and the unicity of the initial conditions for a linear PDE system. M. Janet proved that if a linear polynomial PDE system of the form 
\[
D_i\varphi=\sum_{j}a_{i,j}D_{i,j}\varphi, \quad i \in I,
\]
of one unknown function~$\varphi$ and all the functions $a_{i,j}$ are supposed to be analytic in a neighborhood of a point $P$  in $\mathbb{C}^n$ and is complete with respect to some a total order, then it admits at most one analytic solution satisfying the initial condition formulated in terms of complementary monomials,  see Theorems~\ref{thm_exist-EDP1} and \ref{thm_exist-EDP2}.

\subsubsection*{Integrability conditions} 
A linear polynomial PDE system of the above form is said to be \emph{completely integrable} if it admits an analytic solution for any given initial condition. M. Janet gave an algebraic characterization of complete integrability by introducing integrability conditions formulated in terms of factorization of leading derivative of the PDE by non-multiplicative variables. These integrability conditions are given explicitly in~\ref{Subsubsection:IntegrabilityConditions} as generalization to the polynomial situation of the integrability conditions formulated above for monomial PDE systems in Subsection~\ref{SS:InverseDerivation}. M. Janet proved that a linear polynomial PDE system is completely integrable if and only if any integrability condition is trivial, as stated in Theorem~\ref{Theorem:CaracterizationCompleteIntegrability}.

\subsubsection*{Janet's procedure of reduction of linear PDE systems to a canonical form}
In order to extend algorithmically Cauchy-Kowalevsky's theorem on the existence and uniqueness of solutions of initial condition problems as presented in~\ref{SSS:CauchyKowalevskyTheorem}, M. Janet considered normal forms of linear PDE systems with respect to a suitable total order on derivatives, satisfying some analytic conditions on coefficients and a complete integrability condition on the system, as defined in~\ref{Subsubsection:CanonicalSystems}. Such normal forms of PDE systems are called \emph{canonical} by M. Janet.

Procedure~\ref{Subsubsection:JanetCompletionProcedure} is the \emph{Janet's procedure} that decides if a linear PDE system can be transformed into a completely integrable system. If the system cannot be reduced to a canonical form, the procedure returns the obstructions of the system to be transformed into a completely integrable system.
This procedure depends on a total order on derivatives of unknown functions of the PDE system. 
For this purpose, M.~Janet introduced a general method to define a total order on derivatives using a parametrization of a weight order on variables and unknown functions, as recalled in \ref{sect-gen}. The Janet procedure uses a specific weight order called canonical and defined in~\ref{Subsubsection:CanonicalWeightOrder}.

The first step of Janet's procedure consists in applying \emph{autoreduction procedure}, defined in~\ref{SSS:ProcedureAutoreduce}, in order to reduce any PDE of the system with respect to the total order on derivatives. Namely two PDE of the system cannot have the same leading derivative, and any PDE of the system is reduced with respect to the leading derivatives of the others PDE, as specified in Procedure~\ref{A:RightReduce}.

The second step consists in applying the \emph{completion procedure}, Procedure~\ref{Procedure:CompletePDESystem}. That is, the set of leading derivatives of the system defines a complete set of monomials in the sense given in~\ref{CompletenessWithJanetOrdering}.

Having transformed the PDE system to an autoreduced and complete system, one can discuss about its integrability conditions.  M. Janet shown that this set of integrability conditions is  a finite set of relations that does not contain principal derivative, as explained in~\ref{Subsubsection:IntegrabilityConditions}. Hence, these integrability conditions are $\div{J}$-normal forms and uniquely defined. By Theorem~\ref{Theorem:CaracterizationCompleteIntegrability}, if all of these normal forms are trivial, then the system is completely integrable. Otherwise, if there is a non-trivial condition in the set of integrability conditions that contains only unknown functions and variables, then this relation imposes a relation on the initial conditions of the system, else if there is no such relation, the procedure is applied again on the PDE system completed by all the integrability conditions.
Note that this procedure depends on the Janet division and on a total order on the set of derivatives.

By this algorithmic method, M. Janet has generalized in certain cases Cauchy-Kowalevsky's theorem at the time where the algebraic structures have not been introduced to compute with polynomial ideals. This is pioneering work in the field of formal approaches to analysis on PDE systems. Algorithmic methods to deals with polynomial ideals were developed throughout the twentieth century and extended to wide range of algebraic structures. In the next subsection, we present some milestones on these formal mathematics.

\subsection*{Constructive methods and rewriting in algebra through the twentieth century}
\label{Section:ConstructiveMethodsRewritingTheory}

The constructions developed by M. Janet in his formal theory of linear partial differential equation systems are based on the structure of ideal, that he called \emph{module of forms}. This notion corresponds to those introduced previously by D. Hilbert in \cite{Hilbert1890} with the terminology of \emph{algebraic form}. Notice that N. M. Gunther dealt with such a structure in \cite{Gunther13b}.
The axiomatization of the notion of ideal on an arbitrary ring were given by E. Noether in \cite{Noether1921}.
As we will explain in this note, M. Janet introduced algorithmic methods to compute a family of generators of an ideal having the involutive property and called \emph{involutive bases}. This property is used to obtain a normal form of linear partial differential equation systems.

Janet's computation of involutive bases is based on a refinement of classical polynomial division called \emph{involutive division}. He defined a division that was suitable for reduction of linear partial differential equation systems. Thereafter, other involutive divisions were studied in particular by J. M. Thomas \cite{Thomas37} and by J.-F. Pommaret \cite{Pommaret78}, we refer to Section \ref{SS:OtherInvolutiveApproaches} for a discussion on these divisions.

The main purpose is to complete a generating family of an ideal into an involutive bases with respect to a given involutive division. This completion process is quite similar to those introduced with the classical division in Gr\"{o}bner bases theory. In fact, involutive bases appears to be particular cases of Gr\"{o}bner bases. The principle of completion had been developed independently in rewriting theory, that proposes a combinatorial approach of equivalence relation motivated by several computational and decision problems in algebra, computer science and logic.

\subsubsection*{Some milestones on algebraic rewriting and constructive algebra}
The main results in the work of M. Janet rely on constructive methods in linear algebra using  the principle of computing normal forms by rewriting and the principle of completion of a generating set of an ideal. These two principles have been developed during all of the twentieth century in many algebraic contexts with different formulations and at several occasions. We review below some important milestones in this long and wealth history from Kronecker to the more recent developments.
\begin{description}
\item[1882.] L. Kronecker introduced the notion of resultant of polynomials in \cite{Kronecker92} and gave the first result in elimination theory using this notion.
\item[1886.] K. Weierstrass proved a fundamental result called \emph{preparation theorem} on the factorization of analytic functions by polynomials. As an application he showed a division theorem for rings of convergent series, \cite{Weierstrass1886}.
\item[1890.] D. Hilbert proved that any ideals of a ring of commutative polynomials on a finite set of variables over a field and ring of integers are finitely generated, \cite{Hilbert1890}. This is the first formulation of the Hilbert basis theorem stating that a polynomial ring over a Noetherian ring is Noetherian.
\item[1913.] In a paper on number theory, L.E. Dickson proved a monomial version of the Hilbert basis theorem by a combinatorial method, {\cite[Lemma A]{Dickson13}}.
\item[1913.] In a serie of forgotten papers, N. G\"unther develop algorithmic approaches for polynomials rings, \cite{Gunther13a,Gunther13b,Gunther13c}.  A review of the G\"unther theory can be found in \cite{Gunther41}.
\item[1914.] M. Dehn described the word problem for finitely presented groups, \cite{Dehn10}. Using systems of transformations rules, A. Thue studied the problem for finitely presented semigroups,~\cite{Thue14}. It was only much later in 1947, that the problem for finitely presented monoids was shown to be undecidable, independently by E. L. Post~\cite{Post47} and A. Markov~\cite{Markov47a,Markov47b}.
\item[1916.] F. S. Macaulay was one of the pioneers in commutative algebra. In his book \emph{The algebraic theory of modular systems}, \cite{Macaulay16}, following the fundamental Hilbert basis theorem, he initiated an algorithmic approach to treat generators of polynomial ideals. In particular, he introduced the notion of \emph{$H$-basis} corresponding to a monomial version of Gr\"{o}bner bases.
\item[1920.] M. Janet defended his doctoral thesis, \cite{Janet20}, that presents a formal study of systems of partial differential equations following works of Ch. Riquier and \'E. Delassus. In particular, he analyzed completly integrable systems and  Hilbert functions on polynomial ideals. 
\item[1921.] In her seminal paper, \emph{Idealtheorie in Ringbereichen}, \cite{Noether1921}, E. Noether gave the foundation of general commutative ring theory, and gave one of the first general definitions of a commutative ring. She also formulated the theorem of finite chains {\cite[Satz I, \emph{Satz von der endlichen Kette}]{Noether1921}}.
\item[1923.] E. Noether stated in \cite{Noether23,Noether24} concepts of elimination theory in the language of ideals that she had introduced in \cite{Noether1921}.
\item[1926.] G. Hermann, a student of E. Noether \cite{Hermann26}, initiated purely algorithmic approaches on ideals, such as ideal membership problem and primary decomposition ideals. This work appears as a fundamental contribution for emergence of computer algebra.
\item[1927.] F. S. Macaulay showed in \cite{Macaulay27} that the Hilbert function of a polynomial ideal $I$ is equal to the Hilbert function of the monomial ideal generated by the set of leading monomials of polynomials in $I$ with respect a monomial order. As a consequence the coefficients of the Hilbert function of a polynomial ideal are polynomial for sufficiently big degree.
\item[1937.] Based on early works by Ch. Riquier and M. Janet, in \cite{Thomas37} J. M. Thomas reformulated in the algebraic language of B. L. van der Waerden, \emph{Moderne Algebra}, \cite{Waerden30, Waerden31}, the theory of normal forms of systems of partial differential equations.
\item[1937.] In \cite{Grobner37}, W. Gr\"obner formulated the isomorphism between the ring of polynomials with coefficients in an arbitrary field and the ring of differential operators with constant coefficients, see Proposition~\ref{Proposition:IsomorphismPartialX}. The identification of these two rings was used before in the algebraic study of systems of partial differential equations but without being explicit.
\item[1942.] In a seminal paper on rewriting theory, M. Newman presented rewriting as a combinatorial approach to study equivalence relations, \cite{Newman42}. He proved a fundamental rewriting result stating that under termination hypothesis, the confluence properties is equivalent to local confluence.
\item[1949.] In its monograph \emph{Moderne algebraische {G}eometrie. {D}ie idealtheoretischen {G}rundlagen}, \cite{Grobner49},  W. Gr\"{o}bner surveyed algebraic computation on ideal theory with applications to algebraic geometry.
\item[1962.] A. Shirshov introduced in \cite{Shirshov62} an algorithmic method to compute normal forms in a free Lie algebra with respect to a family of elements of the Lie algebra satisfying a confluence property, the method is based on a completion procedure and he proved a version of Newman's lemma for Lie alegbras, called \emph{composition lemma}. He deduced a constructive proof of the Poincaré-Birkhoff-Witt theorem.
\item[1964.] H. Hironaka introduced in \cite{Hironaka64} a division algorithm and introduced the notion of \emph{standard basis}, that is analogous to the notion of Gr\"{o}bner basis, for power series rings in order to solve problems of resolution of singularities in algebraic geometry.
\item[1965.] Under the supervision of W. Gr\"{o}bner, B. Buchberger developed in his PhD thesis the algorithmic theory of Gr\"{o}bner bases for commutative polynomial algebras, \cite{Buchberger65,Buchberger70,Buchberger06}.
Buchberger gave a characterization of Gr\"{o}bner bases in terms of \emph{$S$-polynomials}  and an algorithm to compute such bases, with a complete implementation in the assembler language of the computer ZUSE Z 23 V.
\item[1967.] D. Knuth and P. Bendix defined in \cite{KnuthBendix70} a completion procedure that complete with respect to a termination a set of equations in an algebraic theory into a confluent term rewriting system. The procedure is similar to the Buchberger's completion procedure. We refer the reader to  \cite{Buchberger87} for an historical account on critical-pair/completion procedures. 
\item[1972.] H. Grauert introduced in \cite{Grauert72} a generalization of Weierstrass's preparation division theorem in the language of Banach algebras.
\item[1973.] M. Nivat formulated a critical pair lemma for string rewriting systems and proved that for a terminating rewriting system, the local confluence is decidable, \cite{Nivat73}.
\item[1976, 1978.] L. Bokut in \cite{Bokut76} and G. Bergman in \cite{Bergman78} extended Gr\"obner bases and Buchberger algorithm to associative algebras. They obtained the confluence Newman's Lemma for rewriting systems in free associative algebras compatible with a monomial order, called respectively Diamond Lemma for ring theory and composition Lemma.
\item[1978.] J.-F. Pommaret introduced in \cite{Pommaret78} a global involutive division simpler than those introduced by M. Janet.
\item[1980.] F.-O. Schreyer in his PhD thesis \cite{Schreyer80} gave a method that computes syzygies in commutative multivariate polynomial rings using the division algorithm, see {\cite[Theorem 15.10]{Eisenbud95}}.
\item[1980.] G. Huet gave in~\cite{Huet80} a proof of Newman's lemma using a Noetherian well-founded induction method.
 \item[1985.] Gr\"{o}bner basis theory was extended to Weyl algebras by A. Galligo in~\cite{Galligo85}.
  \item[1997.] V. P. Gerdt and Y. A. Blinkov introduced in \cite{Gerdt97, GerdtBlinkovYuri98} the notion of involutive monomial division and its axiomatization.
  \item[2005.] V. P. Gerdt in \cite{Gerdt05} presented and analyzed an efficient involutive algorithm for computing Gröbner bases.
 \item[1999, 2002.] J.-C. Faugère developed efficient algorithms for computing Gr\"{o}bner bases, algorithm F4, \cite{Faugere99} then and algorithm F5, \cite{Faugere02}. 
 \item[2012.] T. Bächler, V. P. Gerdt, M. Lange-Hegermann and D. Robertz algorithmized in \cite{BachlerGerdtLange-HegermannRobertz12} the Thomas decomposition of algebraic and differential systems.
\end{description}

\subsection*{Conventions and notations}
The set of non-negative integers is denoted by $\mathbb{N}$.
In this note, $\K[x_1,\ldots,x_n]$ denotes the polynomial ring on the variables $x_1,\ldots,x_n$ over a field $\K$ of characteristic zero. 
For a subset $G$ of polynomials of~$\K[x_1,\ldots,x_n]$, we will denote by $\ideal{G}$ the ideal of~$\K[x_1,\ldots,x_n]$ generated by $G$. A polynomial is either zero or it can be written as a sum of a finite number of non-zero \emph{terms}, each term being the product of a scalar in $\K$ by a \emph{monomial}.

\subsubsection*{Monomials}
We will denote by $\Mr(x_1,\ldots,x_n)$ the set of monomials in the ring $\K[x_1,\ldots,x_n]$. For a subset $I$ of $\{x_1,\ldots,x_n\}$ we will denote by $\Mr(I)$ the set of monomials in $\Mr(x_1,\ldots,x_n)$ whose variables lie in $I$. A monomial $u$ in $\Mr(x_1,\ldots,x_n)$ is written as $u=x_1^{\alpha_1}\ldots x_n^{\alpha_n}$, were the $\alpha_i$ are non-negative integers. The integer $\alpha_i$ is called the \emph{degree} of the variable $x_i$ in $u$, it will be also denoted by $\deg_i(u)$. For $\alpha=(\alpha_1,\ldots,\alpha_n)$ in $\mathbb{N}^n$, we denote $x^\alpha=x_1^{\alpha_1}\ldots x_n^{\alpha_n}$ and $|\alpha|=\alpha_1+\ldots+\alpha_n$.

For a finite set $\Ur$ of monomials of $\Mr(x_1,\ldots,x_n)$ and $1\leq i\leq n$, we denote by $\deg_i(\Ur)$ the largest possible degree in variable $x_i$ of the monomials in $\Ur$, that is
\[
\deg_i(\Ur) =
\max\big(\deg_i(u)\;|\;u\in \Ur\,\big).
\]
We call the \emph{cone} of a set $\Ur$ of monomials of $\Mr(x_1,\ldots,x_n)$ the set of all multiple of monomials in $\Ur$ defined by
\[
\cone(\Ur) = \bigcup_{u\in \Ur} u\Mr(x_1,\ldots,x_n) =
\{\, uv \;|\; u \in \Ur,\; v \in \Mr(x_1,\ldots, x_n) \,\}.
\]

\subsubsection*{Homogeneous polynomials}
An \emph{homogenous polynomial} of $\K[x_1,\ldots,x_n]$ is a polynomial all of whose non-zero terms have the same degree. An homogenous polynomial is of \emph{degree} $p$ all of whose non-zero terms have degree $p$. We will denote by $\K[x_1,\ldots, x_n]_p$ the space of homogenous polynomials of degree $p$. The dimension of this space is given by the following formula:
\[
\Gamma_n^p:=\dim \big(\,\K[x_1,\ldots, x_n]_p\,\big)=\frac{(p+1)(p+2)\ldots (p+n-1)}{1\cdot 2\cdot \ldots \cdot(n-1)}. 
\]

\subsubsection*{Monomial order}
Recall that a \emph{monomial order} on $\Mr(x_1,\ldots,x_n)$ is a relation $\preccurlyeq$ on $\Mr(x_1,\ldots,x_n)$ satisfying the following three conditions
\begin{enumerate}[{\bf i)}]
\item $\preccurlyeq$ is a total order on $\Mr(x_1,\ldots,x_n)$,
\item $\preccurlyeq$ is compatible with multiplication, that is, if $u\preccurlyeq u'$, then $uw\preccurlyeq u'w$ for any monomials $u,u',w$ in $\Mr(x_1,\ldots,x_n)$,
\item $\preccurlyeq$ is a well-order on $\Mr(x_1,\ldots,x_n)$, that is, every nonempty subset of $\Mr(x_1,\ldots,x_n)$ has a smallest element with respect to $\preccurlyeq$.
\end{enumerate} 

The \emph{leading term}, \emph{leading monomial} and \emph{leading coefficient} of a polynomial $f$ of $\K[x_1,\ldots,x_n]$, with respect to a monomial order $\preccurlyeq$, will be denoted respectively by $\lt_\preccurlyeq(f)$, $\lm_\preccurlyeq(f)$ and $\lc_\preccurlyeq(f)$. For a set $F$ of polynomials in $\K[x_1,\ldots,x_n]$, we will denote by $\lm_\preccurlyeq(F)$ the set of leading monomials of the polynomials in $F$. For simplicity, we will use notations $\lt(f)$, $\lm(f)$, $\lc(f)$ and $\lm(F)$ if there is no possible confusion.

\section{Exterior differential systems}
\label{Section:ExteriorDifferentialSystems}

Motivated by problems in analytical mechanics, L. Euler (1707 - 1783) and J.-L. Lagrange (1736 - 1813) initiated the so-called \emph{variational calculus}, cf. \cite{Lagrange88}, which led to the problem of solving partial differential equations, PDE for short. In this section, we briefly explain the evolutions of these theory to serve as a guide to the M. Janet contributions. We present the historical background of exterior differential systems and of the questions on PDE. For a deeper discussion of the theory of differential equations and the Pfaff problem, we refer the reader to \cite{Forsyth90, Weber00} or \cite{Cartan1899}.

\subsection{Pfaff's problem}
\label{SS:PfaffProblem}

\subsubsection{Partial differential equations for one unknown function}

In 1772, J.-L. Lagrange \cite{Lagrange72} considered a PDE of the following form 
\begin{equation}
\label{Equation:Lagrange0}
F(x,y,\varphi,p,q)=0
\quad\text{with}\quad 
p=\frac{\partial \varphi}{\partial x}
\quad\text{and}\quad
q=\frac{\partial \varphi}{\partial y},
\end{equation}
i.e., a PDE of one unknown function $\varphi$ of two variables $x$ and $y$.  Lagrange's method to solve this PDE can be summarized as follows.
\begin{enumerate}[{\bf i)}]
\item Express the PDE (\ref{Equation:Lagrange0}) in the form 
\begin{equation}
\label{Equation:Lagrange}
q=F_1(x,y,\varphi,p)
\quad\text{with}\quad
p=\frac{\partial \varphi}{\partial x}
\quad\text{and}\quad q=\frac{\partial \varphi}{\partial y}.
\end{equation}
\item `Temporally, forget the fact $p=\frac{\partial \varphi}{\partial x}$' and consider the following $1$-form
\[ \Omega=d\varphi-pdx-qdy=d\varphi-pdx-F_1(x,y,\varphi,p)dy,
\]
by regarding $p$ as some (not yet fixed) function of $x,y$ and $\varphi$.
\item If there exist functions $M$ and $\Phi$ of $x,y$ and $\varphi$ satisfying $M\Omega=d\Phi$, then $\Phi(x,y,\varphi)=C$ for some constant $C$. Solving this new equation, we obtain a solution $\varphi=\psi(x,y,C)$ to the equation (\ref{Equation:Lagrange}).
\end{enumerate}

\subsubsection{Pfaffian systems}

In 1814-15, J. F. Pfaff (1765 - 1825) \cite{Pfaff15} has treated a PDE for one unknown function of $n$ variables, which was then succeeded to C. G. Jacobi (1804 - 1851) (cf. \cite{Jacobi27}).
Recall that any PDE of any order is equivalent to a system of PDE of first order. Thus we may only think of system of PDE of first order with $m$ unknown function
\[
F_k\big( x_1, \ldots, x_n, \varphi^1, \ldots, \varphi^m, \frac{\partial \varphi^a}{\partial x_i}~(1\leq a \leq m, 1\leq i\leq n)\big)=0,
\quad\text{for}\quad 
1\leq k \leq r.
\]
Introducing the new variables $p^{a}_i$, the system is defined on the space with coordinates $(x_i, \varphi^{a},p_i^{a})$ and is given by
\[ \begin{cases} F_k(x_i, \varphi^{a},p_i^{a})=0, & \\
                         d\varphi^{a}- \displaystyle{\sum_{i=1}^n} p_i^{a}dx_i=0, & \\
                         dx_1 \wedge \ldots \wedge dx_n \neq 0. & \end{cases}
\]
Noticed that the last condition means that the variables $x_1,\ldots , x_n$ are independent. Such a system is called a \emph{Pfaffian system}. One is interested in the questions, whether this system admits a solution or not, and if there exists a solution whether it is unique under some conditions. 
These questions are \emph{Pfaff's problems}.

\subsubsection{Cauchy-Kowalevsky's theorem}
\label{SSS:CauchyKowalevskyTheorem}
A naive approach to Pfaff's problems, having applications to mechanics in mind, is the question of the initial conditions.  In series of articles published in 1842, A. Cauchy (1789 - 1857) studied the system of PDE of first order in the following form: 
\[ 
\frac{\partial \varphi^{a}}{\partial t}=f_a( t, x_1,\cdots, x_n)
+\sum_{b=1}^{m}\sum_{i=1}^{n} f_{a,b}^{i}( t, x_1,\ldots, x_n) \frac{\partial \varphi^{b}}{\partial x_i},
\quad\text{for}\quad 1\leq a\leq m, 
\]
where $f_a, f_{a,b}^i$ and $\varphi^1, \ldots, \varphi^m$ are functions of $n+1$ variables   $t,x_1,\ldots,x_n$.
S. Kowalevsky (1850 - 1891) \cite{Kowalevsky75} in 1875 considered the system of PDE in the following form: for some $r_a \in \mathbb{Z}_{>0}$~($1\leq a\leq m$),
\[ 
\frac{\partial^{r_a} \varphi^{a}}{\partial t^{r_a}}=\sum_{b=1}^{m}
\sum_{\substack{j=0 \\ j+\vert \alpha\vert \leq r_a}}^{r_a-1} f_{a,b}^{j,\alpha}( t, x_1,\ldots, x_n) \frac{\partial^{j+\vert \alpha\vert} \varphi^{b}}{\partial t^j \partial x^{\alpha}},
\]
where, $f_{a,b}^{j,\alpha}$ and $\varphi^1, \ldots, \varphi^m$ are functions of $n+1$ variables $t,x_1,\ldots,x_n$, and where for \linebreak $\alpha=(\alpha_1,\cdots, \alpha_n)$ in $(\mathbb{Z}_{\geq 0})^n$, we set $\vert \alpha \vert=\sum_{i=1}^n \alpha_i$ and $\partial x^\alpha=\partial x_1^{\alpha_1} \ldots \partial x_n^{\alpha_n}$.
They showed that under the hypothesis on the analyticity of the coefficients, such a system admits a unique analytic local  solution satisfying a given initial condition, that is now called the \emph{Cauchy-Kowalevsky theorem}.

\subsubsection{Completely integrable systems}
\label{Subsubsection:InvolutionFrobenius}
A first geometric approach to this problem was taken over by G.~Frobenius (1849 - 1917) \cite{Frobenius77} and independently by G.~Darboux (1842 - 1917) \cite{Darboux82}.
Let~$X$ be a differentiable manifold of dimension $n$. We consider the Pfaffian system:
\[
\omega_i=0 
\qquad 1\leq i\leq r,
\]
where $\omega_i$'s are $1$-forms defined on a neighbourhood $V$ of a point $x$ in $X$. Suppose that the family 
\[
\{(\omega_i)_y\}_{1\leq i\leq r}\subset T_y^\ast X
\]
is linearly independent for $y$ in $V$. 
For $0\leq p \leq n$, let us denote by $\Omega_X^p(V)$ the space of differentiable $p$-forms on $V$. A \emph{$p$-dimensional distribution} $\mathcal{D}$ on $X$ is a 
subbundle of $TX$ whose fibre is of dimension $p$.
A distribution $\mathcal{D}$ is \emph{involutive} if, for any vector field $\xi$ and $\eta$ taking values in $\mathcal{D}$, the Lie bracket 
\[
[\xi, , \eta]:=\xi \eta-\eta \xi 
\]
takes values in $\mathcal{D}$ as well.
Such a Pfaffian system is called \emph{completely integrable}.

G. Frobenius and G. Darboux showed that the ideal $I$ of $\bigoplus_{p=0}^n \Omega_X^p(V)$, generated by the \linebreak $1$-forms~$\omega_1,\ldots,\omega_r$ is a differential ideal, i.e. $dI \subset I$, if and only if the distribution $\mathcal{D}$ on $V$ defined as the annihilator of $\omega_1,\ldots,\omega_r$ is involutive. 

\subsection{The Cartan-K\"{a}hler theory}
\label{SS:CartanKahlerTheory}

Here, we give a brief exposition of the so-called Cartan-K\"{a}hler theory from view point of its history. In particular, we will present the notion of systems in involution.
For the expositions by the founders of the theory, we refer the reader to 
\cite{Cartan45} and \cite{Kahler34}, for a modern introduction by leading experts, we refer to~\cite{BC3G91} and \cite{Malgrange05}.

\subsubsection{Differential forms}
H. Grassmann (1809 - 1877), \cite{Grassmann44}, introduced in 1844 the first equational formulation of the structure of exterior algebra with the anti-commutativity rules, 
\[
x\wedge y = - y \wedge x. 
\]
Using this notion, \'E. Cartan (1869 - 1951), \cite{Cartan1899} defined in 1899 the \emph{exterior differential} and \emph{differential $p$-form}. He showed that these notions are invariant with respect to any coordinate transformation. Thanks to this differential structures, several results obtained in 19th century were reformulated in a clear manner. 

\subsubsection{Exterior differential systems}
An \emph{exterior differential system} $\Sigma$ is a finite set of homogeneous differential forms, i.e. $\Sigma \subset \bigcup_p \Omega_X^p$. \'E. Cartan, \cite{Cartan01}, in 1901 studied exterior differential systems generated by $1$-forms, i.e. Pfaffian systems. Later, E. K\"{a}hler (1906 - 2000) \cite{Kahler34} generalized the Cartan theory to any differential ideal $I$ generated by an exterior differential system. By this reason, the general theory on exterior differential systems is nowadays called the \emph{Cartan-K\"{a}hler theory}.

In the rest of this subsection, we briefly describe the existence theorem for such a system. Since the argument developed here is \textit{local} and 
we need the Cauchy-Kowalevsky theorem, we assume that every function is  \textit{analytic} in $x_1, \ldots, x_n$ unless otherwise stated.

\subsubsection{Integral elements}
Let $\Sigma$ be an exterior differential system on a real analytic manifold $X$ of dimension $n$ such that the ideal generated by $\Sigma$ is an differential ideal. For $0\leq p \leq n$, set $\Sigma^p=\Sigma \cap \Omega_X^p$. We fix $x$ in $X$. For $p>0$, the pair $(E_p,x)$, for a $p$-dimensional vector subspace $E_p\subset T_xX$ is called an \emph{integral $p$-element} if $\omega\vert_{E_p}=0$ for any $\omega$ in $\Sigma^p_x:=\Sigma^p \cap \Omega_{X,x}^p$, where $\Omega_{X,x}^p$ denotes the space of differentila $p$-forms defined on a neighbourhood of $x \in X$. We denote the set of integral elements of dimension $p$ by $I\Sigma^p_x$.

An \emph{integral manifold} $Y$ is a submanifold of $X$ whose tangent space $T_{y}Y$ at any point $y$ in $Y$ is an integral element. 
Since the exterior differential system defined by $\Sigma$ is completely integrable, there exists independent $r$-functions $\varphi_1(x), \cdots, \varphi_r(x)$, called \emph{integral of motion} or \emph{first integral}, defined on a neighbourhood $V$ of a point $x \in \Ur$ such that their restrictions on $V\cap Y$ are constants.

The \emph{polar space} $H(E_p)$ of an integral element $E_p$ of $\Sigma$ at origin $x$ is the vector subspace of $T_xX$ generated by those $\xi \in T_xX$ such that $E_p+\Rb \xi$ is an integral element of $\Sigma$. 

\subsubsection{Regular integral elements}
Let $E_0$ be the real analytic subvariety of $X$ defined as the zeros of~$\Sigma^0$ and let $\Ur$ the subset of smooth points. A point in $E_0$ is called \emph{integral point}. 
A tangent vector $\xi$ in~$T_xX$ is called \emph{linear integral element} if $\omega(\xi)=0$ for any $\omega \in \Sigma_x^1$ with $x \in \Ur$.
We define inductively the properties called "regular" and "ordinary" as follows:
\begin{enumerate}
\item
The $0$th order \emph{character} is the integer $s_0=\max_{x \in \Ur} \{\dim \Rb\Sigma_x^1\}$. A point $x \in E_0$ is said to be \emph{regular} if $\dim \Rb\Sigma_x^1=s_0$, and a linear integral element $\xi \in T_xX$ is called \emph{ordinary} if  $x$ is regular. 
\item Set $E_1=\Rb \xi$, where $\xi$ is an ordinary linear integral element. The $1$st order \emph{character} is the integer~$s_1$ satisfying $s_0+s_1=\max_{x \in \Ur} \{\dim H(E_1)\}$. The ordinary integral $1$-element $(E_1,x)$ is said to be \emph{regular} if $\dim H(E_1)=s_0+s_1$.
Any integral $2$-element $(E_2,x)$ is called \emph{ordinary} if it contains at least one regular linear integral element.
\item Assume that all these are defined up to $(p-1)$th step and that $s_0+s_1+\cdots+s_{p-1}<n-p+1$.

The $p$th order \emph{character} is the integer $s_p$ satisfying 
\[
\sum_{i=0}^p s_i=\max_{x \in \Ur} \,\{\dim H(E_p)\}. 
\]
An integral $p$-element $(E_p,x)$ is said to be \emph{regular} if 
\[
\sum_{i=0}^p s_i=\dim H(E_p).
\]
The integral $p$-element $(E_p,x)$ is said to be \emph{ordinary} if it contains at least one regular integral element $(E_{p-1},x)$.
\end{enumerate}
Let $h$ be the smallest positive integer such that $\sum_{i=0}^h s_i=n-h$. In such a case, there does not exist an integral $(h+1)$-element. The integer $h$ is called the \emph{genus} of the system $\Sigma$. In such a case, for $0<p\leq h$, one has
\[ \sum_{i=0}^{p-1} s_i \leq n-p. \]

\begin{theorem} 
Let $0<p\leq h$ be an integer. 
\begin{enumerate}
\item The case $\sum_{i=0}^{p-1} s_i=n-p:$ let  $(E_p,x)$ be an ordinary integral $p$-element and let $Y_{p-1}$ be an integral manifold of dimension $p-1$ such that $(T_xY_{p-1},x)$ is a regular integral $(p-1)$-element contained in $(E_p,x)$.  Then, there exists a unique integral manifold $Y_p$ of dimension $p$ containing $Y_{p-1}$ such that $T_xY_p=E_p$.
\item The case $\sum_{i=0}^{p-1} s_i<n-p:$ let  $(E_p,x)$ be an integral $p$-element and let $Y_{p-1}$ be an integral manifold of dimension $p-1$ such that $(T_xY_{p-1},x)$ is a regular integral $(p-1)$-element contained in $(E_p,x)$. Then, for each choice of $n-p-\sum_{i=0}^{p-1} s_i$ differentiable functions on $x_1,\cdots, x_p$, there exists a unique integral manifolds $Y_p$ of dimension $p$ containing $Y_{p-1}$ such that $T_xY_p=E_p$.
\end{enumerate}
\end{theorem}
This theorem states that a given chain of ordinary integral elements
\[ (E_0,x) \subset (E_1,x) \subset \cdots \subset (E_h,x), \qquad \dim E_p=p \quad (0\leq p\leq h) , 
\]
one can inductively find an integral manifold $Y_p$ of dimension $p$ such that $Y_0=\{x\}$,  $Y_{p-1} \subset Y_p$ and~$T_xY_p=E_p$. 
Notice that to obtain $Y_p$ from $Y_{p-1}$, one applies the Cauchy-Kowalevsky theorem to the system of PDE defined by $\Sigma^p$ and the choice of arbitrary differentiable functions in the above statement provide initial data consisting of

\subsubsection{Systems in involution}
\label{Subsubsection:InvolutionCartan}
In many applications, the exterior differential systems one considers admit $p$-independent variables $x_1,\ldots, x_p$. In such a case, we are only interested in the $p$-dimensional integral manifolds among which it imposes no additional relation between $x_1, \ldots, x_p$. In general, an exterior differential system $\Sigma$ for $n-p$ unknown functions and $p$ independent variables $x_1,\ldots, x_p$ is said to be \emph{in involution} if it satisfies the two following conditions
\begin{enumerate}
\item its genus is more than or equal to $p$,
\item the defining equations of the generic ordinary integral $p$-element introduce no linear relation among~$dx_1,\ldots, dx_p$.
\end{enumerate}

\subsubsection{Reduced characters}
Consider a family $\mathcal{F}$ of integral elements of dimensions $1,2,\cdots, p-1$ than can be included in an integral $p$-element at a generic integral point $x \in X$.
Take a local chart of with origin $x$. The \emph{reduced polar system} $H^{\mathrm{red}}(E_i)$ of an integral element $x$ is the polar system of the restriction of the exterior differential system $\Sigma$ to the submanifold 
\[
\{x_1=x_2=\cdots =x_p=0\}. 
\]
The integers $s_0',s_1',\cdots, s_{p-1}'$, called the \emph{reduced characters}, are defined in such a way that $s_0'+s_1'+\cdots +s_i'$ is the dimension of the reduced polar system $H^{\mathrm{red}}(E_i)$ at a generic integral element. For convenience, one sets $s_p'=n-p-(s_0'+s_1'+\cdots+s_{p-1}')$.  \\
Let $\Sigma$ be an exterior differential system of $n-p$ unknown functions of $p$ independent variables such that the ideal generated by $\Sigma$ is an differential ideal. \'E. Cartan showed that it is a \textit{system in involution} iff the most general integeral $p$-element in $\mathcal{F}$ depends upon $s_1'+2s_2'+\cdots+ps_p'$ independent parameters.

\subsubsection{Recent developments}
In 1957, M. Kuranishi (1924- ), \cite{Kuranishi57}, considered the problem of the prolongation of a given exterior differential system and treated the cases what \'E. Cartan called total. Here, M. Kuranishi as well as \'E. Cartan studied locally in analytic category. After an algebraic approach to the integrability due to V. Guillemin and S. Sternberg, \cite{Guillemin-Sternberg64}, in 1964, I. Singer and S. Sternberg, \cite{Singer-Sternberg65}, in 1965 studied some classes of infinite dimensional which is even applicable to $C^\infty$-category. In 1970's, with the aid of Jet bundles and the Spencer cohomology, J. F. Pommaret (cf. \cite{Pommaret78}) reworked on the formal integrable involutive differential systems which generalized works of M. Janet, in the language of sheaf theory. 
For other geometric aspects not using sheaf theory, see the books by 
P.~Griffiths~(1938-),~\cite{Griffiths83}, and R. Bryant  et al., \cite{BC3G91}.

\section{Monomial partial differential equations systems}
\label{Subsection:JanetWork}

In this section, we present the method introduced by M. Janet called \emph{inverse calculation of the derivation} in his monograph~\cite{Janet29}. In~{\cite[Chapter I]{Janet29}} M. Janet considered \emph{monomial PDE}, that is PDE of the form
\begin{equation}
\label{Equation:PDEform}
\frac{\partial^{\alpha_1+\alpha_2+\ldots +\alpha_n}\varphi}{\partial x_1^{\alpha_1}\partial x_2^{\alpha_2}\ldots \partial x_n^{\alpha_n}}
=
f_{\alpha_1\alpha_2\ldots \alpha_n}(x_1,x_2,\ldots, x_n),
\end{equation}
where $\varphi$ is an unknown function and the $f_{\alpha_1\alpha_2\ldots \alpha_n}$ are several variables analytic functions.
By an algebraic method he analyzed the solvability of such an equation, namely the existence and the uniqueness of an analytic function $u$ solution of the system. Notice that the analyticity condition guarantees the commutativity of partial differentials operators. This property is crucial for the constructions that he developed in the ring of commutative polynomials. Note that the first example of PDE that does not admit any solution was found by H. Lewy in the fifties in~\cite{Lewy57}.

\subsection{Ring of partial differential operators and multiplicative variables}

\subsubsection{Historical context}
\label{Subsection:HistoricalContext}

In the beginning of 1890's, following  collaboration with C. M\'{e}ray (1835-1911), C. Riquier (1853-1929) initiated his research on finding normal forms of systems of (infinitely many) PDE for finitely many unknown functions with finitely many independent variables (see~\cite{Riquier10} and \cite{Riquier28} for more details).

In 1894, A. Tresse \cite{Tresse94} showed that such systems can be always reduced to systems of finitely many PDE. This is the first result on Noeterianity of a module over a ring of differential operators.
Based on this result, \'{E}. Delassus (1868 - 19..)
formalized and simplified Riquier's theory. In these works, one already finds an algorithmic approach analysing ideals of the ring $\K[\frac{\partial}{\partial x_1}, \ldots, \frac{\partial}{\partial x_n}]$.

It was M. Janet (1888 - 1983), already in his thesis \cite{Janet20} published in 1920, who had realized that the latter ring is isomorphic to the ring of polynomials with $n$ variables $\K[x_1, \cdots, x_n]$ at the time where several abstract notions on rings introduced by E. Noether in Germany had not been known by M. Janet in France.
It was only in 1937 that W. Gr\"{o}bner (1899-1980) proved this isomorphism.

\begin{proposition}[{\cite[Sect. 2.]{Grobner37}}]
\label{Proposition:IsomorphismPartialX}
There exists a ring isomorphism 
\[
\Phi: \K[x_1,\ldots, x_n] \longrightarrow \K[\frac{\partial}{\partial x_1}, \ldots, \frac{\partial}{\partial x_n}],
\]
from the ring of polynomials with $n$ variables $x_1, \ldots, x_n$ with coefficients in an arbitrary field $\K$ to the ring of differential operators with constant coefficients.
\end{proposition}

\subsubsection{Derivations and monomials}
M. Janet considers monomials in the variables $x_1,\ldots,x_n$ and use implicitly the isomorphism $\Phi$ of Proposition~\ref{Proposition:IsomorphismPartialX}. To a monomial $x^\alpha=x_1^{\alpha_1}x_2^{\alpha_2}\ldots x_n^{\alpha_n}$ he associates the differential operator  
\[
D^\alpha:=\Phi(x^\alpha)=
\frac{\partial^{|\alpha|}\quad}{\partial x_1^{\alpha_1}\partial x_2^{\alpha_2}\ldots \partial x_n^{\alpha_n}}.
\]
In {\cite[Chapter I]{Janet29}, M. Janet considered finite monomial PDE systems. The equations are of the form~(\ref{Equation:PDEform}) and the system having a finitely many equations, the set of monomials associated to the PDE system is finite.
The first result of the monograph is a finiteness result on monomials stating that a sequence of monomials in which none is a multiple of an earlier one is necessarily finite. He proved this result by induction on the number of variables. We can formulate this result as follows.

\begin{lemma}[{\cite[\textsection 7]{Janet29}}]
\label{Lemma:Janet1}
Let $\Ur$ be a subset of $\Mr(x_1,\ldots,x_n)$. If, for any monomials $u$ and $u'$ in~$\Ur$, the monomial $u$ does not divide $u'$, then the set $\Ur$ is finite.
\end{lemma}

This result corresponds to Dickson's Lemma, \cite{Dickson13}, which asserts that any monomial ideal of~$\K[x_1,\ldots,x_n]$ is finitely generated.

\subsubsection{Stability of the multiplication}
\label{Subsubsection:StabilityMultiplication}
M. Janet paid a special attention to families of monomials with the following property.
A subset of monomial $\Ur$ of $\Mr(x_1,\ldots,x_n)$ is called \emph{multiplicatively stable} if for any monomial $u$ in $\Mr(x_1,\ldots,x_n)$ such that there exists $u'$ in $\Ur$ that divides $u$, then $u$ is in $\Ur$. In other words, the set $\Ur$ is closed under multiplication by monomials in $\Mr(x_1,\ldots,x_n)$.

As a consequence of Lemma~\ref{Lemma:Janet1}, if $\Ur$ is a multiplicatively stable subset of $\Mr(x_1,\ldots,x_n)$, then it contains only finitely many elements which are not multiples of any other elements in $\Ur$. Hence, there exists a finite subset $\Ur_f$ of $\Ur$ such that for any $u$ in $\Ur$, there exists $u_f$ in $\Ur_f$ such that $u_f$ divides $u$.

\subsubsection{Ascending chain condition}
M. Janet observed an other consequence of Lemma~\ref{Lemma:Janet1}: the \emph{ascending chain condition} on multiplicatively stable monomial sets that he formulated as follows. Any ascending sequence of multiplicatively stable subsets of $\Mr(x_1,\ldots,x_n)$
\[
\Ur_1 \subset \Ur_2 \subset \; \ldots \; \subset \Ur_k \subset \ldots 
\]
is finite. This corresponds to the Noetherian property on the set of monomials in finitely-many variables.

\subsubsection{Inductive construction}
Let us fix a total order on variables $x_n>x_{n-1}>\ldots > x_1$.
Let $\Ur$ be a finite subset of $\Mr(x_1,\ldots, x_n)$.
Let us define, for every $0\leq \alpha_n \leq \deg_n(\Ur)$,
\[
[\alpha_n] = \{u \in \Ur\;|\; \deg_{n}(u)=\alpha_n\,\}.
\]
The family $([0],\ldots ,[\deg_n(\Ur)])$ forms a partition of $\Ur$.
We define for every $0\leq \alpha_n \leq \deg_n(\Ur)$
\[
\overline{[\alpha_n]} =\{u \in \Mr(x_1,\ldots,x_{n-1}) \;|\; ux_n^{\alpha_n} \in \Ur\,\}.
\]
We set for every $0\leq i \leq \deg_n(\Ur)$
\[
\Ur_i' = \underset{0\leq \alpha_n \leq i}{\bigcup} \{u \in \Mr(x_1,\ldots,x_{n-1}) \;|\; \text{there exists $u'\in \overline{[\alpha_n]}$ such that $u'|u$}\,\}.
\]
We set 
\[
\Ur_k =
\begin{cases} 
\{\,ux_n^k\;|\; u\in \Ur'_k\,\} & \text{if $k< \deg_n(\Ur)$,} \\ 
\{\,ux_n^k\;|\; u\in \Ur'_{\deg_n(\Ur)}\,\} & \mbox{if $k\geq \deg_n(\Ur)$.}
\end{cases} 
\]
and  $M(\Ur) = \underset{k\geq 0}{\bigcup} \Ur_k$. 
By this inductive construction, M. Janet obtains the monomial ideal generated by~$\Ur$. Indeed, $M(\Ur)$ consists in the following set of monomial
\[
\{\,u \in \Mr(x_1,\ldots,x_n) \;|\; 
\text{there exists $u'$ in $\Ur$ such that $u'|u$} \,\}.
\]

\subsubsection{Example}
Consider the subset $\Ur=\{\,x_3x_2^2,x_3^3x_1^2\,\}$ of monomials in $\Mr(x_1,x_2,x_3)$.  We have
\[
[0] = \emptyset,
\quad
[1] = \{x_3x_2^2\},
\quad
[2] = \emptyset,
\quad
[3] = \{x_3^3x_1^2\}.
\]
Hence,
\[
\overline{[0]} = \emptyset,
\quad
\overline{[1]} = \{x_2^2\},
\quad
\overline{[2]} = \emptyset,
\quad
\overline{[3]} = \{x_1^2\}.
\]
The set $M(\Ur)$ is defined using of the following subsets:
\[
\Ur_0' = \emptyset,
\quad
\Ur_1' = \{x_1^{\alpha_1}x_2^{\alpha_2}\;|\; \alpha_2\geq 2\},
\quad
\Ur_2' = \Ur_1',
\quad
\Ur_3' = \{x_1^{\alpha_1}x_2^{\alpha_2}\;|\; \alpha_1\geq 2\;\text{ou}\;\alpha_2\geq 2\}.
\]

\subsubsection{Janet's multiplicative variables, 	{\cite[\textsection 7]{Janet20}}}
\label{Subsection:JanetMultiplicativeVariables}
Let us fix a total order $x_n>x_{n-1}>\ldots >x_1$ on variables.
Let $\Ur$ be a finite subset of $\Mr(x_1,\ldots,x_n)$.
For all $1\leq i \leq n$, we define the following subset of $\Ur$:
\[
[\alpha_i,\ldots,\alpha_n] = 
\{u\in \Ur\;|\;\deg_j(u)=\alpha_j\;\;\text{for all}\;\;i\leq j \leq n\}.
\]
That is $[\alpha_i,\ldots,\alpha_n]$ contains monomials of $\Ur$ of the form $vx_i^{\alpha_i}\ldots x_n^{\alpha_n}$, with $v$ in $\Mr(x_1,\ldots,x_{i-1})$.
The sets $[\alpha_i,\ldots,\alpha_n]$, for $\alpha_i,\ldots,\alpha_n$ in $\Nb$, form a partition of $\Ur$. Moreover, for all $1\leq i \leq n-1$, we have $[\alpha_i,\alpha_{i+1},\ldots,\alpha_n] \subseteq [\alpha_{i+1},\ldots,\alpha_n]$ and the sets $[\alpha_i,\ldots,\alpha_n]$, where $\alpha_i\in \Nb$, form a partition of~$[\alpha_{i+1},\ldots,\alpha_n]$.

Given a monomial $u$ in $\Ur$, the variable $x_n$ is said to be \emph{multiplicative for $u$ in the sense of Janet} if 
\[
\deg_n(u) = \deg_n(\Ur).
\]
For $i\leq n-1$, the variable $x_i$ is said to be \emph{multiplicative for $u$ in the sense of Janet} if 
\[
u\in [\alpha_{i+1},\ldots,\alpha_n]
\qquad\text{and}\qquad
\deg_i(u)=
\deg_i([\alpha_{i+1},\ldots,\alpha_n]).
\]
We will denote by $\mult_{\div{J}}^\Ur(u)$ the set of multiplicative variables of $u$ in the sense of Janet with respect to the set~$\Ur$, also called \emph{$\div{J}$-multiplicative variables}.

Note that, by definition, for any $u$ and $u'$ in $[\alpha_{i+1},\ldots,\alpha_n]$, we have 
\[
\{x_{i+1},\ldots,x_n\} \cap \mult_\div{J}^\Ur(u) 
=
\{x_{i+1},\ldots,x_n\} \cap \mult_\div{J}^\Ur(u').
\]
As a consequence, we will denote by $\mult_\div{J}^\Ur([\alpha_{i+1},\ldots,\alpha_n])$ this set of multiplicative variables. 

\subsubsection{Example}
Consider the subset $\Ur=\{x_2x_3,x_2^2,x_1\}$ of $\Mr(x_1,x_2,x_3)$ with the order $x_3 > x_2 > x_1$. We have $\deg_3(\Ur)=1$, hence the variable $x_3$ is $\div{J}$-multiplicative for~$x_3x_2$ and not $\div{J}$-multiplicative for~$x_2^2$ and $x_1$.

For $\alpha \in \Nb$, we have 
$[\alpha]=\{u\in \Ur \;|\; \deg_3(u)=\alpha\}$, hence
\[
[0] = \{x_2^2,x_1\},
\qquad
[1]=\{x_2x_3\}.
\]
We have $\deg_2(x_2^2)=\deg_2([0])$, $\deg_2(x_1)\neq \deg_2([0])$ and 
$\deg_2(x_2x_3)=\deg_2([1])$, hence the variable $x_2$ is $\div{J}$-multiplicative for $x_2^2$ and $x_2x_3$ and not $\div{J}$-multiplicative for $x_1$. We have
\[
[0,0]=\{x_1\},
\quad
[0,2]=\{x_2^2\},
\quad 
[1,1]=\{x_2x_3\}
\]
and $\deg_1(x_2^2)=\deg_1([0,2])$, $\deg_1(x_1)= \deg_1([0,0])$ and $\deg_1(x_3x_2)=\deg_1([1,1])$, hence the variable~$x_1$ is $\div{J}$-multiplicative for $x_1$, $x_2^2$ and $x_3x_2$.

\subsubsection{Janet divisor}
\label{Definition:JanetDivisor}
Let $\Ur$ be a subset of $\Mr(x_1,\ldots,x_n)$. A monomial $u$ in $\Ur$ is called \emph{Janet divisor} of a monomial $w$ in $\Mr(x_1,\ldots,x_n)$ with respect to $\Ur$, if there is a decomposition $w=uv$, where any variable occurring in $v$ is $\div{J}$-multiplicative with respect to $\Ur$. 

\begin{proposition}
\label{Proposition:UnicityJanetDivisor}
Let $\Ur$ be a subset of $\Mr(x_1,\ldots,x_n)$ and $w$ be a monomial in $\Mr(x_1,\ldots,x_n)$. Then~$w$ admits in $\Ur$ at most one Janet divisor with respect to $\Ur$.
\end{proposition}
\begin{proof}
If $u$ is a Janet divisor of $w$ with respect to $\Ur$, there is $v$ in $\Mr(\mult_\div{J}^\Ur(u))$ such that $w=uv$. We have $\deg_n(v)=\deg_n(w)-\deg_n(u)$. If $\deg_n(w)\geq \deg_n(\Ur)$, then the variable $x_n$ is $\div{J}$-multiplicative and~$\deg_n(v)=\deg_n(w)-\deg_n(\Ur)$. If $\deg_n(w)<\deg_n(\Ur)$, then $x_n$ cannot be $\div{J}$-multiplicative and~$\deg_n(v)=0$. 

As a consequence, for any Janet divisors $u$ and $u'$ in $\Ur$ of $w$, we have $\deg_n(u)=\deg_n(u')$ and~$u,u'\in [\alpha]$ for some $\alpha\in \Nb$. 

Suppose now that $u$ and $u'$ are two distinct Janet divisor of $w$ in $\Ur$. There exists $1<k\leq n$ such that $u,u'\in [\alpha_k,\ldots,\alpha_n]$ and $\deg_{k-1}(u)\neq \deg_{k-1}(u')$. Suppose that $\deg_{k-1}(u) >\deg_{k-1}(u')$, then the variable $x_{k-1}$ cannot be $\div{J}$-multiplicative for $u'$ with respect to $\Ur$. It follows that $u'$ cannot be a Janet divisor of $w$. This leads to a contradiction, hence $u=u'$.
\end{proof}

\subsubsection{Complementary monomials}
\label{Subsubsection:ComplementaryMonomials}
Let $\Ur$ be a finite subset of $\Mr(x_1,\ldots,x_n)$. 
The set of \emph{complementary monomials of $\Ur$} is the set of monomial denoted by $\comp{\Ur}$ defined by 
\begin{equation}
\label{Equation:ComplementaryMonomials}
\comp{\Ur} = \bigcup_{1\leq i \leq n} \; \compp{i}{\Ur},
\end{equation}
where 
\[
\compp{n}{\Ur} = 
\{ x_n^\beta \; | \; 0 \leq \beta \leq \deg_n(\Ur)\;\text{and}\; 
[\beta]= \emptyset \},
\]
and for every $1\leq i < n$
\[
\compp{i}{\Ur} = 
\big\{\,
x_i^\beta x_{i+1}^{\alpha_{i+1}}\ldots x_n^{\alpha_n} \; \big| \;
[\alpha_{i+1},\ldots,\alpha_n] \neq \emptyset, 
\;
0 \leq \beta < \deg_i([\alpha_{i+1},\ldots,\alpha_n]),
\;
[\beta, \alpha_{i+1}, \ldots , {\alpha_n}] = \emptyset
\,\big\}.
\]
Note that the union in~(\ref{Equation:ComplementaryMonomials}) is disjoint, since for $i\neq j$ we have $\compp{i}{\Ur} \cap \compp{j}{\Ur} = \emptyset$.

\subsubsection{Multiplicative variables of complementary monomials}
For any monomial $u$ in $\comp{\Ur}$, we define the set $\cmult^{\comp{\Ur}}$ of \emph{multiplicative variables for $u$ with respect to complementary monomials} in $\comp{\Ur}$ as follows. If the monomial $u$ is in $\compp{n}{\Ur}$, we set
\[
\cmult_\div{J}^{\compp{n}{\Ur}}(u) = \{x_1,\ldots,x_{n-1}\}.
\]
For $1\leq i \leq n-1$, for any monomial $u$ in $\compp{i}{\Ur}$, there exists $\alpha_{i+1},\ldots,\alpha_n$ such that $u\in [\alpha_{i+1},\ldots,\alpha_n]$. Then 
\[
\cmult_\div{J}^{\compp{i}{\Ur}}(u) = \{x_1,\ldots,x_{i-1}\} \cup 
\mult_\div{J}^\Ur([\alpha_{i+1},\ldots,\alpha_n]).
\] 
Finally, for $u$ in $\comp{\Ur}$, there exists an unique $1\leq i_u \leq n$ such that $u\in \compp{i_u}{\Ur}$. Then we set
\[
\cmult_\div{J}^{\comp{\Ur}}(u) = \cmult_\div{J}^{\compp{i_u}{\Ur}}(u).
\]

\subsubsection{Example, {\cite[p. 17]{Janet29}}}
Consider the subset $\Ur=\{\,x_3^3x_2^2x_1^2,x_3^3x_1^3,x_3x_2x_1^3,x_3x_2\,\}$ of monomials in $\Mr(x_1,x_2,x_3)$ with the order $x_3>x_2>x_1$. The following table gives the multiplicative variables for each monomial:
\begin{center}
\begin{tabular}{c|ccc}
$x_3^3x_2^2x_1^2$ & $x_3$ & $x_2$ & $x_1$\\
$x_3^3x_1^3$ & $x_3$ & & $x_1$\\
$x_3x_2x_1^3$ & & $x_2$ & $x_1$\\
$x_3x_2$ & & $x_2$ & \\
\end{tabular}
\end{center}

The set of complementary monomials are
\[
\compp{3}{\Ur} = \{1,x_3^2\},
\quad
\compp{2}{\Ur} = \{x_3^3x_2,x_3\},
\quad
\compp{1}{\Ur} = \{x_3^3x_2^2x_1,x_3^3x_2^2,x_3^3x_1^2,x_3^3x_1,x_3^3,x_3x_2x_1^2,x_3x_2x_1\}.
\]
The following table gives the multiplicative variables for each monomial:
\begin{center}
\begin{tabular}{c|ccc}
$1$, $x_3^2$ &  & $x_2$ & $x_1$\\
\hline
$x_3^3x_2$ & $x_3$ & & $x_1$\\
$x_3$ & & & $x_1$\\
\hline
$x_3^3x_2^2x_1$, $x_3^3x_2^2$ & $x_3$ & $x_2$ & \\
$x_3^3x_1^2$, $x_3x_1$, $x_3^3$ & $x_3$ &  & \\
$x_3x_2x_1^2$, $x_3x_2x_1$ &  & $x_2$  & \\
\end{tabular}
\end{center}

\subsection{Completion procedure}

In this subsection, we present the notion of complete system introduced by M. Janet in \cite{Janet29}. In particular, we recall the completion procedure that he gave in order to complete a finite set of monomials.

\subsubsection{Complete systems}
\label{Subsection:CompleteSystem}
Let $\Ur$ be a set of monomials of $\Mr(x_1,\ldots,x_n)$.
For a monomial $u$ in $\Ur$ \linebreak (resp. in $\comp{\Ur}$), M.~Janet defined the \emph{involutive cone of $u$ with respect to $\Ur$} (resp. \emph{to $\comp{\Ur}$}) as the following set of monomials:
\[
\cone_\div{J}(u,\Ur) =
\{\, uv \; | \; v\in \Mr(\mult_\div{J}^\Ur(u))\,\},
\qquad
\text{(resp.}\quad 
\comp{\cone}_\div{J}(u,\Ur) =
\{\, uv \; | \; v\in \Mr(\cmult_\div{J}^{\comp{\Ur}}(u))\,\}
\text{\;)}.
\]
The \emph{involutive cone of the set $\Ur$} is defined by
\[
\cone_\div{J}(\Ur) = 
\underset{u \in \Ur}{\bigcup} \; \cone_\div{J}(u,\Ur),
\qquad
\text{(resp.}\quad 
\comp{\cone}_\div{J}(\Ur) = 
\underset{u \in \comp{\Ur}}{\bigcup} \; \comp{\cone}_\div{J}(u,\Ur)
\text{\;)}.
\]
M. Janet called \emph{complete} a set of monomials $\Ur$ when
$\cone(\Ur) = \cone_\div{J}(\Ur)$.
An involutive cone is called class in Janet's monograph~\cite{Janet29}. The terminology "\emph{involutive}" first appear in \cite{Gerdt97} by V. P. Gerdt and became standard now. We refer the reader to \cite{Mansfield96} for a discussion on relation between this notion with the notion of involutivity in the work of \'E. Cartan.

\begin{proposition}[{\cite[p. 18]{Janet29}}]
\label{Proposition:PartitionMr}
For any finite set $\Ur$ of monomials of $\Mr(x_1,\ldots,x_n)$, we have the following partition 
\[
\Mr(x_1,\ldots,x_n) = \cone_\div{J}(\Ur) \amalg \comp{\cone}_\div{J}(\Ur).
\]
\end{proposition}

\subsubsection{A proof of completeness by induction}
Let $\Ur$ be a finite set of monomials in $\Mr(x_1,\ldots,x_n)$. We consider the partition $[0],\ldots ,[\deg_n(\Ur)]$ of monomials in $\Ur$ by their degrees in $x_n$. Let $\alpha_1<\alpha_2<\ldots <\alpha_k$ be the positive integers such that $[\alpha_i]$ is non-empty. Recall that $\overline{[\alpha_i]}$ is the set of monomials $u$ in $\Mr(x_1,\ldots,x_{n-1})$ such that  $ux_n^{\alpha_i}$ is in  $\Ur$. With these notations, the following result gives an inductive method to prove that a finite set of monomials is complete.

\begin{proposition}[{\cite[p. 19]{Janet29}}]
The finite set $\Ur$ is complete if and only if the two following conditions are satisfied:
\begin{enumerate}[{\bf i)}]
\item the sets $\overline{[\alpha_1]},\ldots,\overline{[\alpha_k]}$ are complete,
\item for any $1\leq i < k$, the set $\overline{[\alpha_i]}$ is contains in $\cone_\div{J}(\overline{[\alpha_i+1]})$.
\end{enumerate}
\end{proposition}

As an immediate consequence of this proposition, M. Janet obtained the following characterisation.

\begin{proposition}[{\cite[p. 20]{Janet29}}]
\label{Proposition:Completeness}
A finite set $\Ur$ of monomials of $\Mr(x_1,\ldots,x_n)$ is complete if and only if, for any $u$ in $\Ur$ and any $x$ non-multiplicative variable of $u$ with respect to $\Ur$, $ux$ is in $\cone_\div{J}(\Ur)$.
\end{proposition} 

\subsubsection{Example, {\cite[p. 21]{Janet29}}}
\label{Example:JanetExampleB}
Consider the subset $\Ur=\{\,x_5x_4,x_5x_3,x_5x_2,x_4^2,x_4x_3,x_3^2\}$ of $\Mr(x_1,\ldots,x_5)$. The multiplicative variables are given by the following table
\begin{center}
\begin{tabular}{c|ccccc}
$x_5x_4$ & $x_5$ & $x_4$ & $x_3$ & $x_2$ & $x_1$\\
$x_5x_3$ & $x_5$ & &$x_3$ & $x_2$ & $x_1$\\
$x_5x_2$ & $x_5$ & &&$x_2$ & $x_1$\\
$x_4^2$ & & $x_4$ & $x_3$ & $x_2$ & $x_1$\\
$x_3x_4$ & & & $x_3$ & $x_2$ & $x_1$\\
$x_3^2$ & & & $x_3$ & $x_2$ & $x_1$\\
\end{tabular}
\end{center}
In order to prove that this set of monomials is complete, we apply Proposition~\ref{Proposition:Completeness}.  The completeness follows from the identities:
\begin{center}
\begin{tabular}{c}
$x_5x_3.x_4=x_5x_4.x_3$, \\
$x_5x_2.x_4=x_5x_4.x_2$,\; $x_5x_2.x_3=x_5x_3.x_2$, \\
$x_4^2.x_5=x_5x_4.x_4$, \\
$x_4x_3.x_5=x_5x_4.x_3$,\; $x_4x_3.x_4 = x_4^2.x_3$, \\
$x_3^2.x_5 = x_5x_3.x_3$,\; $x_3^2.x_4 = x_4x_3.x_3$.
\end{tabular}
\end{center}

\subsubsection{Examples}
For every $1\leq p\leq n$, the set of monomials of degree $p$ is complete.
Any finite set of monomials of degree $1$ is complete.

\begin{theorem}[Janet's Completion Lemma, {\cite[p. 21]{Janet29}}]
\label{Theorem:CompletionJanet}
For any finite set $\Ur$ of monomials of~$\Mr(x_1,\ldots,x_n)$ there exists a finite set~ $J(\Ur)$ satisfying the following three conditions:
\begin{enumerate}[{\bf i)}]
\item $J(\Ur)$ is complete,
\item $\Ur \subseteq J(\Ur)$,
\item $\cone(\Ur) = \cone(J(\Ur))$.
\end{enumerate}
\end{theorem}

\subsubsection{Completion procedure}
\label{Subsubsection:CompletionProcedure}
From Proposition~\ref{Proposition:Completeness}, M. Janet deduced the completion procedure ${\bf Complete}(\Ur)$, Procedure~\ref{A:CompletionProcedure}, that computes a completion of finite set of monomials $\Ur$, {\cite[p. 21]{Janet29}}. M.~Janet did not give a proof of the termination of this procedure. We will present a proof of the correction and termination of this procedure in Section~\ref{InvolutiveCompletionAlgorithm}.

\begin{algorithm}
\SetAlgoLined
\KwIn{$\Ur$ a finite set of monomials in $\Mr(x_1,\ldots,x_n)$}

\BlankLine

\KwOut{A finite set $J(\Ur)$ satisfying the condition of Theorem~\ref{Theorem:CompletionJanet}.}

\BlankLine

\Begin{

$\widetilde{\Ur} \leftarrow \Ur$

\BlankLine

\While{exists $u\in\widetilde{\Ur}$ and $x\in{\nonmult_{\div{J}}^{\widetilde{\Ur}}(u)}$ such that $ux$ is not in $\cone_\div{J}{(\widetilde{\Ur})}$}{

\BlankLine

{\bf Choose} such $u$ and $x$,

$\widetilde{\Ur} \leftarrow \widetilde{\Ur} \cup \{ux\}$.

}
}
\caption{${\bf Complete}(\Ur)$}
\label{A:CompletionProcedure}
\end{algorithm}

\subsubsection{Example, {\cite[p. 28]{Janet29}}}
\label{Example:JanetFinalCompletion}
Consider the set $\Ur=\{\,x_3x_2^2,x_3^3x_1^2\,\}$ of monomials of $\Mr(x_1,x_2,x_3)$ with the order $x_3>x_2>x_1$. The following table gives the multiplicative variables for each monomial:
\begin{center}
\begin{tabular}{c|ccc}
$x_3^3x_1^2$ & $x_3$ & $x_2$ & $x_1$\\
$x_3x_2^2$ &  & $x_2$ & $x_1$\\
\end{tabular}
\end{center}
We complete the set $\Ur$ as follows. The monomial $x_3x_2^2.x_3$ is not in $\cone_\div{J}(\Ur)$, we set $\widetilde{\Ur}\leftarrow \Ur\cup\{x_3^2x_2^2\}$ and we compute multiplicative variables with respect to $\widetilde{\Ur}$:
\begin{center}
\begin{tabular}{c|ccc}
$x_3^3x_1^2$ & $x_3$ & $x_2$ & $x_1$\\
$x_3^2x_2^2$ &  & $x_2$ & $x_1$\\
$x_3x_2^2$ &  & $x_2$ & $x_1$\\
\end{tabular}
\end{center}
The monomial $x_3x_2^2.x_3$ is in $\cone_\div{J}(\widetilde{\Ur})$ but $x_3^2x_2^2.x_3$ is not in $\cone_\div{J}(\widetilde{\Ur})$, then we set 
$\widetilde{\Ur} \leftarrow \widetilde{\Ur} \cup \{x_3^3x_2^2\}$. The multiplicative variable of this new set of monomials is
\begin{center}
\begin{tabular}{c|ccc}
$x_3^2x_2^2$ & $x_3$ & $x_2$ & $x_1$\\
$x_3^3x_1^2$ & $x_3$ &  & $x_1$\\
$x_3^2x_2^2$ &  & $x_2$ & $x_1$\\
$x_3x_2^2$ &  & $x_2$ & $x_1$\\
\end{tabular}
\end{center}
The monomial $x_3x_1^2.x_2$ is not in $\cone_\div{J}(\widetilde{\Ur})$, the other products are in $\cone_\div{J}(\widetilde{\Ur})$, and we prove that the system
\[
\widetilde{\Ur} =
\{\,x_3x_2^2,x_3^3x_1^2,x_3^3x_2^2,x_3^3x_2x_1^2,x_3^2x_2^2\,\}
\]
is complete.

\subsection{Inverse of derivation}
\label{SS:InverseDerivation}

In this subsection, we recall the results of M. Janet from \cite{Janet29} on solvability of monomial PDE systems of the form
\begin{equation}
\label{Equation:MonomialPDESystem}
(\Sigma)
\qquad
D^\alpha \varphi
=
f_{\alpha}(x_1,x_2,\ldots, x_n) \qquad \alpha \in \Nb^n, 
\end{equation}
where $\varphi$ is an unknown function and the $f_{\alpha}$ are analytic functions of several variables. As recalled in \ref{Subsection:HistoricalContext}, an infinite set of partial differential equations can be always reduced to a finite set of such equations. This is a consequence of Dickson's Lemma whose formulation due to M. Janet is given in Lemma~\ref{Lemma:Janet1}. By this reason, we can assume that the system $(\Sigma)$ is finite without loss of generality. Using Proposition~\ref{Proposition:IsomorphismPartialX}, M. Janet associated to each differential operator $D^\alpha$ a monomial $x^\alpha$ in $\Mr(x_1,\ldots,x_n)$.
In this way, to a PDE system $(\Sigma)$ on variables $x_1,\ldots,x_n$ he associated a finite set $\lm(\Sigma)$ of monomials. By Theorem~\ref{Theorem:CompletionJanet}, any such a set $\lm(\Sigma)$ of monomials can be completed into a finite complete set $J(\lm(\Sigma))$ having the same cone as $\lm(\Sigma)$.

\subsubsection{Computation of inverse of derivation}
\label{Subsubsection:CalculInverseDerivation}
Let us now assume that the set of monomials $\lm(\Sigma)$ is finite and complete. The cone of $\lm(\Sigma)$ being equal to the involutive cone of $\lm(\Sigma)$, for any monomial~$u$ in $\lm(\Sigma)$ and non-multiplicative variable $x_i$ in $\nonmult_\div{J}^{\lm(\Sigma)}(u)$, there exists a decomposition
\[
ux_i = vw,
\]
where $v$ is in $\lm(\Sigma)$ and $w$ belongs to $\Mr(\mult_\div{J}^{\lm(\Sigma)}(v))$.
For any such a decomposition, it corresponds a compatibility condition of the PDE system $(\Sigma)$, that is,  for $u=x^{\alpha}$, $v=x^{\beta}$ and $w=x^{\gamma}$
with $\alpha, \beta$ and $\gamma$ in $\Nb^n$, 
\[
\frac{\partial f_{\alpha}}{\partial x_i} 
=
D^\gamma f_{\beta}.
\]
Let us denote by $(C_\Sigma)$ the set of all such compatibility conditions.
M. Janet showed that with the completeness hypothesis this set of compatibility conditions is sufficient for the PDE system $(\Sigma)$ to be formally integrable in the sense of~\cite{Pommaret78}. 

\subsubsection{The space of initial conditions}
Let us consider the set $\comp{\lm(\Sigma)}$ of complementary monomials of the finite complete set~$\lm(\Sigma)$. Suppose that the PDE system $(\Sigma)$ satisfies the set $(C_\Sigma)$ of compatibility conditions. M. Janet associated to each monomial $v=x^{\beta}$ in $\comp{\lm(\Sigma)}$ with $\beta \in \Nb^n$ an analytic function 
\[
\varphi_{\beta}(x_{i_1},\ldots,x_{i_{k_v}}),
\]
where $\{x_{i_1},\ldots,x_{i_{k_v}}\}=\cmult_\div{J}^{\comp{\lm(\Sigma)}}(v)$. 
By Proposition~\ref{Proposition:PartitionMr}, the set of such analytic functions provides a compatible initial condition. Under these assumptions, M. Janet proved the following result.

\begin{theorem}[{\cite[p. 25]{Janet29}}]
\label{Theorem:BoundaryConditions}
Let $(\Sigma)$ be a finite monomial PDE system such that $\lm(\Sigma)$ is complete. If~$(\Sigma)$ satisfies the compatibility conditions $(C_\Sigma)$, then it always admits a unique solution with initial conditions given for any $v=x^{\beta}$ in $\comp{\lm(\Sigma)}$ with $\beta \in \Nb^n$ by 
\[
\left. D^\beta \varphi  \right|_{x_j=0 \; \forall x_j\in \cnonmult_\div{J}^{\comp{\lm(\Sigma)}}(v)}= 
\varphi_{\beta}(x_{i_1},\ldots,x_{i_{k_v}}),
\]
where $\{x_{i_1},\ldots,x_{i_{k_v}}\} = \cmult_\div{J}^{\comp{\lm(\Sigma)}}(v)$.
\end{theorem}  
These  initial conditions are called \emph{initial conditions} by M. Janet. The method to obtain this initial conditions is illustrated by the two following examples.

\subsubsection{Example, {\cite[p. 26]{Janet29}}}
\label{Example:JanetExampleB2}
Consider the following monomial PDE system $(\Sigma)$ of unknown function $\varphi$ of variables $x_1,\ldots,x_5$:
\begin{align*}
\frac{\partial^2\varphi}{\partial x_5 \partial x_4}&= f_1(x_1,\ldots,x_5), &\;  \frac{\partial^2\varphi}{\partial x_5\partial x_3}&= f_2(x_1,\ldots,x_5), &\;
\frac{\partial^2\varphi}{\partial x_5\partial x_2}&= f_3(x_1,\ldots,x_5), \\ \frac{\partial^2\varphi}{\partial x_4^2}&= f_4(x_1,\ldots,x_5), &\;
\frac{\partial^2\varphi}{\partial x_4\partial x_3}&= f_5(x_1,\ldots,x_5),  &\; \frac{\partial^2\varphi}{\partial x_3^2}&= f_6(x_1,\ldots,x_5).\\
\end{align*}
The set $(C_\Sigma)$ of compatibility relations of the PDE system $(\Sigma)$ is a consequence of the identities used in Example~\ref{Example:JanetExampleB} to prove the completeness of the system:
\begin{center}
\begin{tabular}{c|c}
$x_5x_3.x_4=x_5x_4.x_3$, & $\frac{\partial f_2}{\partial x_2}=\frac{\partial f_1}{\partial x_3}$,\\
$x_5x_2.x_4=x_5x_4.x_2$,\; $x_5x_2.x_3=x_5x_3.x_2$, & $\frac{\partial f_3}{\partial x_4}=\frac{\partial f_1}{\partial x_2}$, $\frac{\partial f_3}{\partial x_3}=\frac{\partial f_2}{\partial x_2}$,\\
$x_4^2.x_5=x_5x_4.x_4$, &$\frac{\partial f_4}{\partial x_5}=\frac{\partial f_1}{\partial x_4}$,\\
$x_4x_3.x_5=x_5x_4.x_3$,\; $x_4x_3.x_4 = x_4^2.x_3$, & $\frac{\partial f_5}{\partial x_5}=\frac{\partial f_1}{\partial x_3}$,\; $\frac{\partial f_5}{\partial x_4}=\frac{\partial f_4}{\partial x_3}$, \\
$x_3^2.x_5 = x_5x_3.x_3$,\; $x_3^2.x_4 = x_4x_3.x_3$, &$\frac{\partial f_6}{\partial x_5}=\frac{\partial f_2}{\partial x_3}$,\; $\frac{\partial f_6}{\partial x_4}=\frac{\partial f_5}{\partial x_3}$.\\
\end{tabular}
\end{center}
The initial conditions are obtained using the multiplicative variables of the set $\comp{\lm(\Sigma)}$ of complementary monomials of $\lm(\Sigma)$.
We have
\[
\compp{5}{\lm(\Sigma)} = \compp{4}{\lm(\Sigma)} = \compp{1}{\lm(\Sigma)} = \emptyset,
\quad 
\compp{3}{\lm(\Sigma)} = \{ 1,x_3,x_4\},
\quad 
\compp{2}{\lm(\Sigma)} = \{ x_5\}.
\]
The multiplicative variables of these monomials are given by the following table
\begin{center}
\begin{tabular}{c|c}
$1$, $x_3$, $x_4$ & $x_1$, $x_2$,\\
$x_5$ & $x_1$, $x_5$.\\
\end{tabular}
\end{center}
By Theorem \ref{Theorem:BoundaryConditions}, the PDE system $(\Sigma)$ admits always a unique solution with any given initial conditions of the following type
\begin{align*}
\left.\frac{\partial \varphi}{\partial x_4}\right|_{x_3=x_4=x_5=0} &= \varphi_{0,0,0,1,0}(x_1,x_2)\\
\left.\frac{\partial \varphi}{\partial x_3}\right|_{x_3=x_4=x_5=0} &= \varphi_{0,0,1,0,0}(x_1,x_2)\\
\left. \varphi\right|_{x_3=x_4=x_5=0} &= \varphi_{0,0,0,0,0}(x_1,x_2)\\
\left.\frac{\partial \varphi}{\partial x_5}\right|_{x_2=x_3=x_4=0} &= \varphi_{0,0,0,0,1}(x_1,x_5).\\
\end{align*}

\subsubsection{Example}
In a last example, M. Janet considered a monomial PDE system where the partial derivatives of the left hand side do not form a complete set of monomials. It is the PDE system $(\Sigma)$ of unknown function $\varphi$ of variables~$x_1,x_2,x_3$ given by
\[
\frac{\partial^3\varphi}{\partial x_2^2 \partial x_3}= f_1(x_1,x_2,x_3),
\qquad
\frac{\partial^5\varphi}{\partial x_1^2 \partial x_3^3}= f_2(x_1,x_2,x_3).
\]
We consider the set of monomials $\lm(\Sigma)=\{x_3x_2^2,x_3^3x_1^2\}$.
In Example~\ref{Example:JanetFinalCompletion}, we complete $\lm(\Sigma)$ into the following complete set of monomials
\[
J(\lm(\Sigma)) =
\{\,x_3x_2^2,x_3^3x_1^2,x_3^3x_2^2,x_3^3x_2x_1^2,x_3^2x_2^2\,\}.
\]
The complementary set of monomials are
\[
\compp{3}{J(\lm(\Sigma))} = \{1\},
\quad 
\compp{2}{J(\lm(\Sigma))} = \{x_3^2x_2,x_3^2,x_3x_2,x_3\},
\quad 
\compp{1}{J(\lm(\Sigma))} = \{ x_3^3x_2x_1, x_3^3x_2, x_3^3x_1,x_3^3\}.
\]
The multiplicative variables of these monomials are given by the following table
\begin{center}
\begin{tabular}{c|c}
$\compp{3}{J(\lm(\Sigma))}$ & $x_1$, $x_2$,\\
$\compp{2}{J(\lm(\Sigma))}$ & $x_1$.\\
$\compp{1}{J(\lm(\Sigma))}$ & $x_3$.\\
\end{tabular}
\end{center}
By Theorem \ref{Theorem:BoundaryConditions}, the PDE system $(\Sigma)$ admits always a unique solution with any given initial conditions of the following type
\[
\left.\varphi\right|_{x_3=0} 
= \varphi_{0,0,0}(x_1,x_2)
,\quad
\left.\frac{\partial \varphi}{\partial x_3}\right|_{x_2=x_3=0} 
= \varphi_{0,0,1}(x_1)
,\quad
\left.\frac{\partial^2 \varphi}{\partial x_3\partial x_2}\right|_{x_2=x_3=0} 
= \varphi_{0,1,1}(x_1)
\]
\[
\left.\frac{\partial^2 \varphi}{\partial x_3^2}\right|_{x_2=x_3=0} 
= \varphi_{0,0,2}(x_1)
,\quad
\left.\frac{\partial^3 \varphi}{\partial x_3^2\partial x_2}\right|_{x_2=x_3=0} 
= \varphi_{0,1,2}(x_1)
,\quad
\left.\frac{\partial^3 \varphi}{\partial x_3^3}\right|_{x_1=x_2=0} 
= \varphi_{0,0,3}(x_3),
\]
\[
\left.\frac{\partial^4 \varphi}{\partial x_3^3\partial x_1}\right|_{x_1=x_2=0} 
= \varphi_{1,0,3}(x_3)
,\quad
\left.\frac{\partial^4 \varphi}{\partial x_3^3\partial x_2}\right|_{x_1=x_2=0} 
= \varphi_{0,1,3}(x_3)
,\quad
\left.\frac{\partial^5 \varphi}{\partial x_3^3\partial x_2\partial x_1}\right|_{x_1=x_2=0} 
= \varphi_{1,1,3}(x_3).
\]

\section{Monomial involutive bases}
\label{Section:InvolutiveDivision}

In this section, we recall a general approach of involutive monomial divisions introduced by V. P. Gerdt in \cite{Gerdt97}, see also \cite{GerdtBlinkovYuri98,GerdtBlinkovYuri98b}. In particular, we give the axiomatic properties of an involutive division. The partition of variables into multiplicative and non-multiplicative can be deduced from this axiomatic. In this way, we explain how the notion of multiplicative variable in the sense of Janet can be deduced from a particular involutive division.

\subsection{Involutive division}

\subsubsection{Involutive division}
\label{Definition:InvolutiveDivision}
An \emph{involutive division} $\div{I}$ on the set of monomials $\Mr(x_1,\ldots,x_n)$ is defined by a relation $|_\div{I}^\Ur$ in~$\Ur\times \Mr(x_1,\ldots,x_n)$, for every subset $\Ur$ of $\Mr(x_1,\ldots,x_n)$, satisfying, for all monomials $u$, $u'$ in $\Ur$ and $v$, $w$ in $\Mr(x_1,\ldots,x_n)$, the following six conditions 
\begin{enumerate}[{\bf i)}]
\item $u |_\div{I}^\Ur w$ implies $u|w$,
\item $u|_\div{I}^\Ur u$, for all $u$ in $\Ur$,
\item $u|_\div{I}^\Ur uv$ and $u |_\div{I}^\Ur uw$ if and only if $u |_\div{I}^\Ur uvw$,
\item if $u|_\div{I}^\Ur w$ and $u'|_\div{I}^\Ur w$, then $u|_\div{I}^\Ur u'$ or $u'|_\div{I}^\Ur u$,
\item if $u|_\div{I}^\Ur u'$ and $u'|_\div{I}^\Ur w$, then $u|_\div{I}^\Ur w$,
\item if $\Ur' \subseteq \Ur$ and $u\in \Ur'$, then $u |_\div{I}^\Ur w$ implies $u|_\div{I}^{\Ur'}w$.
\end{enumerate}
When no confusion is possible, the relation $|_\div{I}^\Ur$ will be also denoted by $|_\div{I}$.

\subsubsection{Multiplicative monomial}
If $u|_\div{I}^\Ur w$, by {\bf i)} there exists a monomial $v$ such that $w=uv$. We say that $u$ is an \emph{$\div{I}$-involutive divisor} of $w$, $w$ is an \emph{$\div{I}$-involutive multiple} of $u$ and $v$ is \emph{$\div{I}$-multiplicative} for $u$ with respect to $\Ur$.
When the monomial $uv$ is not an involutive multiple of $u$ with respect to $\Ur$, we say that $v$ is \emph{$\div{I}$-non-multiplicative} for $u$ with respect to $\Ur$.

We define in a same way the notion of multiplicative (resp. non-multiplicative) variable.
We denote by $\mult_\div{I}^{\Ur}(u)$ (resp. $\nonmult_\div{I}^\Ur(u)$) the set of multiplicative (resp. non-multiplicative) variables for the division $\div{I}$ of a monomial $u$ with respect to $\Ur$. We have 
\[
\mult_\div{I}^{\Ur}(u) = \{\;x\in\{x_1,\ldots,x_n\}\;\big|\; u |_\div{I}^\Ur ux \;\}
\]
and thus a partition of the set of variables $\{\,x_1,\ldots,x_n\,\}$ into sets of multiplicative and non-multiplicative variables.
An involutive division $\div{I}$ is thus entirely defined by a partition
\[
\{x_1,\ldots,x_n\} 
= \mult_\div{I}^{\Ur}(u) \sqcup \nonmult_\div{I}^\Ur(u),
\]
for any finite subset $\Ur$ of $\Mr(x_1,\ldots,x_n)$ and any $u$ in $\Ur$, satisfying conditions ${\bf iv)}$, ${\bf v)}$ and ${\bf vi)}$ of Definition~\ref{Definition:InvolutiveDivision}.
The involutive division $\div{I}$ is then defined by setting $u\mid_\div{I}^\Ur w$ if $w=uv$ and the monomial$v$ belongs to $\Mr(\mult_\div{I}^{\Ur}(u))$. Conditions ${\bf i)}$, ${\bf ii)}$ and ${\bf iii)}$ of Definition \ref{Definition:InvolutiveDivision} are consequence of this definition.

\subsubsection{Example}
\label{Example:simpliste}
Consider $\Ur=\{x_1,x_2\}$ in $\Mr(x_1,x_2)$ and suppose that $\div{I}$ is an involutive division such that 
$\mult_\div{I}^{\Ur}(x_1)=\{x_1\}$ and $\mult_\div{I}^\Ur(x_2)=\{x_2\}$. Then we have
\[
x_1\nmid_\div{I} x_1x_2,
\quad\text{and}\quad
x_2\nmid_\div{I} x_1x_2.
\]

\subsubsection{Autoreduction}
\label{Subsubsection:MonomialAutoreduction}
A subset $\Ur$ of $\Mr(x_1,\ldots,x_n)$ is said to be \emph{autoreduced} with respect to an involutive division $\div{I}$, or \emph{$\div{I}$-autoreduced}, if it does not contain a monomial $\div{I}$-divisible by another monomial of $\Ur$.

In particular, by definition of the involutive division, for any monomials $u,u'$ in $\Ur$ and monomial~$w$ in~$\Mr(x_1,\ldots,x_n)$, we have $u|_\div{I}w$ and $u'|_\div{I} w$ implies $u|_\div{I}u'$ or $u'|_\div{I}u$. As a consequence, if a set of monomials $\Ur$ is $\div{I}$-autoreduced, then any monomial in $\Mr(x_1,\ldots,x_n)$ admits at most one $\div{I}$-involutive divisor in $\Ur$.

\subsubsection{The Janet division}
We call \emph{Janet division} the division on $\Mr(x_1,\ldots,x_n)$ defined by the multiplicative variables in the sense of Janet defined in \ref{Subsection:JanetMultiplicativeVariables}. Explicitely, for a subset $\Ur$ of $\Mr(x_1,\ldots,x_n)$ and monomials $u$ in $\Ur$ and $w$ in $\Mr(x_1,\ldots,x_n)$, we define $u|_\div{J}^\Ur w$ if $u$ is a Janet divisor of $w$ as defined in \ref{Definition:JanetDivisor}, that is $w=uv$, where $v\in \Mr(\mult_\div{J}^\Ur(u))$ and $\mult_\div{J}^\Ur(u)$ is the set of Janet's multiplicative variables defined in~\ref{Subsection:JanetMultiplicativeVariables}.

By Proposition~\ref{Proposition:UnicityJanetDivisor}, for a fixed subset of monomial $\Ur$, any monomial of $\Mr(x_1,\ldots,x_n)$ has a unique Janet divisor in $\Ur$ with respect to $\Ur$. As a consequence, the conditions {\bf iv)} and {\bf v)} of Definition\ref{Definition:InvolutiveDivision} trivially hold for the Janet division. 
Now suppose that $\Ur' \subseteq \Ur$ and $u$ is a monomial in $\Ur'$. If $u |_\div{J}^\Ur w$ there is a decomposition $w=uv$ with $v\in \Mr(\mult_\div{J}^\Ur(u))$. As $\mult_\div{J}^{\Ur}(u) \subseteq \mult_\div{J}^{\Ur'}(u)$, this implies that $u|_\div{J}^{\Ur'}w$. Hence, the conditions {\bf vi)} of Definition \ref{Definition:InvolutiveDivision} holds for the Janet division. 
We have thus proved

\begin{proposition}[{\cite[Proposition 3.6]{GerdtBlinkovYuri98}}]
The Janet division is involutive.
\end{proposition}

\subsection{Involutive completion procedure}
\label{InvolutiveCompletionAlgorithm}

\subsubsection{Involutive set}
\label{SSS:InvolutiveSet}
Let $\div{I}$ be an involutive division on $\Mr(x_1,\ldots,x_n)$ and let $\Ur$ be a set of monomials.
The \emph{involutive cone} of a monomial $u$ in $\Ur$ with respect to the involutive division~$\div{I}$ is defined by
\[
\cone_\div{I}(u,\Ur) = \{\;uv \;\big|\; v\in \Mr(x_1,\ldots,x_n)\; \text{and} \; u |_\div{I}^\Ur uv \;\}.
\]
The \emph{involutive cone} of $\Ur$ with respect to the involutive division $\div{I}$ is the following subset of monomials:
\[
\cone_\div{I}(\Ur) = \bigcup_{u\in \Ur} \cone_\div{I}(u,\Ur).
\]
Note that the inclusion $\cone_\div{I}(\Ur) \subseteq \cone(\Ur)$ holds for any set $\Ur$. Note also that when the set $\Ur$ is $\div{I}$-autoreduced, by involutivity this union is disjoint.

A subset $\Ur$ of $\Mr(x_1,\ldots,x_n)$ is \emph{$\div{I}$-involutive} if the following equality holds
\[
\cone(\Ur)=\cone_\div{I}(\Ur).
\]
In other words, a set $\Ur$ is $\div{I}$-involutive if any multiple of an element $u$ in $\Ur$ is also $\div{I}$-involutive multiple of an element $v$ of $\Ur$. Note that the monomial $v$ can be different from the monomial $u$, as we have seen in Example~\ref{Example:JanetExampleB}.

\subsubsection{Involutive completion}
A \emph{completion} of a subset $\Ur$ of monomials of $\Mr(x_1,\ldots,x_n)$ with respect to an involutive division $\div{I}$, or \emph{$\div{I}$-completion} for short, is a set of monomials $\widetilde{\Ur}$ satisfying the following three conditions
\begin{enumerate}[{\bf i)}]
\item $\widetilde{\Ur}$ is involutive,
\item $\Ur \subseteq \widetilde{\Ur}$, 
\item $\cone(\widetilde{\Ur})=\cone(\Ur)$.
\end{enumerate}

\subsubsection{Noetherianity}
\label{SSS:Noetherianity}
An involutive division $\div{I}$ is said to be \emph{noetherian} if all finite subset $\Ur$ of~$\Mr(x_1,\ldots,x_n)$ admits a finite $\div{I}$-completion~$\widetilde{\Ur}$.

\begin{proposition}[{\cite[Proposition 4.5]{GerdtBlinkovYuri98}}]
The Janet division is noetherian.
\end{proposition}

\subsubsection{Prolongation}
Let $\Ur$ be a subset of $\Mr(x_1,\ldots,x_n)$. We call \emph{prolongation} of an element $u$ of $\Ur$ a  multiplication of $u$ by a variable $x$. 
Given an involutive division $\div{I}$, a prolongation $ux$ is \emph{multiplicative} (resp. \emph{non-multiplicative}) if $x$ is a multiplicative (resp. non-multiplicative) variable.

\subsubsection{Local involutivity}
A subset $\Ur$ of $\Mr(x_1,\ldots,x_n)$ is \emph{locally involutive with respect to an involutive division $\div{I}$} if any non-multiplicative prolongation of an element of $\Ur$ admit an involutive divisor in $\Ur$. That is
\[
\forall u \in \Ur
\quad 
\forall x_i \in \nonmult_\div{I}^\Ur(u)
\quad
\exists v \in \Ur
\quad\text{such that}\quad
v|_\div{I} ux_i.
\]

\subsubsection{Example,~{\cite[Example 4.8]{GerdtBlinkovYuri98}}}
\label{Exemple:Involutiflocalvsglobal}
By definition, if $\Ur$ is $\div{I}$-involutive, then it is locally $\div{I}$-involutive. The converse is false in general. Indeed, consider the involutive division $\div{I}$ on $\Mr=\Mr(x_1,x_2,x_3)$ defined by
\[
\mult_\div{I}^\Mr(x_1) = \{x_1,x_3\},\quad
\mult_\div{I}^\Mr(x_2) = \{x_1,x_2\},\quad
\mult_\div{I}^\Mr(x_3) = \{x_2,x_3\},
\]
with $\mult_\div{I}^\Mr(1) = \{x_1,x_2,x_3\}$ and $\mult_\div{I}^\Mr(u)$ is empty for $\deg(u)\geq 2$.
Then the set $\{x_1,x_2,x_3\}$ is locally $\div{I}$-involutive but not $\div{I}$-involutive.

\subsubsection{Continuity}
An involutive division $\div{I}$ is  \emph{continuous} if for all finite subset $\Ur$ of $\Mr(x_1,\ldots,x_n)$ and any finite sequence $(u_1,\ldots,u_k)$ of elements in $\Ur$ such that, there exists $x_{i_j}$ in $\nonmult_\div{I}^\Ur(u_j)$ such that
\[
u_{k} |_\div{I} u_{k-1}x_{i_{k-1}},
\;
\ldots
\;
, u_{3} |_\div{I} u_2x_{i_2},
\;
u_{2} |_\div{I} u_1x_{i_1},
\]
then $u_i\neq u_j$, for any $i\neq j$. 

For instance, the involutive division in Example~\ref{Exemple:Involutiflocalvsglobal} is not continuous. Indeed, there exists the following cycle of divisions:
\[
x_2|_\div{I} x_1x_2,
\quad
x_1|_\div{I} x_3x_1,
\quad
x_3|_\div{I} x_2x_3,
\quad
x_2|_\div{I} x_1x_2.
\]

\subsubsection{From local to global involutivity}
Any $\div{I}$-involutive subset $\Ur$ of $\Mr(x_1,\ldots,x_n)$ is locally $\div{I}$-involutive. When the division $\div{I}$ is continuous the converse is also true. Indeed, suppose that $\Ur$ is locally $\div{I}$-involutive. Let us show that $\Ur$ is $\div{I}$-involutive when the division $\div{I}$ is continuous.

Given a monomial $u$ in $\Ur$ and a monomial $w$ in $\Mr(x_1,\ldots,x_n)$, let us show that the monomial~$uw$ admits an $\div{I}$-involutive divisor in $\Ur$.
If $u|_\div{I} uw$ the claim is proved. Otherwise, there exists a non-multiplicative variable $x_{k_1}$ in $\nonmult_\div{I}^\Ur(u)$ such that $x_{k_1} | w$. By local involutivity, the monomial $ux_{k_1}$ admits an $\div{I}$-involutive divisor $v_1$ in $\Ur$.
If $v_1|_\div{I} uw$ the claim is proved. Otherwise, there exists a non-multiplicative variable $x_{k_2}$ in $\nonmult_\div{I}^\Ur(v_1)$ such that $x_{k_2}$ divides $\frac{uw}{v_1}$. By local involutivity, the monomial~$v_1x_{k_2}$ admits an $\div{I}$-involutive divisor $v_2$ in $\Ur$.

In this way, we construct a sequence $(u,v_1,v_2,\ldots)$ of monomials in $\Ur$ such that
\[
v_1|_\div{I} ux_{k_1},
\quad
v_2|_\div{I} v_1x_{k_2},
\quad
v_3|_\div{I} v_2x_{k_3},
\quad
\ldots
\]
By continuity hypothesis, all monomials $v_1,v_2,\ldots$ are distinct. Moreover, all these monomials are divisor of $uw$, that admits a finite set of distinct divisors. As a consequence, previous sequence is finite. It follows, that its last term $v_k$ is an $\div{I}$-involutive monomial of $uw$.
We have thus proved the following result.

\begin{theorem}[{\cite[Theorem 4.10]{GerdtBlinkovYuri98}}]
Let $\div{I}$ be a continuous involutive division. A subset of monomials of $\Mr(x_1,\ldots,x_n)$ is locally $\div{I}$-involutive if and only if it is $\div{I}$-involutive.
\end{theorem}

\begin{proposition}[{\cite[Corollary 4.11]{GerdtBlinkovYuri98}}]
The Janet division is continuous.
\end{proposition}

\subsubsection{Involutive completion procedure}
\label{Procedure:CompletionInvolutive}
Procedure~\ref{A:InvolutiveCompletionProcedure} generalizes Janet's completion procedure given in~\ref{Subsubsection:CompletionProcedure} to any involutive division.
Let us fix a monomial order $\preccurlyeq$ on $\Mr(x_1,\ldots,x_n)$. Given a set of monomials~$\Ur$, the procedure completes the set~$\Ur$ by all possible non-involutives prolongations of monomials in~$\Ur$. 

\begin{algorithm}
\SetAlgoLined
\KwIn{$\Ur$ a finite set of monomials of $\Mr(x_1,\ldots,x_n)$}

\BlankLine

\Begin{

$\widetilde{\Ur} \leftarrow \Ur$

\BlankLine

\While{exist $u\in\widetilde{\Ur}$ and $x\in\nonmult_\div{I}^{\widetilde{\Ur}}(u)$ such that $ux$ does not have $\div{I}$-involutive divisor in~$\widetilde{\Ur}$}{

\BlankLine

{\bf Choose} such a $u$ and $x$ corresponding to the smallest monomial $ux$ with respect to the monomial order $\preccurlyeq$

$\widetilde{\Ur} \leftarrow \widetilde{\Ur} \cup \{ux\}$

}

}

\BlankLine

\KwOut{$\widetilde{\Ur}$ the minimal involutive completion of the set~$\Ur$.}
\caption{Involutive completion procedure.}
\label{A:InvolutiveCompletionProcedure}
\end{algorithm}

By introducing the notion of constructive involutive division, V. P. Gerdt and Y. A. Blinkov gave in~\cite{GerdtBlinkovYuri98} some conditions on the involutive division $\div{I}$ in order to show the correction and the termination of this procedure. A continuous involutive division $\div{I}$ is \emph{constructive} if for any subset of monomials $\Ur$ of $\Mr(x_1,\ldots,x_n)$ and for any non-multiplicative prolongation $ux$ of a monomial $u$ in $\Ur$ satisfying the following two conditions
\begin{enumerate}[{\bf i)}]
\item $ux$ does not have an $\div{I}$-involutive divisor in $\Ur$,
\item any non-multiplicative prolongation $vy \neq ux$ of a monomial $v$ in $\Ur$ that divides $ux$ has an $\div{I}$-involutive divisor in $\Ur$,
\end{enumerate}
the monomial $ux$ cannot be $\div{I}$-involutively divided by a monomial $w$ in $\cone_\div{I}(\Ur)$ with respect to~$\Ur\cup\{w\}$.

If $\div{I}$ is a constructive division, then the completion procedure completes the set $\Ur$ into an involutive set. We refer the reader to {\cite[Theorem~4.14]{GerdtBlinkovYuri98}} for a proof of correctness and termination of the completion procedure under these hypothesis.

\subsubsection{Example}
An application of this procedure on the set of monomials $\Ur=\{\,x_3x_2^2,x_3^3x_1^2\,\}$ given by M. Janet in \cite{Janet29} is developed in~\ref{Example:JanetFinalCompletion}.

\subsection{Others involutive approaches}
\label{SS:OtherInvolutiveApproaches}

For analysis on differential systems several other notions of multiplicative variables were studied by J. M. Thomas 1937 and J.-F. Pommarret in 1978. Others examples of involutive divisions can be found in~\cite{GerdtBlinkovYuri98b}.

\subsubsection{Thomas division}
\label{SSS:ThomasDivision}
In~\cite{Thomas37}, Thomas introduced an involutive division that differs from those of M. Janet also used in the analysis on differential systems.
The multiplicative variables in the sense of Thomas's division for a monomial $u$ with of a finite subset $\Ur$ of $\Mr(x_1,\ldots,x_n)$ are defined as follows:
\[
x_i\in \mult_{\div{T}}^\Ur(u)
\quad\text{if}\quad
\deg_i(u)=\deg_i(\Ur).
\]
In particular, we have $u|_{\div{T}}^\Ur w$ if $w=uv$ and for all variable $x_i$ in $v$, we have $\deg_i(u)=\deg_i(\Ur)$. The Thomas division is a noetherian and continuous involutive division. We refer the reader to~\cite{GerdtBlinkovYuri98} for detailed proofs of this results. 
Note also that the Janet division is a refinement of Thomas division in the sense that for any finite set of monomials $\Ur$ and any monomial $u$ in $\Ur$, the following inclusions hold
\[
\mult_{\div{T}}^\Ur(u) \subseteq \mult_\div{J}^\Ur(u)
\quad\text{and}\quad
\nonmult_\div{J}^\Ur(u) \subseteq \nonmult_{\div{T}}^\Ur(u).
\]

\subsubsection{Pommaret division}
\label{SSSPommaretDivision}
In~\cite{Pommaret78}, Pommaret studied an involutive division that is defined globally, that is the multiplicative variables for the Pommaret division does not depend of a given subset of monomials. In this way, Pommaret's division can be defined on an infinite set of monomials. 

We fix an order on the variables $x_1>x_2>\ldots>x_n$. Given a monomial $u=x_1^{\alpha_1}\ldots x_k^{\alpha_k}$, with~$\alpha_k>0$, the Pommaret multiplicative variables for $u$ are defined by
\[
x_j\in \mult_{\div{P}}^{\Mr(x_1,\ldots,x_n)}(u), \quad \text{if $j\geq k$,}
\qquad\text{and}\qquad
x_j\in \nonmult_{\div{P}}^{\Mr(x_1,\ldots,x_n)}(u), \quad \text{if $j<k$.}
\]
We set $\mult_{\div{P}}^{\Mr(x_1,\ldots,x_n)}(1)=\{x_1,\ldots,x_n\}$.
The Pommaret division is a continuous involutive division that is not noetherian, \cite{GerdtBlinkovYuri98}. The Janet division is also a refinement of the Pommaret division, that is, for an autoreduced finite set of monomials $\Ur$, the following inclusions hold for any monomial $u$ in $\Ur$,
\[
\mult_{\div{P}}^\Ur(u) \subseteq \mult_\div{J}^\Ur(u)
\quad\text{and}\quad
\nonmult_\div{J}^\Ur(u) \subseteq \nonmult_{\div{P}}^\Ur(u).
\]

Finally, let us remark that the separation of variables into multiplicative and non-multiplicative ones in the Pommaret division was used first by Janet in~{\cite[\textsection 20]{Janet29}}. For this reason, the terminology \emph{Pommaret division} is not historically correct. We refer the reader to the monograph by W. M. Seiler~{\cite[Section 3.5]{Seiler10}} for an historical account.

\section{Polynomial partial differential equations systems}
\label{Section:PolynomialPartialDifferentialEquationsSystems}

In this section, we extend the results presented in Section~\ref{Subsection:JanetWork} on monomial systems to linear (polynomial) systems. All PDE systems are considered in analytic categories, namely all unknown functions, coefficients and initial conditions are supposed to be analytic. In a first part, we recall the notion of principal derivative with respect to an order on derivatives introduced by M. Janet. This notion is used to give an algebraic characterization of complete integrability conditions of a PDE system. Then we present a procedure that decides whether a given finite linear PDE system can be transformed into a completely integrable linear PDE system. Finally, we recall the algebraic formulation of involutivity introduced by M. Janet in \cite{Janet29}.

\subsection{Parametric and principal derivatives}
\label{Subsection:ParametricPrincipalDerivative}

\subsubsection{Motivations}
In {\cite[Chapter 2]{Janet29}}, M. Janet first considered the following PDE of one unknown function on $\mathbb{C}^n$:
\begin{equation}\label{EDP-2.1}
\frac{\partial^2 \varphi}{\partial x_n^2}=
\sum_{1\leq i,\, j<n}a_{i,j}(x)\frac{\partial^2 \varphi}{\partial x_i \partial x_j}+
\sum_{1\leq i<n}a_i(x) \frac{\partial^2 \varphi}{\partial x_i \partial x_n}+
\sum_{r=1}^n b_r(x) \frac{\partial \varphi}{\partial x_r}+c(x)\varphi+f(x),
\end{equation}
where the functions $a_{i,j}(x)$, $a_i(x)$, $b_r(x)$, $c(x)$ and $f(x)$ are analytic functions in a neighborhood of a point  $P=(x_1^0,\ldots, x_n^0)$ in $\mathbb{C}^n$.
Given two analytic functions $\varphi_1$ and $\varphi_2$ in a neighborhood $U_Q$ of a point $Q=(x_1^0,\ldots, x_{n-1}^0)$ in $\mathbb{C}^{n-1}$, M. Janet studied the existence of solutions of equation~\eqref{EDP-2.1} under the following initial condition:
\begin{equation}
\label{CI:EDP-2.1}
  \varphi\vert_{x_n=x_n^0}=\varphi_1, \qquad \left.\frac{\partial \varphi}{\partial x_n} \right\vert_{x_n=x_n^0}=\varphi_2,
\end{equation}
holds in a neighborhood of the point $Q$. In \ref{Subsubsection:BoundaryConditions}, we will formulate such condition for higher-order linear PDE systems with several unknown functions, called initial condition.

\subsubsection{Principal and parametric derivatives}
In order to analyse the existence and the uniqueness of a solution of equation~\eqref{EDP-2.1} under the initial condition~(\ref{CI:EDP-2.1}), M. Janet introduced the notions of parametric and principal derivative defined as follows.
The partial derivatives $D^\alpha\varphi$, with $\alpha=(\alpha_1,\ldots,\alpha_n)$,
of an analytic function $\varphi$ are determined by
\begin{enumerate}[{\bf i)}]
\item $\varphi_1$ and its derivatives for $\alpha_n=0$, 
\item $\varphi_2$ and its derivatives for $\alpha_n=1$,
\end{enumerate}
in the neighborhood $U_Q$. These derivatives for $\alpha_n=0$ and $\alpha_n=1$ are called \emph{parametric}, those derivatives for $\alpha_n\geq 2$, i.e. the derivative of $\frac{\partial^2 \varphi}{\partial x_n^2}$, are called \emph{principal}. 
Note that the values of the principal derivative at the point $P$ are entirely given by $\varphi_1$ and $\varphi_2$ and by their derivatives thanks to equation~\eqref{EDP-2.1}.} Note that the notion of parametric derivative corresponds to a parametrization of initial conditions of the system.

\subsubsection{Janet's orders on derivatives}
\label{SubsubsectionDegreeLexicographicOrder}

Let $\alpha=(\alpha_1,\ldots,\alpha_n)$ and $\beta=(\beta_1,\ldots,\beta_n)$ be in $\Nb^n$. Let $\varphi$ be an analytic function.
The derivative $D^\alpha \varphi$ is said to be \emph{posterior} (resp. \emph{anterior}) to $D^\beta \varphi$ if 
\begin{center}
$\vert\alpha\vert>\vert\beta\vert$ \quad (resp. $\vert\alpha\vert<\vert\beta\vert$) \qquad or \qquad $\vert\alpha\vert=\vert\beta\vert$ and $\alpha_n>\beta_n$ \quad (resp. $\alpha_n<\beta_n$).
\end{center}
Obviously, any derivative of $\varphi$ admits only finitely many anterior derivatives of $\varphi$. Using this notion of posteriority, M. Janet showed the existence and unicity problem of equation~\eqref{EDP-2.1} under the initial condition~\eqref{CI:EDP-2.1}.

In his monograph, M. Janet gave several generalizations of the previous posteriority notion. The first one corresponds to the degree lexicographic order, {\cite[\textsection 22]{Janet29}}, formulated as follows:
\begin{enumerate}[{\bf i)}]
\item for $\vert\alpha\vert\neq \vert\beta\vert$, the derivative $D^\alpha\varphi$ is called \emph{posterior} (resp. \emph{anterior}) to~$D^\beta\varphi$, if $\vert\alpha\vert>\vert\beta\vert$ (resp. $\vert\alpha\vert<\vert\beta\vert$),
\item for $\vert\alpha\vert=\vert\beta\vert$, the derivative $D^\alpha\varphi$ is called \emph{posterior} (resp. \emph{anterior}) to $D^\beta\varphi$ if the first non-zero difference 
\[ 
\alpha_n-\beta_n\; , \quad \alpha_{n-1}-\beta_{n-1}\; , \quad \ldots\quad, \; \alpha_1-\beta_1,
\]
is positive (resp. negative).
\end{enumerate}

\subsubsection{Generalization}
\label{Subsubsection:GeneralizationPrincipalDerivative}

Let us consider the following generalization of equation \eqref{EDP-2.1}:
\begin{equation}
\label{EDP-2.2}
D\varphi=\sum_{i \in I}a_{i}D_i\varphi+f, 
\end{equation}
where $D$ and the $D_i$ are differential operators such that $D_i\varphi$ is anterior to $D\varphi$ for all $i$ in $I$. 
The derivative~$D\varphi$ and all its derivatives are called \emph{principal derivatives of the equation (\ref{EDP-2.2})}. All the other derivative of~$u$ are called \emph{parametric derivatives of the equation (\ref{EDP-2.2})}.

\subsubsection{Weight order}
\label{sect-gen}
Further generalization of these order relations were given by M. Janet by introducing the notion of \emph{cote}, that corresponds to a parametrization of a weight order defined as follows.
Let us fix a positive integer $s$. We define a matrix of \emph{weight}
\[
C = \left[
\begin{tabular}{cccccc} 
$C_{1,1}$  & $\ldots$ &  $C_{n,1}$ \\
$\vdots$ & &$\vdots$  \\
$C_{1,s}$ &  $\ldots$ & $C_{n,s}$ \\
\end{tabular}
\right]
\]
that associates to each variable~$x_i$ non-negative integers $C_{i,1}, \ldots, C_{i,s}$, called the \emph{$s$-weights} of~$x_i$. This notion was called \emph{cote} by M. Janet in {\cite[\textsection 22]{Janet29}} following the terminology introduced by Riquier,~\cite{Riquier10}. 
For each derivative $D^\alpha\varphi$, with $\alpha=(\alpha_1,\ldots,\alpha_n)$ of an analytic function $\varphi$, we associate a \emph{$s$-weight}~$\Gamma(C)=(\Gamma_1,\ldots,\Gamma_s)$ where the $\Gamma_k$ are defined by 
\[ 
\Gamma_k =\sum_{i=1}^n \alpha_i C_{i,k}.
\]
Given two monomial partial differential operators $D^\alpha$ and $D^\beta$ as in \ref{SubsubsectionDegreeLexicographicOrder}, we say that $D^\alpha\varphi$ is \emph{posterior} (resp. \emph{anterior}) to $D^\beta\varphi$ with respect to a weigh matrix $C$ if 
\begin{enumerate}[{\bf i)}]
\item $\vert\alpha\vert\neq \vert\beta\vert$ and $\vert\alpha\vert>\vert\beta\vert$ (resp. $\vert\alpha\vert<\vert\beta\vert$),
\item otherwise $\vert\alpha\vert=\vert\beta\vert$ and the first non-zero difference  
\[ 
\Gamma_1-\Gamma_1',  \quad \Gamma_2-\Gamma_2'\; , \quad \ldots\quad, \;  \Gamma_s-\Gamma_s',
\]
is positive (resp. negative). 
\end{enumerate}
In this way, we define an order on the set of monomial partial derivatives, called \emph{weight order}.
Note that, we recover the Janet order defined in \ref{SubsubsectionDegreeLexicographicOrder}  by setting $C_{i,k}=\delta_{i+k,n+1}$.

\subsection{First order PDE systems}

We consider first resolution of first order PDE systems.

\subsubsection{Complete integrability}
In {\cite[\textsection 36]{Janet29}}, M. Janet considered the following first order PDE system 
\begin{equation}
\label{EDP-3.1}
(\Sigma) \qquad  \, \frac{\partial \varphi}{\partial y_\lambda}=f_\lambda(y_1,\cdots, y_h, z_1, \cdots, z_k, \varphi, q_1,\cdots, q_k) \qquad (1\leq \lambda\leq h)
\end{equation}
where $\varphi$ is an unknown function of independent variables $y_1,\ldots, y_h, z_1,\ldots, z_k$,
with $h+k=n$ and~$q_i=\frac{\partial \varphi}{\partial z_i}$. Moreover, we suppose that the functions  $f_\lambda$ are analytic in a neighborhood of a point $P$.
M. Janet wrote down explicitly the integrability condition of the PDE systems $(\Sigma)$ defined in (\ref{EDP-3.1}) namely by the following equality 
\[ 
\frac{\partial}{\partial y_\lambda}\left(\frac{\partial \varphi}{\partial y_\mu}\right) 
=
\frac{\partial}{\partial y_\mu}\left(\frac{\partial \varphi}{\partial y_\lambda}\right),
\]
for any $1\leq \lambda, \mu \leq h$. Following \eqref{EDP-3.1}, we deduce that
\begin{align*}
\frac{\partial}{\partial y_\lambda}\left(\frac{\partial \varphi}{\partial y_\mu}\right)=
&\frac{\partial f_\mu}{\partial y_\lambda}+\frac{\partial \varphi}{\partial y_\lambda}\frac{\partial f_\mu}{\partial \varphi}+\sum_{i=1}^k \frac{\partial f_\mu}{\partial q_i}\frac{\partial^2 \varphi}{\partial y_\lambda \partial z_i}, \\
=
&\frac{\partial f_\mu}{\partial y_\lambda}+f_\lambda\frac{\partial f_\mu}{\partial \varphi}
+\sum_{i=1}^k \frac{\partial f_\mu}{\partial q_i}
\left( \frac{\partial f_\lambda}{\partial z_i}+
q_i\frac{\partial f_\lambda}{\partial \varphi}\right)+
\sum_{i,j=1}^k\frac{\partial f_\lambda}{\partial q_i}\frac{\partial f_\mu}{\partial q_j}
\frac{\partial^2 \varphi}{\partial z_i \partial z_j}. \\
\end{align*}
Hence, the integrability condition is expressed as
\begin{equation}\label{eq-int}
\begin{split}
&\frac{\partial}{\partial y_\lambda}\left(\frac{\partial \varphi}{\partial y_\mu}\right)
-\frac{\partial}{\partial y_\mu}\left(\frac{\partial \varphi}{\partial y_\lambda}\right)\\
=
&\frac{\partial f_\mu}{\partial y_\lambda}+f_\lambda\frac{\partial f_\mu}{\partial \varphi}
+\sum_{i=1}^k \frac{\partial f_\mu}{\partial q_i}
\left( \frac{\partial f_\lambda}{\partial z_i}+
q_i\frac{\partial f_\lambda}{\partial \varphi}\right) -
\frac{\partial f_\lambda}{\partial y_\mu}-f_\mu\frac{\partial f_\lambda}{\partial \varphi}
-\sum_{i=1}^k \frac{\partial f_\lambda}{\partial q_i}
\left( \frac{\partial f_\mu}{\partial z_i}+
q_i\frac{\partial f_\mu}{\partial \varphi}\right) \\
=
&0,
\end{split}
\end{equation}
for any $1\leq \lambda\neq \mu\leq h$. When the PDE system $(\Sigma)$ defined in (\ref{EDP-3.1}) satisfies relation \eqref{eq-int}, the system  $(\Sigma)$ is said to be \emph{completely integrable}.

\begin{theorem}
Suppose that the PDE system $(\Sigma)$ in (\ref{EDP-3.1}) is completely integrable. Let $P$ be a point in~$\mathbb{C}^n$ and $\varphi(z_1,\ldots, z_k)$ be an analytic function in the neighborhood of the point $\pi(P)$, where $\pi:\mathbb{C}^n \rightarrow \mathbb{C}^k$ denotes the canonical projection $(y_1,\ldots, y_h, z_1, \ldots z_k) \, \mapsto \, (z_1,\ldots, z_k)$.
Then, the system $(\Sigma)$ admits only one analytic solution satisfying $u=\varphi \circ \pi$ in a neighborhood of the point~$P$.
\end{theorem}

\subsection{Higher-order finite linear PDE systems}
In {\cite[\textsection 39]{Janet29}}, M. Janet discussed the existence of solutions of a finite linear PDE system of one unknown function~$\varphi$ in which each equation is of the following form:
\begin{equation}
\label{Equation:HigherOrderPDE}
(\Sigma)\qquad
D_i\varphi=\sum_{j}a_{i,j}D_{i,j}\varphi, \quad i \in I.
\end{equation}
All the functions $a_{i,j}$ are supposed analytic in a neighborhood of a point $P$  in $\mathbb{C}^n$.

\subsubsection{Principal and parametric derivatives}
\label{Subsection:ParametricPrincipalDerivativesSystems}
Consider Janet's order $\jo$ on derivatives as the generalization defined in~\ref{SubsubsectionDegreeLexicographicOrder}. We suppose that each equation of the system $(\Sigma)$ defined by~(\ref{Equation:HigherOrderPDE}) satisfies the following two conditions:
\begin{enumerate}[{\bf i)}]
\item $D_{i,j}\varphi$ is anterior to $D_i\varphi$, for any $i$ in $I$,
\item all the $D_i$'s for $i$ in $I$ are distinct.
\end{enumerate}

We extend the notion of principal derivative introduced in \ref{Subsubsection:GeneralizationPrincipalDerivative} for one PDE equation to a system of the form (\ref{Equation:HigherOrderPDE}) as follows. The derivative $D_i\varphi$, for $i$ in $I$, and all its derivatives are called \emph{principal derivatives} of the PDE system $(\Sigma)$ given in (\ref{Equation:HigherOrderPDE}) with respect to Janet's order. Any other derivative of~$\varphi$ is called \emph{parametric derivative}. 

\subsubsection{Completeness with respect to Janet's order}
\label{CompletenessWithJanetOrdering}
Let us fix an order $x_n>x_{n-1}>\ldots >x_1$ on variables.
By the isomorphism of Proposition~\ref{Proposition:IsomorphismPartialX}, that identifies monomial partial differential operators with monomials in $\Mr(x_1,\ldots,x_n)$, we associate to the set of operators $D_i$'s, $i$ in $I$, defined in~\ref{Subsection:ParametricPrincipalDerivativesSystems}, a set $\lm_{\jo}(\Sigma)$ of monomials. By definition, the set $\lm_{\jo}(\Sigma)$ contains the monomials associated to leading derivatives of the PDE system $(\Sigma)$ with respect to Janet's order. 

The PDE system $(\Sigma)$ is said to be \emph{complete} with respect to Janet's order $\jo$ if the set of monomials~$\lm_{\jo}(\Sigma)$ is complete in the sense of \ref{Subsection:CompleteSystem}.
Procedure \ref{Procedure:CompletePDESystem} consists in a completion procedure that transforms a finite linear PDE system into an equivalent complete linear PDE system.

By definition the set of principal derivatives corresponds, by isomorphism of Proposition~\ref{Proposition:IsomorphismPartialX}, to the multiplicative cone of the monomial set $\lm_{\jo}(\Sigma)$. Hence, when $(\Sigma)$ is complete, the set of principal derivatives corresponds to the involutive cone of $\lm_{\jo}(\Sigma)$. By Proposition~\ref{Proposition:PartitionMr}, there is a partition 
\[
\Mr(x_1,\ldots,x_n) = \cone_\div{J}(\lm_{\jo}(\Sigma)) \amalg \comp{cone}_\div{J}(\lm_{\jo}(\Sigma)).
\]
It follows that set of parametric derivatives of a complete PDE system $(\Sigma)$ corresponds to the involutive cone of the set of monomials $\comp{\lm_{\jo}(\Sigma)}$.

\subsubsection{Initial conditions}
\label{Subsubsection:BoundaryConditionOneUnknown}
Consider the set~$\comp{\lm_{\jo}(\Sigma)}$ of complementary monomials of $\lm_{\jo}(\Sigma)$, as defined in~\ref{Subsubsection:ComplementaryMonomials}. 
To a monomial $x^{\beta}$ in~$\comp{\lm_{\jo}(\Sigma)}$, with $\beta=(\beta_1,\ldots,\beta_n)$ in $\Nb^n$ and 
\[
\cmult_{\div{J}}^{\comp{\lm_{\jo}(\Sigma)}}(x^{\beta}) = \{x_{i_1},\ldots,x_{i_{k_\beta}}\}, 
\]
we associate an arbitrary analytic function 
\[
\varphi_{\beta}(x_{i_1},\ldots,x_{i_{k_\beta}}).
\]
Using these functions, M. Janet defined an \emph{initial condition}:
\[
(C_\beta) \qquad \left. D^\beta \varphi \right|_{x_j=0 \; \forall x_j\in \cnonmult_{\div{J}}^{\comp{\lm_{\jo}(\Sigma)}}(x^{\beta})}= 
\varphi_{\beta}(x_{i_1},\ldots,x_{i_{k_\beta}}).
\]
Then he formulated an \emph{initial condition of the equation~(\ref{Equation:HigherOrderPDE})} with respect to Janet's order as the following set
\begin{equation}
\label{Equation:BoundaryCondition}
\{\, C_\beta \;|\; x^\beta \in \comp{\lm_{\jo}(\Sigma)}\,\}.
\end{equation}

\begin{theorem}[{\cite[\textsection 39]{Janet29}}]
\label{thm_exist-EDP1}
If the PDE system $(\Sigma)$ in~(\ref{Equation:HigherOrderPDE}) is complete with respect to Janet's order~$\jo$, then it admits at most one analytic solution satisfying the initial condition (\ref{Equation:BoundaryCondition}).
\end{theorem}

\subsubsection{PDE systems with several unknown functions}
\label{SubsubsectionPDEsystemsWithSeveralUnknownFunctions}
The construction of initial conditions given in~\ref{Subsubsection:BoundaryConditionOneUnknown} for one unknown function can be extended to linear PDE systems on $\mathbb{C}^n$ with several unknown functions using a weight order. Let us consider a linear PDE system of $m$ unknown analytic functions~$\varphi^1,\ldots, \varphi^m$ of the following form
\begin{equation}
\label{Equation:CanonicalSystem0}
(\Sigma) \qquad D^\alpha \varphi^r=\sum_{\substack{(\beta,s) \in \mathbb{N}^n\times\{1,2,\ldots,m\}}}
a_{\alpha,\beta}^{r,s}D^\beta \varphi^s, \qquad \alpha \in I^r,
\end{equation}
for $1\leq r\leq m$, where $I^r$ is a finite subset of $\mathbb{N}^n$ and the $a_{\alpha,\beta}^{r,s}$ are analytic functions.

For such a system, we define a weight order as follows. Let us fix a positive integer $s$. To any variable~$x_i$ we associate $s+1$ weights $C_{i,0},C_{i,1},\ldots, C_{i,s}$ by setting $C_{i,0}=1$ and the $C_{i,1},\ldots, C_{i,s}$ as defined in \ref{sect-gen}. For each unknown function $\varphi^{j}$, we associate~$s+1$~weights $T_0^{(j)},T_1^{(j)}\ldots, T_s^{(j)}$. With these data, we define the $s+1$ weights $\Gamma_0^{(j)},\Gamma_1^{(j)}, \ldots, \Gamma_s^{(j)}$  of the partial derivative
$D^\alpha \varphi^{j}$ with~$\alpha=(\alpha_1,\ldots, \alpha_n)$ in~$\Nb^n$ by setting
\[
\Gamma_k^{(j)}=\sum_{i=1}^n \alpha_i C_{i,k}+T_k^{(j)} \qquad (0\leq k\leq s). 
\]
We define the notions of anteriority and posteriority on derivatives with respect to this weight order, denoted by $\wo$, as it is done in~\ref{Subsection:ParametricPrincipalDerivativesSystems}  for systems of one unknown function. In particular, we define the notions of principal and parametric derivatives in a similar way to systems of one unknown function. 

Now suppose that the system~(\ref{Equation:CanonicalSystem0}) is written in the form 
\begin{equation}
\label{Equation:CanonicalSystem1}
(\Sigma) \qquad D^\alpha \varphi^r=\sum_{\substack{(\beta,s) \in \mathbb{N}^n\times\{1,2,\ldots,m\} \\ D^\beta \varphi^s \wostrict D^\alpha \varphi^r}}
a_{\alpha,\beta}^{r,s}D^\beta \varphi^s, \qquad \alpha \in I^r.
\end{equation}
We can formulate the notion of completeness with respect to the weight order $\wo$ as in \ref{CompletenessWithJanetOrdering}. Let consider $\lm_{\wo}(\Sigma,\varphi^r)$ be the set of monomials associated to leading derivatives $D^\alpha$ of all PDE in $(\Sigma)$ such that $\alpha$ belongs to $I^r$. The PDE system $(\Sigma)$ is \emph{complete} with respect to $\wo$, if for any $1\leq r \leq m$, the set of monomials $\lm_{\wo}(\Sigma,\varphi^r)$ is complete in the sense of \ref{Subsection:CompleteSystem}.
Finally, we can formulate as in (\ref{Equation:BoundaryCondition}) an initial condition for the linear PDE system~(\ref{Equation:CanonicalSystem1}) with respect to such a weight order:
\begin{equation}
\label{Equation:BoundaryConditionWO}
\{\, C_{\beta,r} \;\;|\;\; x^\beta \in \comp{\lm_{\wo}(\Sigma,\varphi^r)},\;\;\text{for $1\leq r \leq m$}\,\}.
\end{equation}

\begin{theorem}[{\cite[\textsection 40]{Janet29}}]
\label{thm_exist-EDP2}
If the PDE system $(\Sigma)$ in~(\ref{Equation:CanonicalSystem1}) is complete with respect to a weight order~$\wo$, then it admits at most one analytic solution satisfying the initial condition (\ref{Equation:BoundaryConditionWO}).
\end{theorem}

M. Janet said that this result could be proved in a way similar to the proof of Theorem~\ref{thm_exist-EDP1}.

\subsection{Completely integrable higher-order linear PDE systems}

In this subsection we will introduce integrability conditions for higher-order linear PDE systems  of several unknown functions. The main result, Theorem~\ref{Theorem:CaracterizationCompleteIntegrability}, algebraically characterizes the complete integrability property for complete PDE systems. It states that, under the completeness property, the complete integrability condition is equivalent to have all integrability conditions trivial. In this subsection, we will assume that the linear PDE systems are complete.
In Subsection~\ref{Subsection:ReductionPDEsystemToCanonicalForm} we will provide Procedure~\ref{Procedure:CompletePDESystem} that transforms a linear PDE system of the form (\ref{Equation:CanonicalSystem1}) into a complete linear PDE system with respect to a weight order. 

\subsubsection{Formal solutions}
Let consider a linear PDE system $(\Sigma)$ of the form~(\ref{Equation:CanonicalSystem1}) of unknown functions~$\varphi^1,\ldots,\varphi^m$ and independent variables $x_1,\ldots,x_n$. We suppose that $(\Sigma)$ is complete, hence the set of monomials  $\lm_{\wo}(\Sigma,\varphi^r)=\{x^\alpha \;|\; \alpha \in I^r\}$ is complete for all $1\leq r\leq m$. For the remaining part of this subsection, we will denote $\lm_{\wo}(\Sigma,\varphi^r)$ by $\Ur_r$.
Let denote by $(\cone_{\div{J},\wo}(\Sigma))$ the following PDE system, for $1\leq r\leq m$,
\[ 
\Phi(u)(D^{\alpha} \varphi^r)=
\sum_{\substack{(\beta,s) \in \mathbb{N}^n\times\{1,2,\ldots,m\} \\ D^\beta \varphi^s \wostrict D^\alpha \varphi^r}}
\Phi(u)\left(a_{\alpha,\beta}^{r,s}D^\beta \varphi^s\right), 
\]
for $\alpha \in I^r$ and $u \in \mathcal{M}(\mathrm{Mult}(x^\alpha,\Ur_r))$. 

We use the PDE system $(\cone_{\div{J},\wo}(\Sigma))$ to compute the values of the principal derivative at a point~$P^0=(x_1^0,\ldots, x_n^0)$ of~$\mathbb{C}^n$. 
We call \emph{formal solutions} of the PDE system $(\Sigma)$ at the point $P^0$ the elements $\varphi^1,\ldots,\varphi^m$ in $\mathbb{C}[[x_1-x_1^0, \ldots,x_n -x_n^0]]$ which are solutions of $(\Sigma)$. If the system $(\Sigma)$ admits an analytic solution then these formal solutions are convergent series and give analytic solutions of $(\Sigma)$ on a neighbourhood of the point $P^0$.

\subsubsection{Initial conditions}
\label{Subsubsection:BoundaryConditions}
The question is to determine under which condition the system $(\Sigma)$ admits a solution for any given initial condition. These initial conditions are parametrized by the set~$\comp{\Ur_r}$ of complementary monomials of the set of monomials $\Ur_r$ as in \ref{Subsubsection:BoundaryConditionOneUnknown}. Explicitly, for $1\leq r \leq m$, to a monomial $x^{\beta}$ in $\comp{\Ur_r}$, with $\beta$ in $\Nb^n$ and $\cmult_{\div{J}}^{\comp{\Ur_r}}(x^{\beta}) = \{x_{i_1},\ldots,x_{i_{k_r}}\}$, we associate an arbitrary analytic function 
\[
\varphi_{\beta,r}(x_{i_1},\ldots,x_{i_{k_r}}).
\]
Formulating \emph{initial condition} as the following data:
\[
(C_{\beta,r}) \qquad \left. D^{\beta} \varphi^r \right|_{x_j=x_j^0 \; \forall x_j\in \cnonmult_{\div{J}}^{\comp{\Ur_r}}(x^{\beta_r})}= 
\varphi_{\beta,r}(x_{i_1},\ldots,x_{i_{k_r}}).
\]
We set the \emph{initial condition} of the system $(\Sigma)$ in~(\ref{Equation:CanonicalSystem0}) to be the following set
\begin{equation}
\label{Equation:BoundaryCondition2}
\underset{1\leq r \leq m}{\bigcup}\{\, C_{\beta,r} \;|\; x^{\beta_r} \in \comp{\Ur_r}\,\}.
\end{equation}
Note that M. Janet call \emph{degree of generality} of the solution of the PDE system $(\Sigma)$ the dimension of the initial conditions of the system, that is
\[
\underset{\;\;\;u \in \comp{\Ur_r}}{\mathrm{Max}} \big|\cmult_{\div{J}}^{\comp{\Ur_r}}(u)\big|.
\]

\subsubsection{$\div{J}$-normal form}
Suppose that the PDE system $(\Sigma)$ is complete. Given a linear equation $E$ amongst the unknown functions $\varphi^1,\ldots,\varphi^m$ and variables $x_1,\ldots,x_n$. A \emph{$\div{J}$-normal form of $E$ with respect to the system $(\Sigma)$} is an equation obtained from $E$ by the reduction process that replaces principal derivatives by parametric derivatives with a similar procedure to {\bf RightReduce} given in Procedure~\ref{A:RightReduce}. 

\subsubsection{Integrability conditions}
\label{Subsubsection:IntegrabilityConditions}
Given $1\leq r\leq m$ and  $\alpha \in I^r$, let $x_i$ be in $\nonmult_{\div{J}}^{\Ur_r}(x^\alpha)$ a non-multiplicative variable. Let us differentiate the equation
\[ D^\alpha \varphi^r=\sum_{\substack{(\beta,s) \in \mathbb{N}^n\times\{1,2,\ldots,m\} \\ D^\beta \varphi^s \wostrict  D^\alpha \varphi^r}}
a_{\alpha,\beta}^{r,s}D^\beta \varphi^s
\]
by the partial derivative $\Phi(x_i)=\frac{\partial}{\partial x_i}$. We obtain the following PDE
\begin{equation}
\label{Equation:CanonicalSystem}
\Phi(x_i)(D^\alpha \varphi^r)=
\sum_{\substack{(\beta,s) \in \mathbb{N}^n\times\{1,2,\ldots,m\} \\ D^\beta \varphi^s \wostrict D^\alpha \varphi^r}} \left(\frac{\partial a_{\alpha,\beta}^{r,s}}{\partial x_i}D^\beta \varphi^s+a_{\alpha,\beta}^{r,s}\Phi(x_i)(D^\beta \varphi^s)\right). 
\end{equation}
Using system $(\cone_{\div{J},\wo}(\Sigma))$, we can rewrite the PDE~(\ref{Equation:CanonicalSystem}) into an PDE formulated in terms of parametric derivatives and independent variables. The set of monomials $\Ur_r$ being complete, there exists $\alpha'$  in $\Nb^n$ with $x^{\alpha'}$ in $\Ur_r$ and $u$  in~$\mathcal{M}(\mult_{\div{J}}^{\Ur_r}(x^{\alpha'}))$ such that $x_ix^{\alpha}=ux^{\alpha'}$. Then $\Phi(x_i)D^\alpha=\Phi(u)D^{\alpha'}$ as a consequence, we obtain the following equation
\begin{equation}
\label{Equation:CanonicalSystem2}
\sum_{\substack{(\beta,s) \in \mathbb{N}^n\times\{1,2,\ldots,m\} \\ D^\beta \varphi^s \wostrict D^\alpha \varphi^r}}
\left(\frac{\partial a_{\alpha,\beta}^{r,s}}{\partial x_i}D^\beta \varphi^s+a_{\alpha,\beta}^{r,s}\Phi(x_i)(D^\beta \varphi^s)\right)
=
\sum_{\substack{(\beta',s) \in \mathbb{N}^n\times\{1,2,\ldots,m\} \\ D^{\beta'} \varphi^s \wostrict D^{\alpha'} \varphi^r}}
\Phi(u)(a_{\alpha',\beta'}^{r,s}D^{\beta'} \varphi^s).
\end{equation}
Using equations of system $(\cone_{\div{J},\wo}(\Sigma))$, we replace all  principal derivatives in the equation (\ref{Equation:CanonicalSystem2}) by parametric derivatives and independent variables. The order $\wo$ being well-founded this process is terminating. Moreover, when the PDE system $(\Sigma)$ is complete this reduction process is confluent in the sense that any transformations of an equation~(\ref{Equation:CanonicalSystem2}) ends on a unique $\div{J}$-normal forms.
This set of $\div{J}$-normal forms is denoted by $\integralcond_{\div{J},\wo}(\Sigma)$. 
\subsubsection{Remarks} 

The system $(\Sigma)$ being complete any equation~(\ref{Equation:CanonicalSystem2}) is reduced to a unique normal form. Such a normal form allows us to judge whether a given integrability condition is trivial or not.

Recall that the parametric derivatives correspond to the initial conditions. Hence, a non-trivial relation in 
$\integralcond_{\div{J},\cwo}(\Sigma)$ provides a non-trivial relation among the initial conditions. In this way, we can decide whether the system $(\Sigma)$ is completely integrable or not.

\subsubsection{Completely integrable systems}
A complete linear PDE system $(\Sigma)$ of the form~(\ref{Equation:CanonicalSystem1}) is said to be \emph{completely integrable} if it admits an analytic solution for any given initial condition~(\ref{Equation:BoundaryCondition2}). For the geometrical interpretation of these condition, we refer the reader to~\ref{Subsubsection:InvolutionFrobenius}.

\begin{theorem}[{\cite[\textsection 42]{Janet29}}]
\label{Theorem:CaracterizationCompleteIntegrability}
Let $(\Sigma)$ be a complete finite linear PDE system of the form~(\ref{Equation:CanonicalSystem1}). Then the system $(\Sigma)$ is completely integrable if and only if any relation in $\integralcond_{\div{J},\wo}(\Sigma)$ is a trivial identity.
\end{theorem}

A proof of this result is given in {\cite[\textsection 43]{Janet29}}.
Note that the later condition is equivalent to say that any relation~(\ref{Equation:CanonicalSystem2}) is an algebraic consequence of a PDE equation of the system $(\cone_{\div{J},\wo}(\Sigma))$.

\subsection{Canonical forms of linear PDE systems}

In this subsection, we recall from~\cite{Janet29} the notion of canonical linear PDE system. A canonical system is a normal form with respect to a weight order on derivatives, and satisfying some analytic conditions, allowing to extend the Cauchy-Kowalevsky's theorem given in~\ref{SSS:CauchyKowalevskyTheorem}. Note that this terminology refers to a notion of normal form, but it does not correspond to the well known notion for a rewriting system meaning both terminating and confluence. In this notes, we present canonical systems with respect to weight order as it done in Janet's monograph \cite{Janet29}, but we notice that this notion can be defined with any total order on derivative. 

\subsubsection{Autoreduced PDE systems}
\label{Subsubsection:AutoreducedSystem}
Let $(\Sigma)$ be a finite linear PDE system. Suppose that a weight order~$\wo$ is fixed on the set of unknown functions $\varphi^1,\ldots,\varphi^m$ of $(\Sigma)$ and their derivatives, as defined in~\ref{SubsubsectionPDEsystemsWithSeveralUnknownFunctions}.
We suppose also that each equation of the system $(\Sigma)$ can be expressed in the following form
\[
(\Sigma^{(\alpha,r)})
\qquad
D^\alpha \varphi^r=\sum_{\substack{(\beta,s) \in \mathbb{N}^n\times\{1,2,\ldots,m\} \\ D^\beta \varphi^s \wostrict D^\alpha \varphi^r}}
a_{(\beta,s)}^{(\alpha,r)}D^\beta \varphi^s,
\]
so that 
\begin{equation}
\label{Equation:SigmaAlphaR}
(\Sigma) = \bigcup_{(\alpha,r)\in I} \Sigma^{(\alpha,r)},
\end{equation}
the union being indexed by a multiset $I$. 
The support of the equation $(\Sigma^{(\alpha,r)})$ is defined by
\[
\mathrm{Supp}( \Sigma^{(\alpha,r)}) = \{\,(\beta,s)\;|\;a_{(\beta,s)}^{(\alpha,r)} \neq 0\;\}.
\]

For $1\leq r \leq m$, consider the set of monomials~$\lm_{\wo}(\Sigma,\varphi^r)$ corresponding to leading derivatives, that is monomials $x^{\alpha}$ such $(\alpha,r)$ belongs to $I$.
The system $(\Sigma)$ is said to be 
\begin{enumerate}[{\bf i)}]
\item \emph{$\div{J}$-left-reduced with respect to $\wo$} if for any $(\alpha,r)$ in $I$ there is no $(\alpha',r)$ in $I$ and non-trivial monomial $x^\gamma$ in $\Mr(\mult_\div{J}^{\lm_{\wo}(\Sigma,\varphi^r)}(x^{\alpha'}))$ such that $x^\alpha=x^\gamma x^{\alpha'}$,
\item \emph{$\div{J}$-rigth-reduced with respect to $\wo$} if, for any $(\alpha,r)$ in $I$ and any $(\beta,s)$ in $\mathrm{Supp}( \Sigma^{(\alpha,r)})$, there is no~$(\alpha',s)$ in 
$I$ and non-trivial monomial $x^\gamma$ in $\Mr(\mult_\div{J}^{\lm_{\wo}(\Sigma,\varphi^r)}(x^{\alpha'}))$ such that $x^\beta=x^\gamma x^{\alpha'}$,
\item \emph{$\div{J}$-autoreduced with respect to $\wo$} if it is both $\div{J}$-left-reduced and $\div{J}$-right-reduced with respect to~$\wo$.
\end{enumerate}

\subsubsection{Canonical PDE systems}
\label{Subsubsection:CanonicalSystems}
A PDE system $(\Sigma)$ is said to be \emph{$\div{J}$-canonical with respect a weight order~$\wo$} if it satisfies the following five conditions
\begin{enumerate}[{\bf i)}]
\item it consists of finitely many equations and each equation can be expressed in the following form
\[
D^\alpha \varphi^r=\sum_{\substack{(\beta,s) \in \mathbb{N}^n\times\{1,2,\ldots,m\} \\ D^\beta \varphi^s \wostrict D^\alpha \varphi^r}}
a_{(\beta,s)}^{(\alpha,r)}D^\beta \varphi^s,
\]
\item the system $(\Sigma)$ is $\div{J}$-autoreduced with respect to $\wo$,
\item the system $(\Sigma)$ is complete,
\item the system $(\Sigma)$ is completely integrable,
\item the coefficients $a_{(\beta,s)}^{(\alpha,r)}$ of the equations in {\bf i)} and the initial conditions of $(\Sigma)$ are analytic.
\end{enumerate}
Under these assumptions, the system $(\Sigma)$ admits a unique analytic solution satisfying appropriate initial conditions parametrized by complementary monomials as in \ref{Subsubsection:BoundaryConditionOneUnknown}.

\subsubsection{Remark}
We note that the notion of canonicity given by Janet in~\cite{Janet29} does not impose the condition being $\div{J}$-autoreduced, even if Janet had mentioned this autoreduced property for some simple cases. The autoreduced property implies the minimality of the system. This fact was formulated by V. P. Gerdt and Y. A. Blinkov in~\cite{GerdtBlinkovYuri98b} with the notion of minimal involutive basis.

\subsubsection{Example}
In~{\cite[\textsection 44]{Janet29}}, M. Janet studied the following linear PDE system of one unknown function~$\varphi$
\[
(\Sigma)\qquad
\begin{cases} 
p_{54}=p_{11}, \\
p_{53}=p_{41}, \\
p_{52}=p_{31}, \\
p_{44}=p_{52}, \\
p_{43}=p_{21}, \\
p_{33}=p_{42}, 
\end{cases}
\]
where $p_{i,j}$ denotes $\dfrac{\partial^2 \varphi}{\partial x_i \partial x_j}$. In Example~\ref{Example:JanetExampleB}, we have shown that the left hand sides of the equations of this system form a complete set of monomials. Let us define the following weights for the variables:
\begin{center}
\begin{tabular}{ccccc} 
$x_1$ & $x_2$ & $x_3$ & $x_4$ & $x_5$ \\ \hline
$1$ & $0$ & $1 $ & $1 $ & $2$ \\
$0$ & $0$ & $0$ & $1$ & $1$ \\
\end{tabular}
\end{center}
We deduce the following weights for the second derivatives:
\begin{center}
\begin{tabular}{c|c|c|c|c|c|c|c|c}
$p_{22}$ & $\begin{matrix} p_{21} \\ p_{32} \end{matrix}$
& $p_{42}$ & $\begin{matrix} p_{11} \\ p_{31} \\ p_{33} \end{matrix}$
& $\begin{matrix} p_{52} \\ p_{41} \\ p_{43} \end{matrix}$
& $p_{44}$ & $\begin{matrix} p_{51} \\ p_{53} \end{matrix}$
& $p_{54}$ & $p_{55}$ \\ \hline
$0$ & $1$ & $1$ & $2$ & $2$ & $2$ & $3$ & $3$ & $4$ \\
$0$ & $0$ & $1$ & $0$ & $1$ & $2$ & $1$ & $2$ & $2$ \\
\end{tabular}
\end{center}
As seen in Example~\ref{Example:JanetExampleB2}, given any four analytic functions
\[ 
\varphi_0(x_1,x_2), \quad \varphi_3(x_1, x_2), \quad \varphi_4(x_1, x_2), \quad \varphi_5(x_1, x_5), 
\]
there exists a unique solution of the PDE system $(\Sigma)$.
Note that the initial condition is given by  
\begin{align*} 
\varphi\vert_{x_3=x_3^0, x_4=x_4^0, x_5=x_5^0}&=\varphi_{0,0,0,0,0}(x_1,x_2),\\
\left. \frac{\partial \varphi}{\partial x_3}\right\vert_{x_3=x_3^0, x_4=x_4^0, x_5=x_5^0}&=\varphi_{0,0,1,0,0}(x_1, x_2), \\
\left. \frac{\partial \varphi}{\partial x_4}\right\vert_{x_3=x_3^0, x_4=x_4^0, x_5=x_5^0}&=\varphi_{0,0,0,1,0}(x_1, x_2), \\
\left. \frac{\partial \varphi}{\partial x_5}\right\vert_{x_2=x_2^0, x_3=x_3^0, x_4=x_4^0}&=\varphi_{0,0,0,0,1}(x_1, x_5).
\end{align*}
We set
\begin{center}
\begin{tabular}{c|ccccc} 
$A=p_{54}-p_{11}$ & $x_5$ & $x_4$ & $x_3$ & $x_2$ & $x_1$ \\
$B=p_{53}-p_{41}$ & $x_5$ & $$ & $x_3$ & $x_2$ & $x_1$ \\
$C=p_{52}-p_{31}$ & $x_5$ & $$ & $$ & $x_2$ & $x_1$ \\
$D=p_{44}-p_{52}$ & $$ & $x_4$ & $x_3$ & $x_2$ & $x_1$ \\
$E=p_{43}-p_{21}$ & $$ & $$ & $x_3$ & $x_2$ & $x_1$ \\
$F=p_{33}-p_{42}$ & $$ & $$ & $x_3$ & $x_2$ & $x_1$ $$
\end{tabular}
\end{center}
where the variable on the right correspond to the multiplicative variables of the first term.
In order to decide if the system $(\Sigma)$ is completely integrable  it suffices to check if the following  terms  
\[ 
B_4, C_4, C_3, D_5, E_5, E_4, F_5, F_4 
\]
are linear combinations of derivative of the terms $A,B,C,D,E,F$ with respect to their multiplicative variables. 
Here $Y_i$ denotes the derivative $\frac{\partial}{\partial x_i} Y$ of a term $Y$.
Finally, we observe that
\begin{align*}
&B_4=A_3-D_1-C_1, \\
&C_4=A_2-E_1, \qquad C_3=B_2-F_1, \\
&D_5=A_4-B_1-C_5, \\
&E_5=A_3-C_1, \qquad E_4=D_3+B_2, \\
&F_5=B_3-A_2+E_1, \qquad F_4=E_3-D_2-C_2.
\end{align*}
As a consequence the system $(\Sigma)$ is completely integrable, hence it is $\div{J}$-canonical.

\subsection{Reduction of a PDE system to a canonical form}
\label{Subsection:ReductionPDEsystemToCanonicalForm}
In his monograph~\cite{Janet29}, M. Janet did not mention about the correctness of the procedures that he introduced in order to reduce a finite linear PDE system to a canonical form.
In this section, we explain how to transform a finite linear PDE system with several unknown functions by derivation, elimination and autoreduction, into an equivalent linear PDE system that is either in canonical form or in incompatible system. For linear PDE systems with constant coefficients, the correctness of the procedure can be verified easily.

\subsubsection{Equivalence of PDE system} 
Janet's procedure transforms by reduction and completion a finite linear PDE system into a new PDE system. The PDE system obtained in this way is equivalent to the original system. In his work, M. Janet dit not explain this notion of equivalence that can be described as follows. 
Consider two finite linear PDE systems with $m$ unknown functions and $n$ independent variables
\[
(\Sigma^l) \qquad \sum_{j=1}^m p_{i,j}^l \varphi^j =0, \qquad i \in I^l,
\]
for $l=1,2$, where $p_{i,j}^l$ are linear differential operators.
We say that the PDE systems $(\Sigma^1)$ and $(\Sigma^2)$ are \emph{equivalent} if the set of solutions of the two systems are the same. This notion can be also formulated by saying that the $D$-modules generated by the families of differentials operators $(p_{i,1}^1,\ldots,p_{i,m}^1)$ for $i\in I^1$ and $(p_{i,1}^2,\ldots,p_{i,m}^2)$ for $i\in I^2$ are equals.

\subsubsection{A canonical weight order}
\label{Subsubsection:CanonicalWeightOrder}
Consider a finite linear PDE system $(\Sigma)$ of $m$ unknown functions~$\varphi^1,\ldots, \varphi^m$ of independent variables $x_1,\ldots, x_n$.
To these variables and functions we associate the following weights 
\begin{center}
\begin{tabular}{ccccc|cccc} 
$x_1$ & $x_2$ & $\ldots$ & $x_{n-1}$ & $x_n$ & $\varphi^1$ & $\varphi^2$ & $\ldots$ & $\varphi^m$ \\ \hline
$1$ & $1$ & $\ldots$ & $1$ & $1$ & $0$ & $0$ & $\ldots$ & $0$ \\
$0$ & $0$ & $\ldots$ & $0$ & $0$ & $1$ & $2$ & $\ldots$ & $m$ \\
$0$ & $0$ & $\ldots$ & $0$ & $1$ & $0$ & $0$ & $\ldots$ & $0$ \\
$0$ & $0$ & $\ldots$ & $1$ & $0$ & $0$ & $0$ & $\ldots$ & $0$ \\
$\vdots$ & $\vdots$ & & $\vdots$ & $\vdots$ & $\vdots$ & $\vdots$ & & $\vdots$ \\
$0$ & $1$ & $\ldots$ & $0$ & $0$ & $0$ & $0$ & $\ldots$ & $0$ \\
$1$ & $0$ & $\ldots$ & $0$ & $0$ & $0$ & $0$ & $\ldots$ & $0$ \\
\end{tabular}
\end{center}
The weight order on monomial partial derivatives defined in~\ref{sect-gen} induced by this weight system is total. This order is called \emph{canonical weight order} following M. Janet and denoted by $\cwo$.

\subsubsection{Combination of equations}
Consider the PDE system $(\Sigma)$ with the canonical weight order~$\cwo$ defined in \ref{Subsubsection:CanonicalWeightOrder}. We suppose that the system $(\Sigma)$ is given in the same form as (\ref{Equation:SigmaAlphaR}) and that each equation of the system is written in the following form
\[
(E^{(\alpha,r)}_i)
\qquad
D^\alpha \varphi^r=\sum_{\substack{(\beta,s) \in \mathbb{N}^n\times\{1,2,\ldots,m\} \\ D^\beta \varphi^s \cwostrict D^\alpha \varphi^r}}
a_{(\alpha,r),i}^{(\beta,s)}D^\beta \varphi^s,
\qquad
i \in I^{(\alpha,r)}.
\]
For such an equation, the \emph{leading pair} $(\alpha,r)$ of the equation $E^{(\alpha,r)}_i$ will be denoted by $\mathrm{ldeg}_{\cwo}(E_{i}^{\alpha,r})$. We will denote by $\mathrm{Ldeg}_{\cwo}(\Sigma)$ the subset of $\mathbb{N}^n\times \{1,\ldots,m\}$ consisting of leading pairs of equations of the system $(\Sigma)$:
\[
\mathrm{Ldeg}_{\cwo}(\Sigma) = \{\,\mathrm{ldeg}_{\cwo}(E) \;|\; E\;\text{is an equation of}\;  \Sigma\,\}.
\]
The canonical weight order $\cwo$ induces a total order on $\mathbb{N}^n\times \{1,\ldots,m\}$ denoted by $\prec_{lp}$.
We will denote by $K(\alpha,r,i)$ the set of pairs $(\beta,s)$ of running indices in the sum of the equation $E^{(\alpha,r)}_i$. Given~$i$ and~$j$ in $I^{(\alpha,r)}$, we set
\[
(\alpha_{i,j},r_{i,j}) = \mathrm{Max} \big( (\beta,s)\in  K(\alpha,r,i) \cup K(\alpha,r,j) \; |\;
a_{(\alpha,r),i}^{(\beta,s)} \neq a_{(\alpha,r),j}^{(\beta,s)}\big).
\]
We define 
\begin{equation}
\label{Equation:CoefficientB}
b_{(\alpha,r)}^{(\alpha_{i,j},r_{i,j})} =
\begin{cases} 
a_{(\alpha,r),i}^{(\alpha_{i,j},r_{i,j})} & \mbox{if } (\alpha_{i,j},r_{i,j})\in K(\alpha,r,i) \setminus K(\alpha,r,j), \\ 
-a_{(\alpha,r),i}^{(\alpha_{i,j},r_{i,j})} & \mbox{if } (\alpha_{i,j},r_{i,j})\in K(\alpha,r,j) \setminus K(\alpha,r,i), \\ 
a_{(\alpha,r),i}^{(\alpha_{i,j},r_{i,j})} - a_{(\alpha,r),i}^{(\alpha_{i,j},r_{i,j})} & \mbox{if } (\alpha_{i,j},r_{i,j})\in K(\alpha,r,i) \cap K(\alpha,r,j), \\ 
\end{cases} 
\end{equation}
and we denote by $E_{i,j}^{(\alpha,r)}$ the equation
\begin{equation}
\label{Equation:Compose}
D^{\alpha_{i,j}}\varphi^{r_{i,j}}= 
\sum_{(\beta,s) \in K(\alpha,r,j) \atop (\beta,s) \prec_{lp} (\alpha_{i,j},r_{i,j})} 
c_{(\alpha_{i,j},r_{i,j}),j}^{(\beta,s)} D^\beta \varphi^s
-
\sum_{(\beta,s) \in K(\alpha,r,i) \atop (\beta,s) \prec_{lp} (\alpha_{i,j},r_{i,j})} 
c_{(\alpha_{i,j},r_{i,j}),i}^{(\beta,s)} D^\beta \varphi^s,
\end{equation}
where, for any $k=i,j$,
\[
c_{(\alpha_{i,j},r_{i,j}),k}^{(\beta,s)} = a_{(\alpha,r),k}^{(\beta,s)}/b_{(\alpha,r)}^{(\alpha_{i,j},r_{i,j})}.
\]

The equation (\ref{Equation:Compose}) corresponds to a combination of the two equations $E_{i}^{(\alpha,r)}$ and $E_{j}^{(\alpha,r)}$ and it will be denoted by $\mathbf{Combine}_{\cwo}(E_{i}^{(\alpha,r)},E_{j}^{(\alpha,r)})$. Procedure~\ref{A:Add} adds to a set of PDE equations $(\Sigma)$ an equation $E$ by combination.

\begin{algorithm}
\SetAlgoLined
\KwIn{

\begin{tabular}{l}
- A canonical weight order $\cwo$ for $\varphi^1,\ldots,\varphi^m$ and $x_1,\ldots,x_n$.\\
- $(\Sigma)$ a finite linear PDE system with unknown functions $\varphi^1,\ldots,\varphi^m$ of independent variables \\
\qquad $x_1,\ldots,x_n$ given in the same form as (\ref{Equation:SigmaAlphaR}) such that the leading derivatives are different.
\\
- $E$ be a linear PDE in the same form as (\ref{Equation:SigmaAlphaR}).
\end{tabular}
}

\BlankLine

\Begin{

$\Gamma \leftarrow \Sigma$

$(\beta,s) \leftarrow \mathrm{ldeg}_{\cwo}(E)$

\If{$(\beta,s)\notin \mathrm{Ldeg}_{\cwo}(\Gamma)$}{

$\Gamma \leftarrow \Gamma \cup \{E\}$

}

\Else{

let $E^{(\beta,s)}$ be the equation of the system $(\Sigma)$ whose leading pair is $(\beta,s)$.

$C \leftarrow {\bf Combine}_{\cwo}(E^{(\beta,s)},E)$ 

${\bf Add}_{\cwo}(\Gamma , C)$

}

}

\BlankLine

\KwOut{$\Gamma$ a PDE system equivalent to the system obtained from $(\Sigma)$ by adding equation $E$.}
\caption{${\bf Add}_{\cwo}(\Sigma$, $E$)}
\label{A:Add}
\end{algorithm}

\begin{algorithm}
\SetAlgoLined
\KwIn{

\begin{tabular}{l}
- A canonical weight order $\cwo$ for $\varphi^1,\ldots,\varphi^m$ and $x_1,\ldots,x_n$.\\
- $(\Sigma)$ a finite linear PDE system with unknown functions $\varphi^1,\ldots,\varphi^m$ of independent variables \\
\qquad $x_1,\ldots,x_n$ given in the same form as (\ref{Equation:SigmaAlphaR}) such that the leading derivatives are different.
\end{tabular}
}

\BlankLine

\Begin{

\BlankLine

$\Gamma \leftarrow \Sigma$

$I \leftarrow \mathrm{Ldeg}_{\cwo}(\Gamma)$

$\Ur_r \leftarrow \{x^{\alpha} \;|\; (\alpha,r)\in I\}$

\While{\big(exist $(\alpha,r),(\alpha',r)$ in $I$ and non-trivial monomial $x^\gamma$ in $\Mr(\mult_\div{J}^{\Ur_r}(x^{\alpha'}))$ such that $x^\alpha=x^\gamma x^{\alpha'}$\big)}{

\BlankLine

$\Gamma \leftarrow \Gamma \setminus \{E^{(\alpha,r)}\}$

Let $D^\gamma E^{(\alpha',r)}$ be the equation obtained from the equation $E^{(\alpha',r)}$ by applying the operator $D^\gamma$ to the two sides.

$C \leftarrow {\bf Combine}_{\cwo}(E^{(\alpha,r)}, D^\gamma E^{(\alpha',r)})$ 

${\bf Add}_{\cwo}(\Gamma,C)$
}}

\BlankLine

\KwOut{$\Gamma$ a $\div{J}$-left-reduced PDE system with respect to $\cwo$ that is equivalent to $(\Sigma)$.}
\caption{${\bf LeftReduce}_{\div{J},{\cwo}}(\Sigma)$}
\end{algorithm}

\begin{algorithm}
\SetAlgoLined
\KwIn{

\begin{tabular}{l}
- A canonical weight order $\cwo$ for $\varphi^1,\ldots,\varphi^m$ and $x_1,\ldots,x_n$.\\
- $(\Sigma)$ a finite linear PDE system with unknown functions $\varphi^1,\ldots,\varphi^m$ of independent variables \\
\qquad $x_1,\ldots,x_n$ given in the same form as (\ref{Equation:SigmaAlphaR}) and that is $\div{J}$-left reduced with respect to $\cwo$.
\end{tabular}
}

\BlankLine

\Begin{

$\Gamma \leftarrow \Sigma$

$\Gamma' \leftarrow \Gamma$

$I \leftarrow \mathrm{Ldeg}_{\cwo}(\Gamma)$ 

$//$ \emph{The canonical weight order $\cwo$ induces a total}

$//$ \emph{order on the set $I$ of leading pairs denoted by $\preccurlyeq_{lp}$}

$(\delta,t) \leftarrow \mathrm{max}(I)$ with respect to $\preccurlyeq_{lp}$

\BlankLine

\While{$\Gamma' \neq \emptyset$}{

$\Gamma' \leftarrow \Gamma' \setminus \{E^{(\delta,t)}\}$

$I \leftarrow I \setminus \{(\delta,t)\}$

$S \leftarrow \mathrm{Supp}(E^{(\delta,t)})$ 

$\Ur_r \leftarrow \{x^{\alpha} \;|\; (\alpha,r)\in I\}$

\While{\big(exist $(\beta,r)$ in $S$, $(\alpha,r)$ in $I$ and non-trivial monomial $x^\gamma$ in $\Mr(\mult_\div{J}^{\Ur_r}(x^{\alpha}))$ such that $x^\beta=x^\gamma x^{\alpha}$\big)}{

\BlankLine

$\Gamma \leftarrow \Gamma \setminus \{E^{(\delta,t)}\}$

$C \leftarrow E^{(\delta,t)} - a_{(\delta,t)}^{(\beta,r)} D^\beta \varphi^r + a_{(\delta,t)}^{(\beta,r)} D^\gamma (\mathrm{Rhs}(E^{(\alpha,r)}))$ 

${\bf Add}_{\cwo}(\Gamma,C)$
}}
}

\BlankLine

\KwOut{$\Gamma$ a $\div{J}$-right-reduced PDE system with respect to $\cwo$ that is equivalent to $(\Sigma)$.}
\caption{${\bf \;RightReduce}_{\div{J},{\cwo}}(\Sigma)$}
\label{A:RightReduce}
\end{algorithm}

Note that at each step of the procedure ${\bf RightReduce}_{\div{J},\cwo}$ the running system $\Gamma$ remains $\div{J}$-left reduced. As consequence by combining this procedure with the procedure ${\bf LeftReduce}_{\div{J},\cwo}$ we obtain the following autoreduce procedure that transform a PDE system into a autoreduced PDE system.

\subsubsection{Procedure ${\bf Autoreduce}_{\div{J},\cwo}(\Sigma)$}
\label{SSS:ProcedureAutoreduce}
Let us fix a canonical weight order $\cwo$ for $\varphi^1,\ldots,\varphi^m$ and~$x_1,\ldots,x_n$.
Let $(\Sigma)$ be a finite linear PDE system given in the same form as~(\ref{Equation:SigmaAlphaR}) with unknown functions $\varphi^1,\ldots,\varphi^m$ of independent variables $x_1,\ldots,x_n$. We suppose that the leading derivatives of $(\Sigma)$ are all different.
The procedure ${\bf Autoreduce}_{\div{J},\cwo}$ transforms the PDE system $(\Sigma)$ into an $\div{J}$-autoreduced PDE system that is equivalent to $(\Sigma)$ by applying successively the procedures ${\bf LeftReduce}_{\div{J},\cwo}$ and ${\bf RightReduce}_{\div{J},\cwo}$. 
An algebraic version of this procedure is given in Procedure~\ref{AutoreductionProcedure2}. 
Let us remark that the autoreduction procedure given in Janet's monographs corresponds to the ${\bf LeftReduce}_{\div{J},\cwo}$, it does not deal with right reduction of equations. 

Note that, the procedure ${\bf Autoreduce}_{\div{J},\cwo}$ fails if and only if the procedure ${\bf Combine}_{\cwo}$ fails. This occurs when the procedure ${\bf Combine}_{\cwo}$ is applied on equations $E_i^{(\alpha,r)}$ and $E_j^{(\alpha,r)}$ and some coefficients~$b_{(\alpha,r)}^{(\alpha_{i,j},r_{i,j})}$, as defined in (\ref{Equation:CoefficientB}), vanish on some point of $\mathbb{C}^n$. In particular, the procedure~${\bf Autoreduce}_{\div{J},\cwo}$ does not fail when all the coefficients are constant. This constraint on the coefficients of the system concerns only the left reduction  and were not discussed in Janet's monograph.
As a consequence, we have the following result.

\begin{theorem}
\label{Theorem:TerminationAutoreduce}
If $(\Sigma)$ is a finite linear PDE system with constant coefficients, the procedure \linebreak ${\bf Autoreduce}_{\div{J},\cwo}$ terminates and produces a finite autoreduced PDE system that is equivalent to $(\Sigma)$. 
\end{theorem}

\subsubsection{Completion procedure of a PDE system}
\label{CompletionPDEsystem}
Consider a finite linear PDE system $(\Sigma)$ with the canonical weight order $\cwo$ given in \ref{Subsubsection:CanonicalWeightOrder}. If the system $(\Sigma)$ is $\div{J}$-autoreduced, then the following procedure ${\bf Complete}_{\div{J},\cwo}(\Sigma)$ transforms the system $(\Sigma)$ into a finite complete $\div{J}$-autoreduced linear PDE system. This procedure of completion appears in Janet's monograph~\cite{Janet29} but not given in an explicit way. 

\begin{algorithm}
\SetAlgoLined
\KwIn{ 

\begin{tabular}{l}
- A canonical weight order $\cwo$ for $\varphi^1,\ldots,\varphi^m$ and $x_1,\ldots,x_n$.\\
- $(\Sigma)$ a finite $\div{J}$-autoreduced linear PDE system with unknown functions $\varphi^1,\ldots,\varphi^m$ of independent \\
\quad variables $x_1,\ldots,x_n$ given in the same form as (\ref{Equation:SigmaAlphaR}) and whose leading derivatives are different.
\end{tabular}
}

\BlankLine

\Begin{

$\Gamma \leftarrow \Sigma$

$\Xi \leftarrow \emptyset$

\For{$r=1,\ldots,m$}{

\BlankLine

\While{$\Xi=\emptyset$}{

\BlankLine

$I \leftarrow \mathrm{Ldeg}_{\cwo}(\Gamma)$

$\Ur_r \leftarrow \{x^{\alpha} \;|\; (\alpha,r)\in I\}$

$\Pr_r \leftarrow \big\{\frac{\partial E}{\partial x}  \;|\; E\in \Gamma,\; x\in \nonmult_\div{J}^{\Ur_r}(x^\delta)\;\text{with $(\delta,r)=\mathrm{ldeg}(E)$ and $xx^\delta \notin \cone_{\div{J}}(\Ur_r) $}\big\}$

$C \leftarrow 0$

\BlankLine

\While{$\Pr_r \neq \emptyset$ and $C=0$}{

\BlankLine

{\bf choose} $E^{(\beta,r)}$ in $\Pr_r$,  whose leading pair $(\beta,r)$ is minimal with respect to $\cwo$. 

\BlankLine

$\Pr_r \leftarrow \Pr_r\setminus\{E^{(\beta,r)}\}$ 

$C \leftarrow E^{(\beta,r)}$

$S_C \leftarrow \mathrm{Supp}(C)$ 

\BlankLine

\While{exist $(\delta,r)$ in $S_C$, $(\alpha,r)$ in $I$ and $x^\gamma$ in $\Mr(\mult_\div{J}^{\Ur_{r}}(x^{\alpha}))$ such that $x^\delta=x^\gamma x^{\alpha}$}{

\BlankLine

$C \leftarrow C - a_{(\beta,r)}^{(\delta,r)} D^\delta \varphi^r + a_{(\beta,r)}^{(\delta,r)} D^\gamma (\mathrm{Rhs}(E^{(\alpha,r)}))$

$S_C \leftarrow \mathrm{Supp}(C)$ 

}

}

\If{$C\neq 0$}{

$\Gamma \leftarrow {\bf Autoreduce}_{\div{J},\cwo}(\Gamma\cup\{C\})$

}

\Else{

$\Xi\leftarrow\Gamma$

}}}}

\BlankLine

\KwOut{$(\Xi)$ a linear $\div{J}$-autoreduced PDE system equivalent to $(\Sigma)$ and that is complete with respect to $\cwo$.}
\caption{${\bf Complete}_{\div{J},\cwo}(\Sigma)$}
\label{Procedure:CompletePDESystem}
\end{algorithm}

\subsubsection{Completion and integrability conditions}
In Procedure~\ref{Procedure:CompletePDESystem}, the set $\Pr_r$ contains all the obstructions of the system to be complete. The procedure ${\bf Complete}_{\div{J},\cwo}$ add to the system the necessary equations in order to eliminate all these obstructions. The equations added to the system have the following form 
\[
D^\beta\varphi^r = \mathrm{Rhs}(E^{(\beta,r)}) - a_{(\beta,r)}^{(\delta,r)} D^\delta \varphi^r + a_{(\beta,r)}^{(\delta,r)} D^\gamma (\mathrm{Rhs}(E^{(\alpha,r)}))
\]
with $\delta \neq \beta$ and lead to the definition of new integrability condition of the form (\ref{Equation:CanonicalSystem2}) by using the construction given in~\ref{Subsubsection:IntegrabilityConditions}. 

\subsubsection{Janet's procedure}
Given a finite linear PDE system $(\Sigma)$ with the canonical weight order~$\cwo$ defined in~\ref{Subsubsection:CanonicalWeightOrder}, \emph{Janet's procedure} ${\bf Janet}_{\div{J},\cwo}$ either transforms the system $(\Sigma)$ into a PDE system $(\Gamma)$ that is $\div{J}$-canonical with respect to $\cwo$ or computes an obstruction to transform the system $(\Sigma)$ to such a form. In the first case, the solutions of the $\div{J}$-canonical system $(\Gamma)$ are solutions of the initial system $(\Sigma)$. In the  second case, the obstruction corresponds to a non-trivial relation on the initial conditions.
We refer the reader to \cite{Schwarz92} or \cite{Robertz14} for a deeper discussion on this procedure and its implementations.

Applying successively the procedures $\autoreduce_{\div{J}}$ and $\complete_{\div{J}}$, the first step of the procedure consists in reducing the PDE system $(\Sigma)$ into a PDE system $(\Gamma)$ that is $\div{J}$-autoreduced and complete with respect to $\cwo$.  

Then it computes the set $\integralcond_{\div{J},\cwo}(\Gamma)$ of integrability conditions of the system $(\Gamma)$.
Recall from~\ref{Subsubsection:IntegrabilityConditions} that this set is  a finite set of relations that does not contain principal derivative.  Hence, these integrability conditions are $\div{J}$-normal forms with respect to  $\Gamma$. The system $(\Gamma)$ being complete, these normal forms are unique and by Theorem~\ref{Theorem:CaracterizationCompleteIntegrability}, if all of these normal forms are trivial, then the system $(\Gamma)$ is completely integrable. Otherwise, the procedure takes a non-trivial condition $\mathcal{R}$ in the set~$\integralcond_{\div{J},\cwo}(\Gamma)$ and distinguishes two cases. If the relation $\mathcal{R}$ is among functions $\varphi^1,\ldots, \varphi^m$ and variables $x_1,\ldots,x_n$, then this relation imposes a relation on the initial conditions of the system $(\Gamma)$.
In the other case, the set $\integralcond_{\div{J},\cwo}(\Gamma)$ contains at least one PDE having a derivative of one of the functions $\varphi^1,\ldots,\varphi^m$ and the procedure  ${\bf Janet}_{\div{J},\cwo}$ is applied again to the PDE system $(\Sigma)$ completed by all the PDE equations in $\integralcond_{\div{J},\cwo}(\Gamma)$.

\begin{algorithm}
\SetAlgoLined
\KwIn{

\begin{tabular}{l}
- A canonical weight order $\cwo$ for $\varphi^1,\ldots,\varphi^m$ and $x_1,\ldots,x_n$.\\
- $(\Sigma)$ a finite linear PDE system with unknown functions $\varphi^1,\ldots,\varphi^m$ of independent variables \\
\qquad $x_1,\ldots,x_n$ given in the same form as (\ref{Equation:SigmaAlphaR}) and whose leading derivatives are different.
\end{tabular}

}

\BlankLine

\Begin{

\BlankLine

$\Gamma \leftarrow \autoreduce_{\div{J},\cwo}(\Sigma)$

$\Gamma \leftarrow \complete_{\div{J},\cwo}(\Gamma)$

$C \leftarrow \integralcond_{\div{J},\cwo}(\Gamma)$

\BlankLine

\If{$C$ consists only of trivial identities}{ 

\BlankLine

\Return The PDE system $(\Sigma)$ is transformable to a $\div{J}$-canonical system $(\Gamma)$.
}

\If{$C$ contains a non-trivial relation $\mathcal{R}$ among functions $\varphi^1,\ldots,\varphi^m$ and variables $x_1,\ldots,x_n$}{

\Return The PDE system $(\Sigma)$ is not reducible to a $\div{J}$-canonical system and the relation $\mathcal{R}$ imposes a non-trivial relation on the initial conditions of the system $(\Gamma)$.
}

\Else{
// \emph{$C$ contains a non-trivial relation among functions $\varphi^1,\ldots,\varphi^m$, variables $x_1,\ldots,x_n$,}

// \emph{and at least one derivative of one of the functions $\varphi^1,\ldots,\varphi^m$.}

$\Sigma \leftarrow \Sigma \cup \{C \}$

$\janet_{\div{J},\cwo}(\Sigma)$.
}
}

\BlankLine

\KwOut{Complete integrability of the system $(\Sigma)$ and its obstructions to be reduced to a $\div{J}$-canonical form with respect to $\cwo$.}

\caption{${\bf Janet}_{\div{J},\cwo}(\Sigma)$}
\label{Subsubsection:JanetCompletionProcedure}
\end{algorithm}

\subsubsection{Remarks}
If the procedure stops at the first loop, that is when $C$ consists only of trivial identities, then the system $(\Sigma)$ is reducible to the $\div{J}$-canonical form $(\Gamma)$ equivalent to $(\Sigma)$.

When the set $C$ contains an integrability condition having at least one derivative of the unknown functions, the procedure is applied again to the system $(\Sigma)\cup C$. Notice that, it could be also possible to recall the procedure on $(\Gamma)\cup C$, but as done in Janet's monograph~\cite{Janet29}, we choose to restart the procedure on $(\Sigma)\cup C$ in order to have a PDE system where each equation has a clear meaning, either it comes from the initial problem or the integrability condition. 

Finally, note that the procedure ${\bf Janet}_{\div{J},\cwo}$ fails on a PDE system $(\Sigma)$ if and only if the procedure \linebreak ${\bf Autoreduce}_{\div{J},\cwo}$ fails on $(\Sigma) \cup C$, where $C$ consists of the potential non-trivial relations among the unknown functions and variables added during the process, as explained in \ref{SSS:ProcedureAutoreduce}.
In particular, by Theorem~\ref{Theorem:TerminationAutoreduce}, if $(\Sigma)$ is a finite linear PDE system with constant coefficients, the procedure ${\bf Autoreduce}_{\div{J},\cwo}$ terminates and produces a finite autoreduced PDE system that is equivalent to $(\Sigma)$. 

\subsubsection{Example}
In~{\cite[\textsection 47]{Janet29}}, M. Janet studied the following PDE system:
\[ 
(\Sigma) \qquad 
\begin{cases} p_{33}=&x_2p_{11}, \\
                       p_{22}=&0, 
\end{cases}
\]
where $p_{i_1\ldots i_k}$ denotes the derivative $\dfrac{\partial^k \varphi}{\partial x_{i_1}\ldots \partial x_{i_k}}$ of an unknown function $\varphi$ of independent variables~$x_1,x_2,x_3$.
The set of monomials of the left hand side of the system $(\Sigma)$ is $\Ur=\{x_3^2, x_2^2\}$. The set $\Ur$ is not complete. Indeed, for instance the monomial $x_3x_2^2$ is not in the involutive cone $\cone_\div{J}(\Ur)$.  If we complete the set $\Ur$ by the monomial $x_3x_2^2$ we obtain a complete set $\widetilde{\Ur}:=\Ur \cup \{x_3x_2^2\}$. The PDE system~$(\Sigma)$ is then equivalent to the following PDE system 
\[ 
(\Gamma) \qquad 
\begin{cases} p_{33}=& x_2p_{11}, \\
                         p_{322}=&0, \\
                         p_{22}=&0.
\end{cases}
\]
Note that $p_{322}=\partial_{x_3}p_{22}=0$. The table of multiplicative variables with respect to the set $\widetilde{\Ur}$ is given by 
\begin{center}
\begin{tabular}{c|ccc}
$x_3^2$ & $x_3$ & $x_2$ & $x_1$ \\
$x_3x_2^2$& $$ & $x_2$ & $x_1$ \\
$x_2^2$ & $$ & $x_2$ & $x_1$ 
\end{tabular}
\end{center}
We deduce that there exists only one non-trivial compatibility condition, formulated as follows
\begin{align*}
p_{3322}=
&\partial_{x_3}p_{322}=\partial_{x_2}^2p_{33}, \qquad (x_3.x_3x_2^2=(x_2)^2.x_3^2) 
\\
=
&\partial_{x_2}^2(x_2p_{11})=2p_{211}+x_2p_{2211}=2p_{211}=0, \qquad (p_{2211}=\partial_{x_1}^2 p_{22}=0).
\end{align*}
Hence, $p_{211}=0$ is a non-trivial relation of the system $(\Gamma)$. As a consequence, the PDE system $(\Sigma)$ is not completely integrable. Then, we consider the new PDE system given by 
\[ 
(\Sigma') \quad 
\begin{cases} p_{33}=&x_2p_{11}, \\
                       p_{22}=&0, \\
                       p_{211}=&0.
\end{cases}
\]
The associated set of monomials $\Ur'=\{x_3^2, x_2^2, x_2x_1^2\}$ is not complete. It can be completed into the complete set $\widetilde{\Ur'}:=\Ur' \cup \{x_3x_2^2, x_3x_2x_1^2\}$. The PDE system $(\Sigma')$ is then equivalent to the following PDE system
\[ 
(\Gamma') \quad
\begin{cases} p_{33}=&x_2p_{11}, \\
                          p_{322}=&0, \\
                          p_{3211}=&0, \\
                          p_{22}=&0, \\
                          p_{221}=&0.
\end{cases}
\]
Note that $p_{322}=\partial_{x_3}p_{22}$ and $p_{3211}=\partial_{x_3}p_{211}$. The multiplicative variables with respect to the set of monomials $\Ur'$ is given by the following table
\begin{center}
\begin{tabular}{c|ccc}
$x_3^2$ & $x_3$ & $x_2$ & $x_1$ \\
$x_3x_2^2$& $$ & $x_2$ & $x_1$ \\
$x_3x_2x_1^2$ & $$ & $$ & $x_1$ \\
$x_2^2$ & $$ & $x_2$ & $x_1$ \\
$x_2x_1^2$ & $$ & $$ & $x_1$
\end{tabular}
\end{center}
We deduce that the only non-trivial compatibility relation is 
\begin{align*}
p_{33211}=
&\partial_{x_3}(p_{3211})=0 \\
=
&\partial_{x_1}^2\partial_{x_2}(p_{33})=\partial_{x_1}^2\partial_{x_2}(x_2p_{11}) \\
=
&\partial_{x_1}^2(p_{11}+x_2p_{211})=p_{1111} \qquad (p_{211}=0).
\end{align*}
We deduce that $p_{1111}=0$ is a non-trivial relation of the system $(\Gamma')$. Hence, the system $(\Sigma')$ is not completely integrable. Then, we consider the new PDE system given by
\[ 
(\Sigma'') \quad 
\begin{cases}
p_{33}=&x_2p_{11}, \\
                         p_{22}=&0, \\
                         p_{211}=&0, \\
                         p_{1111}=&0.
                         \end{cases}
\]
The associated set of monomials $\Ur''=\{x_3^2, x_2^2, x_2x_1^2, x_1^4\}$ is not complete. It can be completed into the set of monomials $\widetilde{\Ur''}:=\Ur''\cup \{x_3x_2^2, x_3x_2x_1^2, x_3x_1^4\}$. The PDE system $(\Sigma'')$ is equivalent to the following system 
\[ 
(\Gamma'')\quad
\begin{cases} p_{33}=&x_2p_{11}, \\
                          p_{322}=&0, \\
                          p_{31111}=&0, \\
                          p_{22}=&0, \\
                          p_{211}=&0, \\
                          p_{1111}=&0. \end{cases}
\]
Note that $p_{322}=\partial_{x_2}p_{22}$ and $p_{31111}=\partial_{x_3}p_{1111}$. All the compatibility conditions are trivial identities, by Theorem~\ref{Theorem:CaracterizationCompleteIntegrability} we deduce that the PDE $(\Sigma'')$, obtained from the initial PDE system $(\Sigma)$ by adding compatibility conditons, is completely integrable. 

\subsubsection{Remark} Let us mention, that using a similar procedure presented in this section, M. Janet in~{\cite[\textsection 48]{Janet29}} gave a constructive proof of a result obtained previously by A. Tresse~\cite{Tresse94}, that a infinite linear PDE system can be reduced to a finite linear PDE system.

\subsection{Algebra, geometry and PDE}
The notion of ideal first appeared in the work of R. Dedekind. This notion appeared also in a seminal paper  \cite{Hilbert1890} of D. Hilbert, were he developed the theory of ideals in polynomial rings. In particular, he proved noetherianity results as the noetherianity of the ring of polynomials over a field, now called Hilbert's basis theorem. In its works on PDE systems,~\cite{Janet22a, Janet22b, Janet24}, M. Janet used the notion of ideal generated by homogeneous polynomials under the terminology of \emph{module of forms}, that he defined as follows. He called \emph{form} a homogeneous polynomial with several variables and he defined a \emph{module} of forms as an algebraic system satisfying the two following conditions: 
\begin{enumerate}[{\bf i)}]
\item if a form $f$ belongs to the system, then the form $hf$ belongs to the system for every form $h$, 
\item if $f$ and $g$ are two forms in the system of the same order, then the form $f+g$ belongs to the system.
\end{enumerate}
Finally, in {\cite[\textsection 51]{Janet29}}, M. Janet recall Hilbert's basis theorem.

\subsubsection{Characteristic functions of homogeneous ideals}
\label{SSS:CharacteristicFunctions}
In {\cite[\textsection 51]{Janet29}}, M. Janet recalled the Hilbert description of the problem of finding the number of independent conditions so that a homogenous polynomial of order $p$ belongs to a given homogeneous ideal. This independent conditions correspond to the independent linear forms that vanish all homogenous polynomials of degree $p$ in the ideal. M. Janet recalled from \cite{Hilbert1890} that this number of independent conditions is expressed as a polynomial in $p$ for sufficiently big $p$.

Let $I$ be a homogenous ideal of $\K[x_1,\ldots, x_n]$ generated by polynomials $f_1, \ldots, f_k$. Given a monomial order on $\Mr(x_1,\ldots,x_n)$, we can suppose that all the leading coefficients are equal to $1$.  For any~$p\geq 0$, consider the homogenous component of degree $p$ so that $I=\bigoplus_{p} I_p$ with
\[
I_p:=I \cap \K[x_1,\ldots x_n]_p.
\]
Let us recall that
\[ 
\dim I_p \leq \dim \big( \K[x_1,\ldots, x_n]_p \big)=\Gamma_n^p.
\]
The number of independent conditions so that a homogenous polynomial of order $p$ belongs to the ideal~$I$ is given by the difference 
\[
\chi(p) := \Gamma_n^p - \dim I_p.
\]
This is the number of monomials of degree $p$ that cannot be divided by the monomials $\lm(f_1), \ldots, \lm(f_k)$. The function $\chi(p)$ corresponds to a coefficient of the Hilbert series of the ideal $I$ and is called \emph{characteristic function} of the ideal $I$ by M. Janet, or \emph{postulation} in~{\cite[\textsection 52]{Janet29}}. 
We refer the reader to~\cite{Eisenbud95} for the definition of Hilbert series of polynomial rings and its applications. 
In the Section \ref{Section:InvolutiveSystem}, we will show that the function $\chi(p)$ is polynomial for sufficiently big~$p$.
Finally, note that the set of monomials that cannot be divided by the monomials $\lm(f_1), \ldots, \lm(f_k)$ forms a finite number of classes of complementary monomials.

\subsubsection{Geometrical remark}
\label{GeometricalDescription}
M. Janet gave the following geometrical observation about the characteristic function.  Suppose that $p$ is sufficiently big so that the function $\chi(p)$ is polynomial. 
Let $\lambda-1$ be the degree of the leading term of the polynomial~$\chi(p)$. 
Consider the projective variety $V(I)$ defined by
\[
V(I) = \{a \in \mathbb{P}^{n-1} \; |\; f(a) = 0 \; \text{for all $f$ in I}\,\}.
\]
The integer $\mu=\lc(\chi(p))(\lambda - 1) !$ corresponds to the degree of the variety $V(I)$, \cite{Hilbert1890}.
If $\chi(p)=0$ then the variety $V(I)$ is empty, in the others cases $V(I)$ is a sub-variety of $\mathbb{P}^{n-1}$ of dimension $\lambda - 1$. 

\subsubsection{Example, {\cite[\textsection 53]{Janet29}}}
Consider the monomial ideal $I$ of $\K[x_1,x_2,x_3]$ generated by $x_1^2, x_1x_2, x_2^2$.
The characteristic function $\chi(p)$ of the ideal $I$ is constant equal to $3$.  The unique point that annihilates the ideal $I$ is $(0,0,1)$ with multiplicity $3$. This result is compatible with the fact that the zeros  of the ideal~$J$ generated by the following polynomials
\[
(x_1-ax_3)(x_1-bx_3), \qquad (x_1-ax_3)(x_2-cx_3), \qquad (x_2-cx_3)(x_2-dx_3). 
\]
consists of the three points 
\[ 
(a,c,1), \qquad (a,d,1), \qquad (b,c,1). 
\]

\subsubsection{The ideal-PDE dictionary}
\label{SS:IdealPDEdictionary}
Let $I$ be a homogeneous ideal of $\K[x_1,\ldots, x_n]$ generated by a set~$F=\{f_1,\ldots, f_k\}$ of polynomials. For a fixed monomial order on $\Mr(x_1,\dots, x_n)$, we set $\Ur=\lm(F)$.
Consider the ring isomorphism $\Phi$ from $\K[x_1,\cdots, x_n]$ to $\K[\frac{\partial}{\partial x_1}, \cdots, \frac{\partial}{\partial x_n}]$ given in \ref{Proposition:IsomorphismPartialX}.
To any polynomial $f$ in $I$ we associate a PDE $\Phi(f)\varphi=0$. In this way, the ideal $I$ defines a PDE system~$(\Sigma(I))$. Let $\lambda$ and $\mu$ be the integers associated to the characteristic function $\chi(p)$ as defined in \ref{GeometricalDescription}.
The maximal number of arguments of the arbitrary analytic functions used to define the initial conditions
\[
\{\, C_\beta \;|\; x^\beta \in \comp{\Ur}\,\}
\]
of the PDE system $(\Sigma(I))$, as defined in (\ref{Equation:BoundaryCondition}), corresponds to $\lambda$, explicitly
\[
\lambda = \max_{v\in \comp{\Ur}} |\cmult_{\div{J}}^{\comp{\Ur}}(v)|,
\]
where $\comp{\Ur}$ denotes the set of complementary monomials of $\Ur$.
Moreover, the number of arbitrary analytic functions with $\lambda$ arguments in the initial conditions $\{\, C_\beta \;|\; x^\beta \in \comp{\Ur}\,\}$ is equal to $\mu$, that is
\[
\mu = \big|\,\{\, v \in \comp{\Ur} \;\text{such that}\; |\cmult_{\div{J}}^{\comp{\Ur}}(v)| = \lambda\,\}\,\big|.
\]

Conversely, let $(\Sigma)$ be a PDE system with one unknown function $\varphi$ of independent variables~$x_1,\ldots, x_n$. Consider the set, denoted by $\mathrm{ldo}(\Sigma)$, made of differential operators associated to the principal derivatives of PDE in~$(\Sigma)$, with respect to Janet's order on derivatives defined in \ref{SubsubsectionDegreeLexicographicOrder}.
By isomorphism~$\Phi$, to any monomial differential operator $\frac{\partial^{\vert \alpha\vert}\;}{\partial x_1^{\alpha_1}\cdots \partial x_n^{\alpha_n}}$ in $\mathrm{ldo}(\Sigma)$, we associate a monomial $x_1^{\alpha_1}\ldots x_n^{\alpha_n}$ \linebreak in~$\Mr(x_1,\ldots,x_n)$. 

Let us denote by $I(\Sigma)$ the ideal of $\K[x_1,\ldots,x_n]$ generated by $\Phi^{-1}(\mathrm{ldo}(\Sigma))$. Note that, by construction the ideal $I(\Sigma)$ is monomial and for any monomial $u$ in $I(\Sigma)$ the derivative $\Phi(u)\varphi$ is a principal derivative of the PDE system $(\Sigma)$ as defined in Section \ref{Subsection:ParametricPrincipalDerivativesSystems}.
In {\cite[\textsection 54]{Janet29}}, M. Janet called \emph{characteristic form} any element of the ideal $I(\Sigma)$.

In this way, M. Janet concluded that the degree of generality of the solutions of a linear PDE system with one unknown function is described by the leading term of the charateristic function of the ideal of characteristic forms defined in \ref{SSS:CharacteristicFunctions}.

\subsubsection{The particular case of first order systems}
Consider a completely integrable first order linear PDE system $(\Sigma)$.
The number $\lambda$, defined in~\ref{SS:IdealPDEdictionary}, that is equal to the maximal number of arguments of the arbitrary functions used to define the initial conditions of the system $(\Sigma)$, is also equal in this case to the cardinal of the set $\comp{\Ur}$ of complementary monomials of the set of monomials $\Ur=\Phi^{-1}(\mathrm{ldo}(\Sigma))$. 

\subsection{Involutive systems}
\label{Section:InvolutiveSystem}

In this subsection, we recall the algebraic formulation of involutive systems as introduced by M. Janet. This formulation first appeared in its work in \cite{Janet22a} and \cite{Janet22b}. But notice that this notion comes from the work of \'E. Cartan in~\cite{Cartan04}.

\subsubsection{Characters and derived systems}

Let $I$ be a proper ideal of $\K[x_1,\ldots, x_n]$ generated by homogeneous polynomials.
M. Janet introduced the \emph{characters} of the homogeneous component $I_p$ as the non-negative integers $\sigma_1,\sigma_2,\ldots, \sigma_n$ defined inductively by the following formula
\[ 
\dim \left(I_p + \left(\sum_{i=1}^h \K[x_1,\cdots, x_n]_{p-1}x_i\right)\right)=\dim (I_p) +\sigma_1+\ldots +\sigma_h, \qquad 1\leq h\leq n. 
\]
Note that the sum $\sigma_1+\sigma_2+\ldots +\sigma_n$ corresponds to the codimension of $I_p$ in $\K[x_1,\cdots, x_n]_p$.

Given a positive integer $\lambda$, we set
\[
J_{p+\lambda}=\K[x_1,\ldots, x_n]_\lambda I_p. 
\]
We define the non-negative integers $\sigma_1^{(\lambda)},\sigma_2^{(\lambda)},\ldots, \sigma_n^{(\lambda)}$ by the relations
\[
\dim \left(J_{p+\lambda} + \left(\sum_{i=1}^h \K[x_1,\cdots, x_n]_{p+\lambda-1}x_i\right)\right)=\dim (J_{p+\lambda})+\sigma_1^{(\lambda)}+\ldots +\sigma_h^{(\lambda)} \qquad 1\leq h\leq n.
\]
For $\lambda=1$, M. Janet called $J_{p+1}$ the \emph{derived system} of $I_p$.
Let us mention some properties on these numbers proved by M. Janet. 
\begin{lemma} 
\label{Lemma:Involution}
We set $\sigma_h'=\sigma_h^{(1)}$ and $\sigma_h''=\sigma_h^{(2)}$ for $1\leq h\leq n$.
\begin{enumerate}[{\bf i)}]
\item $\sigma_1'+\sigma_2'+\ldots +\sigma_n' \leq \sigma_1+2\sigma_2+\ldots +n \sigma_n$.
\item If $ \sigma_1'+\sigma_2'+\ldots +\sigma_n' = \sigma_1+2\sigma_2+\ldots +n \sigma_n$, the two following relations hold:
\begin{enumerate}[{\bf a)}]
\item $\sigma_1''+\sigma_2''+\ldots +\sigma_n'' = \sigma_1'+2\sigma_2'+\ldots +n \sigma_n'$.
\item $\sigma_h'=\sigma_h+\sigma_{h+1}+\ldots +\sigma_n$.
\end{enumerate}
\end{enumerate}
\end{lemma}
We refer the reader to~\cite{Janet29} for a proof of the relations of Lemma~\ref{Lemma:Involution}.

\subsubsection{Involutive systems} 

The homogenous component $I_p$ is said to be in \emph{involution} when the following equality holds:
\[
\sigma_1'+\sigma_2'+\ldots +\sigma_n' = \sigma_1+2\sigma_2+\ldots +n \sigma_n.
\]

Following properties {\bf ii)}-{\bf a)} of Lemma~\ref{Lemma:Involution}, if the component $I_p$ is in involution, then the component~$I_{p+k}$ is in involution for all $k\geq 0$.

\begin{proposition}[{\cite[\textsection 56 \& \textsection 57]{Janet29}}]
The characters of a homogeneous component $I_p$ satisfy the two following properties 
\begin{enumerate}[{\bf i)}]
\item $\sigma_1\geq \sigma_2\geq \ldots \geq \sigma_n$.
\item if $I_p\neq \{0\}$, then $\sigma_n=0$.
\end{enumerate}
\end{proposition}

\subsubsection{Polynomiality of characteristic function}
Suppose that the homogeneous component $I_p$  is in involution. We show that the characteristic function $\chi(P)$ defined in~\ref{SSS:CharacteristicFunctions} is polynomial for $P\geq p$.
Using Lemma~\ref{Lemma:Involution}, we show by induction that for any $1\leq h< n$ and any positive integer $\lambda$, we have the following relation:
\[ 
\sigma_h^{(\lambda)}=\sum_{k=0}^{n-h-1} \binom{\lambda+k-1}{k}\sigma_{h+k}. 
\]
The codimension of $I_{p+\lambda}$ in $\K[x_1,\cdots, x_n]_{p+\lambda}$ is  given by
\begin{align*} 
\sum_{h=1}^{n-1}\sigma_h^{(\lambda)}=
&\sum_{h=1}^{n-1} \sum_{k=0}^{n-h-1}\binom{\lambda+k-1}{k}\sigma_{h+k} 
=
\sum_{i=1}^{n-1} \left(\sum_{k=0}^{i-1}\binom{\lambda+k-1}{k}\right) \sigma_i \\
=
&\sum_{i=1}^{n-1} \left(\sum_{k=0}^{i-1}\binom{P-p+k-1}{k}\right) \sigma_i 
=
\sum_{i=1}^{n-1} \binom{P-p+i-1}{i-1} \sigma_i. 
\end{align*}
This proves the polynomiality of the characteristic function of the ideal $I$ for sufficiently big $p$.

\subsection{Conclusive remarks}

Recall that the so-called Cartan-K\"{a}hler theory is about the Pfaffian systems on a differentiable (or analytic) manifold and its aim is to judge whether a given system is prolongeable to a completely integrable system or an incompatible system. Their method relies on a geometrical argument, which is to construct integral submanifolds of the system inductively. Here, a step of the induction is to find an integral submanifold of dimension $i+1$ containing the integral submanifold of dimension $i$, and their theory does not allow one to see whether such step can be achieved or not.

Janet's method is, even if it works only locally, completely algebraic and algorithmic so that it partially completes the parts where one cannot treat with
Cartan-K\"{a}hler theory.

By these works, there are two seemingly different notions of involutivity; the one by G. Frobenius, G. Darboux and \'E. Cartan and the other by M. Janet. 
The fact is that at each step of the induction in the Cartan-K\"{a}hler theory, one has to study a system of PDE.
Its system is called in \emph{involution} (cf. compare those in Sections \ref{Subsubsection:InvolutionCartan} with \ref{Section:InvolutiveSystem}) if it can be written in a canonical system, as defined in~\ref{Subsubsection:CanonicalSystems}, if necessary after a change of coordinates. Following the algebraic definition of involutivity by M. Janet, several involutive methods were developed for polynomial and differential systems, \cite{Thomas37,Pommaret78}. In these approaches, a differential system is involutive when its non-multiplicative derivatives are consequences of multiplicative derivatives.
In~\cite{Gerdt97,GerdtBlinkovYuri98}, V. P. Gerdt gave an algebraic charaterization of the involutivity for polynomial systems. The Gerdt's approach is developed in the next section.

\section{Polynomial involutive bases}
\label{Section:PolynomialInvolutiveBases}

In this section, we present the algebraic definition of involutivity for polynomial systems given by V.~P.~Gerdt in \cite{Gerdt97,GerdtBlinkovYuri98}. In particular, we relate the notion of involutive basis for a polynomial ideal to the notion of Gr\"obner basis.

\subsection{Involutive reduction on polynomials}

\subsubsection{Involutive basis}
Recall that a \emph{monomial ideal} $I$ of $\K[x_1,\ldots,x_n]$ is an ideal generated by monomials.
An \emph{involutive basis} of the ideal $I$ with respect to an involutive division division $\div{I}$ is an involutive set of monomials $\Ur$ that generates $I$.
By Dickson Lemma,~\cite{Dickson13}, any monomial ideal $I$ admits a finite set of generators. When the involutive division $\div{I}$ is noetherian as defined in~\ref{SSS:Noetherianity}, this generating set admits a finite $\div{I}$-completion that forms an involutive basis of the ideal~$I$. As a consequence, we deduce the following result.

\begin{proposition}
Let $\div{I}$ be a noetherian involutive division on $\Mr(x_1,\ldots,x_n)$. Any monomial ideal of~$\K[x_1,\ldots,x_n]$ admits an $\div{I}$-involutive basis.
\end{proposition}

The objective of this section is to show how to extend this result to polynomial ideals with respect to a monomial order. In the remainder of this subsection we assume that a monomial order $\preccurlyeq$ is fixed on~$\Mr(x_1,\ldots,x_n)$.

\subsubsection{Multiplicative variables for a polynomial}
Let $\div{I}$ be an involutive division on $\Mr(x_1,\ldots,x_n)$.
Let~$F$ be a set of polynomials of $\K[x_1,\ldots,x_n]$ and~$f$ be a polynomial in $F$. We define the set of \emph{$\div{I}$-multiplicative} (resp. \emph{$\div{I}$-non-multiplicative}) \emph{variables} of the polynomial $f$ with respect to $F$  and the monomial order $\preccurlyeq$ by setting
\[
\mult_{\div{I},\preccurlyeq}^F(f) = \mult_\div{I}^{\lm_{\preccurlyeq}(F)}(\lm_{\preccurlyeq}(f)),
\quad 
(\,\text{resp.}\;\; \nonmult_{\div{I},\preccurlyeq}^F(f) = \nonmult_\div{I}^{\lm_{\preccurlyeq}(F)}(\lm_{\preccurlyeq}(f))\,).
\]
Note that the $\div{I}$-multiplicative variables depend on the monomial order $\preccurlyeq$ used to determine leading monomials of polynomials of $F$.

\subsubsection{Polynomial reduction}
The polynomial division can be describe as a rewriting operation as follows.
Given polynomials $f$ and $g$ in $\K[x_1,\ldots,x_n]$, we say that $f$ is \emph{reducible modulo $g$ with respect to~$\preccurlyeq$}, if there is a term $\lambda u$ in $f$ whose monomial $u$ is divisible by $\lm_\preccurlyeq(g)$ for the usual monomial division. In that case, we denote such a reduction by $f \ofl{g_{\preccurlyeq}} h$, where
\[ 
h = f - \frac{\lambda u}{\lt_\preccurlyeq(g)}g.
\]
For a set $G$ of polynomials of $\K[x_1,\ldots,x_n]$, we define a rewriting system corresponding to the division modulo $G$ by considering the relation reduction $\ofl{G_{\preccurlyeq}}$ defined by 
\[
\ofl{G_{\preccurlyeq}} = \bigcup_{g\in G} \ofl{g_{\preccurlyeq}}.
\]
We will denote by $\overset{\displaystyle G_{\preccurlyeq}}{\fll^\ast}$ the reflexive and transitive closure of the relation $\ofl{G_{\preccurlyeq}}$.

\subsubsection{Involutive reduction}
In a same way, we define a notion of reduction with respect to an involutive division~$\div{I}$ on $\Mr(x_1,\ldots,x_n)$. Let $g$ be a polynomial in $\K[x_1,\ldots,x_n]$.
A polynomial $f$ in $\K[x_1,\ldots,x_n]$ is said to be \emph{$\div{I}$-reducible modulo $g$} with respect to the monomial order $\preccurlyeq$, if there is a term $\lambda u$ of $f$, with~$\lambda\in\K-\{0\}$ and $u\in \Mr(x_1,\ldots,x_n)$, such that 
\[
u=\lm_\preccurlyeq(g)v
\quad\text{and}\quad
v\in \Mr(\mult_\div{I}^{\lm_\preccurlyeq(G)}(g)).
\] 
Such a $\div{I}$-reduction is denoted by $f\ofl{g_\preccurlyeq}_\div{I} \;h$, where
\[
h = f - \frac{\lambda}{\lc_\preccurlyeq(g)}gv = f - \frac{\lambda u}{\lt_\preccurlyeq(g)}g.
\]

\subsubsection{Involutive normal forms}
Let $G$ be a set of polynomials of $\K[x_1,\ldots,x_n]$.
A polynomial $f$ is said to be \emph{$\div{I}$-reducible modulo $G$} with respect to the monomial order $\preccurlyeq$, if there exists a polynomial $g$ in $G$ such that $f$ is $\div{I}$-reducible modulo $g$. We will denote by $\ofl{G_\preccurlyeq}_\div{I}$ this reduction relation defined by 
\[
\ofl{G_\preccurlyeq}_\div{I} = \bigcup_{g\in G} \ofl{g_\preccurlyeq}_\div{I}.
\]
The polynomial $f$ is said to be in \emph{$\div{I}$-irreducible modulo $G$} if it is not $\div{I}$-reducible modulo~$G$.  A \emph{$\div{I}$-normal form of a polynomial $f$} is a $\div{I}$-irreducible polynomial $h$ such that there is a sequence of reductions from $f$ to $h$:
\[
f \ofl{G_\preccurlyeq}_\div{I}\; f_1 \ofl{G_\preccurlyeq}_\div{I} \; f_2 \ofl{G_\preccurlyeq}_\div{I} \; \ldots \; \ofl{G_\preccurlyeq}_\div{I} \; h,
\]

The procedure ${\bf InvReduction}_\div{I,\preccurlyeq}(f,G)$ computes a normal form of $f$ modulo $G$ with respect to the division $\div{I}$. The proofs of its correctness and termination can be achieved as in the case of the division procedure for the classical polynomial division, see for instance {\cite[Proposition 5.22]{BeckerWeispfenning93}}.

\begin{algorithm}
\SetAlgoLined
\KwIn{a polynomial $f$ in $\K[x_1,\ldots,x_n]$ and a finite subset $G$ of $\K[x_1,\ldots,x_n]$.}

\BlankLine

\Begin{

$h \leftarrow f$

\BlankLine

\While{exist $g$ in $G$ and a term $t$ of $h$ such that $\lm_\preccurlyeq(g)|_\div{I}^{\lm_\preccurlyeq(G)} \frac{t}{\lc_\preccurlyeq(t)}$}{

\BlankLine

{\bf choose} such a $g$ 

$h \leftarrow h - \frac{t}{\lt_\preccurlyeq(g)}g$ 

}}

\BlankLine

\KwOut{$h$ a $\div{I}$-normal form of the polynomial $f$ with respect to the monomial order $\preccurlyeq$}

\caption{${\bf InvReduction}_{\div{I},\preccurlyeq}(f,G)$}
\end{algorithm}

\subsubsection{Remarks}
Note that the involutive normal form of a polynomial $f$ is not unique in general, it depends on the order in which the reductions are applied. Suppose that, for each polynomial $f$ we have a $\div{I}$-normal form with respect to the monomial order $\preccurlyeq$, that is denoted by $\normf_{\div{I},\preccurlyeq}^G(f)$.
Denote by~$\normf_{\preccurlyeq}^G(f)$ a normal form of a polynomial $f$ obtained by the classical division procedure. In general, the equality~$\normf_\preccurlyeq^G(f)=\normf^G_{\div{I},\preccurlyeq}(f)$ does not hold. Indeed, suppose that $G=\{x_1,x_2\}$ and consider the Thomas division $\div{T}$ defined in \ref{SSS:ThomasDivision}. We have $\normf_\preccurlyeq^G(x_1x_2)=0$, while~$\normf^G_{\div{T},\preccurlyeq}(x_1x_2)=x_1x_2$ because the monomial~$x_1x_2$ is a $\div{T}$-irreducible modulo $G$.

\subsubsection{Autoreduction}
\label{Subsubsection:PolynomialAutoreduction}
Recall from~\ref{Subsubsection:MonomialAutoreduction} that a set of monomials $\Ur$ is $\div{I}$-autoreduced with respect to an involutive division $\div{I}$ if it does not contain a monomial $\div{I}$-divisible by another monomial of $\Ur$. In that case, any monomial in $\Mr(x_1,\ldots,x_n)$ admits at most one $\div{I}$-involutive divisor in $\Ur$. 

A set $G$ of polynomials of $\K[x_1,\ldots,x_n]$ is said to be \emph{$\div{I}$-autoreduced} with respect to the monomial order $\preccurlyeq$, if it satisfies the two following conditions:
\begin{enumerate}[{\bf i)}]
\item (\emph{left $\div{I}$-autoreducibility}) the set of leading monomials $\lm_\preccurlyeq(G)$ is $\div{I}$-autoreduced,
\item (\emph{right $\div{I}$-autoreducibility}) for any $g$ in $G$, there is no term $\lambda u \neq \lt_\preccurlyeq(g)$ of $g$, with $\lambda\neq 0$ \linebreak and~$u\in \cone_\div{I}(\lm_\preccurlyeq(G))$.
\end{enumerate}

Note that the condition {\bf i)}, (resp. {\bf ii)}) corresponds to the left-reducibility (resp. right-reducibility) property given in~\ref{Subsubsection:CanonicalSystems}.
Any finite set $G$ of polynomials of $\K[x_1,\ldots,x_n]$ can be transformed into a finite $\div{I}$-autoreduced set that generates the same ideal by Procedure~\ref{AutoreductionProcedure2}. The proofs of correctness and termination are immediate consequences of the property of involutive division.

\begin{algorithm}
\SetAlgoLined
\KwIn{$G$ a finite subset of $\K[x_1,\ldots,x_n]$.}

\BlankLine

\Begin{

$H \leftarrow G$

$H' \leftarrow \emptyset$

\BlankLine

\While{exists $h\in H$ and $g\in H\setminus\{h\}$ such that $h$ is $\div{I}$-reducible modulo $g$ with respect to $\preccurlyeq$}{

\BlankLine

{\bf choose} such a $h$ 

$H' \leftarrow H\setminus\{h\}$ 

$h' \leftarrow \normf_{\div{I},\preccurlyeq}^{H'}(h)$

\If{$h'=0$}{$H\leftarrow H'$}

\Else{$H\leftarrow H'\cup\{h'\}$}

}}

\BlankLine

\KwOut{$H$ an $\div{I}$-autoreduced set generating the same ideal as $G$ does.}

\caption{${\bf Autoreduce}_{\div{I},\preccurlyeq}(G)$}
\label{AutoreductionProcedure2}
\end{algorithm}

\begin{proposition}[{\cite[Theorem 5.4]{GerdtBlinkovYuri98}}]
\label{Proposition:decompositionAutoreduite}
Let $G$ be an $\div{I}$-autoreduced set of polynomials \linebreak of $\K[x_1,\ldots,x_n]$ and $f$ be a polynomial in $\K[x_1,\ldots,x_n]$. Then $\normf_{\div{I},\preccurlyeq}^G(f)=0$ if and only if the polynomial~$f$ can be written in the form
\[
f = \sum_{i,j} \beta_{i,j}g_iv_{i,j},
\]
where $g_i\in G$, $\beta_{i,j}\in \K$ and $v_{i,j}\in \Mr(\mult_\div{I}^{\lm_\preccurlyeq(G)}(\lm_\preccurlyeq(g_i)))$, with $\lm_\preccurlyeq(v_{i,j})\neq \lm_\preccurlyeq(v_{i,k})$ if $j\neq k$.
\end{proposition}
\begin{proof}
Suppose that $\normf_{\div{I},\preccurlyeq}^G(f)=0$, then there exists a sequence of involutive reductions modulo $G$:
\[
f =f_0\ofl{g_1}_\div{I}\; f_1 \ofl{g_2}_\div{I}\; f_2 \ofl{g_3}_\div{I} \; \ldots \; \ofl{g_{k-1}}_\div{I} \; f_k=0,
\]
terminating on $0$. For any $1\leq i \leq k$, we have $f_{i} = f_{i-1} - \frac{\lambda_{i,j}}{\lc_\preccurlyeq(g_{i})}g_{i}v_{i,j}$, with $v_{i,j}\in \Mr(\mult_\div{I}^{\lm_\preccurlyeq(G)}(\lm_\preccurlyeq(g_i)))$. This show the equality.

Conversely, suppose that $f$ can be written in the given form. Then the leading monomial $\lm_\preccurlyeq(f)$ admits an involutive $\div{I}$-divisor in $\lm_\preccurlyeq(G)$. Indeed, the leading monomial of the decomposition of $f$ has the following form:
\[
\lm_\preccurlyeq\left(\sum_{i,j} g_iv_{i,j}\right)=\lm_\preccurlyeq(g_{i_0})v_{i_0,j_0}.
\]
The monomial $\lm_\preccurlyeq(g_{i_0})$ is an involutive divisor of $\lm_\preccurlyeq(f)$ and by autoreduction hypothesis, such a divisor is unique. Hence the monomial $\lm_\preccurlyeq(g_{i_0})v_{i_0,j_0}$ does not divide other monomial of the form $\lm_\preccurlyeq(g_i)v_{i,j}$. We apply the reduction $g_{i_0}v_{i_0,j_0} \ofl{{g_{i_0}}_\preccurlyeq}_{\div{I}}\; 0$ on the decomposition. In this way, we define a sequence of reductions ending on $0$. This proves that $\normf_{\div{I},\preccurlyeq}^G(f)=0$.
\end{proof}

\subsubsection{Unicity and additivity of involutive normal forms}
\label{SSS:UnicityAdditivityNormalForms}
From decomposition~\ref{Proposition:decompositionAutoreduite}, we deduce two important properties on involutive normal forms. Let $G$ be a $\div{I}$-autoreduced set of polynomials of~$\K[x_1,\ldots,x_n]$ and $f$ be a polynomial. Suppose that $h_1=\normf_{\div{I},\preccurlyeq}^G(f)$ and $h_2=\normf_{\div{I},\preccurlyeq}^G(f)$ are two involutive normal forms of $f$. From the involutive reduction procedure that computes this two normal forms, we deduces two decompositions
\[
h_1 = f - \sum_{i,j} \beta_{i,j}g_i v_{i,j},
\qquad
h_2 = f - \sum_{i,j} \beta'_{i,j}g_i v'_{i,j}.
\]
As a consequence, $h_1-h_2$ admits a decomposition as in Proposition~\ref{Proposition:decompositionAutoreduite}, hence $\normf_{\div{I},\preccurlyeq}^G(h_1-h_2)=0$. 
The polynomial $h_1-h_2$ being in normal form, we deduce that $h_1=h_2$. This shows the unicity of the involutive normal form modulo an autoreduced set of polynomials.

In a same manner we prove the following additivity formula for any polynomial $f$ and $f'$:
\[
\normf_{\div{I},\preccurlyeq}^G(f+f')=\normf_{\div{I},\preccurlyeq}^G(f)+\normf_{\div{I},\preccurlyeq}^G(f').
\]

\subsection{Involutive bases}

We fix a monomial order $\preccurlyeq$ on~$\Mr(x_1,\ldots,x_n)$.

\subsubsection{Involutive bases}
Let $I$ be an ideal of $\K[x_1,\ldots,x_n]$. A set $G$ of polynomials of $\K[x_1,\ldots,x_n]$ is an \emph{$\div{I}$-involutive basis} of the ideal $I$ with respect the monomial order $\preccurlyeq$, if $G$ is $\div{I}$-autoreduced and satisfies the following property: 
\[
\forall g \in G,\;
\forall u\in \Mr(x_1,\ldots,x_n),\quad
\normf_{\div{I},\preccurlyeq}^G(gu)=0.
\]
In other words, for any polynomial $g$ in $G$ and monomial $u$ in $\Mr(x_1,\ldots,x_n)$, there is a sequence of involutive reductions:
\[
gu\ofl{{g_1}_\preccurlyeq}_\div{I} \; f_1 \ofl{{g_2}_\preccurlyeq}_\div{I} \; f_2 \ofl{{g_3}_\preccurlyeq}_\div{I} \; \ldots \; \ofl{{g_{k-1}}_\preccurlyeq}_\div{I} \; 0,
\]
with $g_i$ in $G$.
In particular, we recover the notion of involutive sets of monomials given in  \ref{SSS:InvolutiveSet}. Indeed, if $G$ is an $\div{I}$-involutive basis, then $\lm_\preccurlyeq(G)$ is an $\div{I}$-involutive set of monomials of~$\Mr(x_1,\ldots,x_n)$.

\begin{proposition}
\label{Proposition:reducibilityInvolutiveReducibility}
Let $\div{I}$ be an involutive division on $\K[x_1,\ldots,x_n]$ and $G$ be a $\div{J}$-involutive subset of~$\K[x_1,\ldots,x_n]$.
A polynomial of $\K[x_1,\ldots,x_n]$ is reducible with respect to $G$ if and only if it is $\div{I}$-reducible modulo $G$.
\end{proposition}
\begin{proof}
Let $f$ be a polynomial of $\K[x_1,\ldots,x_n]$. 
By definition of involutive reduction, if $f$ is $\div{I}$-reducible modulo $G$, then it is reducible for the relation $\ofl{G_\preccurlyeq}$. Conversely, suppose that $f$ is reducible by a polynomial $g$ in $G$. That is there exists a term $\lambda u$ in $f$, where $\lambda$ is a nonzero scalar and $u$ is a monomial of~$\Mr(x_1,\ldots,x_n)$ such that $u=\lm_\preccurlyeq(g)v$, where $v\in \Mr(x_1,\ldots,x_n)$.
The set $G$ being involutive, we have~$\normf_{\div{I},\preccurlyeq}^G(gv)=0$. Following Proposition~\ref{Proposition:decompositionAutoreduite}, the polynomial $gv$ can written in the form:
\[
gv = \sum_{i,j} \beta_{i,j}g_i v_{i,j},
\]
where $g_i\in G$, $\beta_{i,j}\in \K$  and $v_{i,j}\in \Mr(\mult_\div{I}^{\lm_\preccurlyeq(G)}(\lm_\preccurlyeq(g_i)))$. In particular, this shows that the monomial $u$ admits an involutive divisor in $G$.
\end{proof}

\subsubsection{Unicity of normal forms}
\label{SSS:UnicityNormalForms}

Let us mention an important consequence of Proposition~\ref{Proposition:reducibilityInvolutiveReducibility} given in~{\cite[Theorem 7.1]{GerdtBlinkovYuri98}}. Let $G$ be a $\div{J}$-involutive subset of $\K[x_1,\ldots,x_n]$, for any reduction procedure that computes a normal form $\normf_\preccurlyeq^G(f)$ of a polynomial $f$ in $\K[x_1,\ldots,x_n]$ and any involutive reduction procedure that computes an involutive normal form $\normf_{\div{I},\preccurlyeq}^G(f)$, as a consequence of unicity of the involutive normal form and  Proposition~\ref{Proposition:reducibilityInvolutiveReducibility}, we have
\[
\normf_\preccurlyeq^G(f) = \normf_{\div{I},\preccurlyeq}^G(f).
\]

\subsubsection{Example}
We set $\Ur=\{x_1,x_2\}$. We consider the deglex order induced by $x_2>x_1$ and the Thomas division $\div{T}$. The monomial $x_1x_2$ is $\div{T}$-irreducible modulo $\Ur$. Hence, it does not admits zero as $\div{T}$-normal form and the set $\Ur$ cannot be an $\div{T}$-involutive basis of the ideal generated by $\Ur$. In turn the set~$\{x_1,x_2,x_1x_2\}$ is a $\div{T}$-involutive basis of the ideal generated by $\Ur$.

We now consider the Janet division $\div{J}$. We have $\deg_2(\Ur)=1$, $[0]=\{x_1\}$ and $[1]=\{x_2\}$. The $\div{J}$-multiplicative variables are given by the following table:
\begin{center}
\begin{tabular}{c|cc}
$u$ & $\mult_{\div{J}}^\Ur(u)$\\
\hline
$x_1$ & $x_1$ &\\
$x_2$ & $x_1$ & $x_2$\\
\end{tabular}
\end{center}
It follows that the monomial $x_1x_2$ is not $\div{J}$-reducible by $x_1$ modulo $\Ur$. However, it is $\div{J}$-reducible by $x_2$. Hence the set $\Ur$ form a $\div{J}$-involutive basis.

\bigskip

As an immediate consequence of involutive bases, the involutive reduction procedure provides a decision method of the ideal membership problem, as stated by the following result.

\begin{proposition}[{\cite[Corollary 6.4]{GerdtBlinkovYuri98}}]
\label{Proposition:formeNormaleNulle}
Let $I$ be an ideal of $\K[x_1,\ldots,x_n]$, and $G$ be an \linebreak $\div{I}$-involutive basis of $I$ with respect to a monomial order $\preccurlyeq$. For any polynomial $f$ of $\K[x_1,\ldots,x_n]$, we have $f \in I$ if and only if $\normf_{\div{I},\preccurlyeq}^G(f)=0$.
\end{proposition}
\begin{proof}
If $\normf_{\div{I},\preccurlyeq}^G(f)=0$, then the polynomial $f$ can be written in the form \ref{Proposition:decompositionAutoreduite}. This shows that $f$ belongs to the ideal $I$.
Conversely, suppose that $f$ belongs to $I$, then it can be decomposed in the form
\[
f = \sum_i h_ig_i,
\]
where $h_i=\sum_j\lambda_{i,j}u_{i,j}\in \K[x_1,\ldots,x_n]$. The set $G$ being $\div{I}$-involutive, we have $\normf_{\div{I},\preccurlyeq}^G(u_{i,j}g_i)=0$, for any monomials $u_{i,j}$ and $g_i$ in $G$. By linearity of the operator $\normf_{\div{I},\preccurlyeq}^G(-)$, we deduce that $\normf_{\div{I},\preccurlyeq}^G(f)=0$. 
\end{proof}

\subsubsection{Local involutivity}
V. P. Gerdt and Y. A. Blinkov introduced in~\cite{GerdtBlinkovYuri98} the notion of local involutivity for a set of polynomials. A set $G$ of polynomials of $\K[x_1,\ldots,x_n]$ is said to be \emph{locally involutive} if the following condition holds
\[
\forall g \in G,\;
\forall x\in \nonmult_\div{I}^{\lm_\preccurlyeq(G)}(\lm_\preccurlyeq(g)),\quad
\normf_{\div{I},\preccurlyeq}^G(gx)=0.
\]
For a continuous involutive division $\div{I}$, they prove that a $\div{I}$-autoreduced set of polynomials is involutive if and only if it is locally involutive, {\cite[Theorem 6.5]{GerdtBlinkovYuri98}}. This criterion of local involutivity is essential for computing the completion of a set of polynomials into an involutive basis. Note that this result is analogous to the critical pair lemma in rewriting theory stating that a rewriting system is locally confluent if and only if all its critical pairs are confluent, see e.g. {\cite[Sect. 3.1]{GuiraudMalbos18}}. Together with the Newman Lemma stating that for terminating rewriting, local confluence and confluence are equivalent properties, this gives a constructive method to prove confluence in a terminating rewriting system by analyzing the confluence of critical pairs.

\subsubsection{Completion procedure}
For a given monomial order $\preccurlyeq$ on $\Mr(x_1,\ldots,x_n)$ and a continuous and constructive involutive division $\div{I}$, as defined in {\cite[Definition 4.12]{GerdtBlinkovYuri98}},
the Procedure~\ref{P:InvolutiveCompletionBasis} computes an $\div{I}$-involutive basis of an ideal from a set of generators of the ideal.  We refer the reader to {\cite[Sect. 8]{GerdtBlinkovYuri98}} or {\cite[Sect. 4.4]{Evans06}} for correctness of this procedure and conditions for its termination. This procedure is in the same vein as the completion procedure for rewriting systems by Knuth-Bendix, \cite{KnuthBendix70}, and completion procedure for commutative polynomials by Buchberger, \cite{Buchberger65}.

\begin{algorithm}
\SetAlgoLined
\KwIn{$F$ a finite set of polynomials in $\K[x_1,\ldots,x_n]$.}

\BlankLine

\Begin{

$F' \leftarrow {\bf Autoreduce}_{\div{I},\preccurlyeq}(F)$

$G \leftarrow \emptyset$

\BlankLine

\While{$G=\emptyset$}{

$\Pr \leftarrow \{fx \;|\; f\in F', x\in \nonmult_{\div{I},\preccurlyeq}^{F'}(f)\}$

$p' \leftarrow 0$

\While{$\Pr \neq \emptyset$ and $p'=0$}{

\BlankLine

{\bf choose} $p$ in $\Pr$ such that $\lm_\preccurlyeq(p)$ is minimal with respect to $\preccurlyeq$. 

$\Pr \leftarrow \Pr\setminus\{p\}$ 

$p' \leftarrow {\bf InvReduction}_{\div{I},\preccurlyeq}(p,F')$

}

\If{$p'\neq 0$}{

$F' \leftarrow {\bf Autoreduce}_{\div{I},\preccurlyeq}(F'\cup\{p'\})$

}

\Else{

$G\leftarrow F'$

}}}

\BlankLine

\KwOut{$G$ an $\div{I}$-involutive basis of the ideal generated by $F$ with respect to the monomial order $\preccurlyeq$.}
\caption{${\bf InvolutiveCompletionBasis}_{\div{I},\preccurlyeq}(F)$}
\label{P:InvolutiveCompletionBasis}
\end{algorithm}

\subsubsection{Example: computation of an involutive basis}
Let $I$ be the ideal of $\Q[x_1,x_2]$ generated by the set $F=\{f_1,f_2\}$, where the polynomial $f_1$ and $f_2$ are defined by
\begin{align*}
f_1 &= x_2^2-2x_1x_2+1,\\
f_2 &= x_1x_2 - 3x_1^2 -1.
\end{align*}
We compute an involutive basis of the ideal $I$ with respect to the Janet division $\div{J}$ and the deglex order induced by $x_2>x_1$. 
We have $\lm(f_1)=x_2^2$ and $\lm(f_2)=x_1x_2$, hence the following $\div{J}$-reductions
\[
x_2^2 \ofl{f_1}_{\div{J}} \; 2x_1x_2 - 1,
\qquad
x_1x_2 \ofl{f_2}_{\div{J}} \; 3x_1^2 + 1.
\]
The polynomial $f_1$ is $\div{J}$-reducible by $f_2$, we have
\[
f_1 \ofl{f_2}_\div{J} \; x_2^2 - 2(3x_1^2+1) + 1 = x_2^2 - 6x_1^2-1.
\]
Thus, we set $f_3=x_2^2-6x_1^2-1$ and we consider the reduction 
\[
x_2^2 \ofl{f_3}_{\div{J}} \; 6x_1^2+1.
\]
The set $F'=\{f_2,f_3\}$ is $\div{J}$-autoreduced and generates the ideal $I$. 

Let us compute the multiplicative variables of the polynomials $f_2$ and $f_3$. 
We have $\deg_2(F')=\deg_2(\{x_2^2,x_1x_2\})=2$, $[1]=\{x_1x_2\}$ and $[2]=\{x_2^2\}$.  Hence the $\div{J}$-multiplicative variables are given by the following table:
\begin{center}
\begin{tabular}{c|c|cc}
$f$ & $\lm(f)$ & $\mult_{\div{J}}^{F'}(f)$\\
\hline
$f_2$ & $x_1x_2$ & $x_1$ &\\
$f_3$ & $x_2^2$ & $x_1$ & $x_2$\\
\end{tabular}
\end{center}

The polynomial $f_2x_2 = x_1x_2^2-3x_1^2x_2-x_2$ is the only non-multiplicative prolongation to consider. This prolongation can be reduced as follows
\[
f_2x_2 \ofl{f_3}_{\div{J}} \; 6x_1^3 + x_1 -3x_1^2x_2 - x_2 \ofl{f_2}_{\div{J}} \; -3x_1^3 - 2x_1 - x_2.
\]
We set $f_4=-3x_1^3 - 2x_1 - x_2$, whose associated reduction is
\[
x_1^3 \ofl{f_4}_{\div{J}} \; -\frac{2}{3}x_1 - \frac{1}{3}x_2,
\]
and we set $F'=\{f_2,f_3,f_4\}$. 
We have $\deg_2(F')=2$, $[0]=\{x_1^3\}$, $[1]=\{x_1x_2\}$ and $[2]=\{x_2^2\}$.  Hence the $\div{J}$-multiplicative variables are given by the following table:
\begin{center}
\begin{tabular}{c|c|cc}
$f$ & $\lm(f)$ & $\mult_{\div{J}}^{F'}(f)$\\
\hline
$f_2$ & $x_1x_2$ & $x_1$ &\\
$f_3$ & $x_2^2$ & $x_1$ & $x_2$\\
$f_4$ & $x_1^3$ & $x_1$ &\\
\end{tabular}
\end{center}
There are two non-multiplicative prolongations to consider:
\[
f_2x_2 = x_1x_2^2 - 3x_1^2x_2 - x_2,\qquad
f_4x_2 = -3x_1^3x_2 - 2x_1x_2 - x_2^2.
\]
We have $\lm(f_2x_2)=x_1x_2^2 < \lm(f_4x_2)=x_1^3x_2$. Hence the prolongation $f_2x_2$ must first be examined. We have the following reductions:
\[
f_2x_2 \ofl{f_3}_{\div{J}} \;6x_1^3 + x_1 - 3x_1^2x_2 - x_2 \ofl{f_2}_{\div{J}} \; -3x_1^3 -2x_1 - x_2 \ofl{f_4}_{\div{J}} \; 0.
\]
Hence, there is no polynomial to add.
The other non-multiplicative prolongation is $f_4x_2$, that can be reduced to an $\div{J}$-irreducible polynomial as follows:
\[
f_4x_2 
\ofl{f_2}_{\div{J}} \;
-3x_1^3x_2 - 6x_1^2 - x_2^2-2
\ofl{f_3}_{\div{J}} \;
-3x_1^3x_2-12x_1^2-3
\]
\[
\hspace{6cm}
\ofl{f_2}_{\div{J}} \; -9x_1^4-15x_1^2-3
\ofl{f_4}_{\div{J}} \;
3x_1x_2 - 9x_1^2 - 3
\ofl{f_2}_{\div{J}} \;
0.
\hspace{3cm}
\]
All the non-multiplicative prolongations are $\div{J}$-reducible to $0$, it follows that the set $F'$ is a Janet basis of the ideal $I$.

\subsection{Involutive bases and Gr\"{o}bner bases}

In this subsection, we recall the notion of Gr\"{o}bner basis and we show that any involutive basis is a Gr\"{o}bner basis. We fix a monomial order $\preccurlyeq$ on~$\Mr(x_1,\ldots,x_n)$.

\subsubsection{Gr\"{o}bner bases}
A subset $G$ of $\K[x_1,\ldots,x_n]$ is a \emph{Gr\"{o}bner basis} with respect to the monomial order~$\preccurlyeq$ if it is finite and satisfies one of the following equivalent conditions 
\begin{enumerate}[{\bf i)}]
\item $\ofl{G_\preccurlyeq}$ is Church-Rosser,
\item $\ofl{G_\preccurlyeq}$ is confluent,
\item $\ofl{G_\preccurlyeq}$ is locally confluent,
\item $\ofl{G_\preccurlyeq}$ has unique normal forms,
\item $f\overset{\displaystyle G_\preccurlyeq}{\fll^\ast} 0$, for all polynomial $f$ in $\ideal{G}$,
\item every polynomial $f$ in $\ideal{G}\setminus\{0\}$ is reducible modulo $G$,
\item for any term $t$ in $\lt_\preccurlyeq(\ideal{G})$, there is $g$ in $G$ such that $\lt_\preccurlyeq(g)$ divides $t$,
\item $S_{\preccurlyeq}(g_1,g_2) \overset{\displaystyle G}{\fll^\ast} 0$ for all $g_1,g_2$ in $G$, where 
\[
S_{\preccurlyeq}(g_1,g_2) = \frac{\mu}{\lt_\preccurlyeq(g_1)} g_1 - \frac{\mu}{\lt_\preccurlyeq(g_2)}g_2,
\]
with $\mu = \mathrm{ppcm}(\lm_\preccurlyeq(g_1),\lm_\preccurlyeq(g_2))$, is the \emph{$S$-polynomial} of $g_1$ and $g_2$ with respect to the monomial order $\preccurlyeq$,
\item any critical pair 
\[
\xymatrix{
& 
\mu 
\ar[dl] _-{\frac{\mu}{\lt(g_1)} g_1}
\ar[dr] ^-{\frac{\mu}{\lt(g_2)} g_2} 
&\\
\mu - \frac{\mu}{\lt(g_1)}g_1
&&
\mu - \frac{\mu}{\lt(g_2)}g_2
}
\]
with $\mu = \mathrm{ppcm}(\lm(g_1),\lm(g_2))$, of the relation $\ofl{G}$ is confluent.
\end{enumerate}

We refer the reader to {\cite[Theorem 5.35]{BeckerWeispfenning93}} for proofs of these equivalences, see also {\cite[Section 3]{GuiraudHoffbeckMalbos19}}, \cite{Malbos19ACM}. The equivalence of conditions {\bf i)}-{\bf iv)} are classical results for terminating rewriting systems. Note that condition {\bf viii)} corresponds to the Buchberger criterion, \cite{Buchberger65}, and condition {\bf ix)} is a formulation of this criterion in rewriting terms. We refer to {\cite[Chapter 8]{BaaderNipkow98}} for the rewriting interpretation of the Buchberger algorithm.

\bigskip

A \emph{Gr\"{o}bner basis of an ideal $I$} of $\K[x_1,\ldots,x_n]$ with respect to a monomial order $\preccurlyeq$ is a Gr\"{o}bner basis with respect to  $\preccurlyeq$  that generates the ideal $I$. This can be also be formulated saying that $G$ is a generating set for $I$ such that $\ideal{\lt(G )} = \ideal{\lt(I)}$.

\subsubsection{Involutive bases and Gr\"{o}bner bases}
Let $I$ be an ideal of $\K[x_1,\ldots,x_n]$. 
Suppose that $G$ is an involutive basis of the ideal $I$ with respect to an involutive division $\div{I}$ and the monomial order $\preccurlyeq$. In particular, the set $G$ generates the ideal $I$. For every $g_1$ and $g_2$ in $G$, we consider the $S$-polynomial $S_{\preccurlyeq}(g_1,g_2)$ with respect to $\preccurlyeq$. By definition, the polynomial~$S_{\preccurlyeq}(g_1,g_2)$ belongs to the ideal $I$.
By involutivity of the set $G$ and following \ref{SSS:UnicityNormalForms} and Proposition~\ref{Proposition:formeNormaleNulle}, we have 
\[
\normf^G(S_{\preccurlyeq}(g_1,g_2))=\normf_\div{I}^G(S_{\preccurlyeq}(g_1,g_2))=0.
\] 
In this way, $G$ is a Gr\"{o}bner basis of the ideal $I$ by the Buchberger criterion {\bf viii)}. We have thus proved the following result due to V.P. Gerdt and Y.A. Blinkov.

\begin{theorem}[{\cite[Corollary 7.2]{GerdtBlinkovYuri98}}]
Let $\preccurlyeq$ be a monomial order on $\Mr(x_1,\ldots,x_n)$ and $\div{I}$ be an involutive division on $\K[x_1,\ldots,x_n]$.
Any $\div{I}$-involutive basis of an ideal $I$ of $\K[x_1,\ldots,x_n]$ is a Gr\"{o}bner basis of $I$.
\end{theorem}

The involutive division used to define involutive bases being a refinement of the classical division with respect to which the Gröbner bases are defined, the converse of this result is false in general.

\begin{small}
\bibliographystyle{alpha}
\bibliography{biblioCURRENT}
\end{small}

\clearpage

\quad

\vfill

\begin{footnotesize}
\auteur{Kenji Iohara}{iohara@math.univ-lyon1.fr}
{Universit\'{e} de Lyon,\\
Institut Camille Jordan, CNRS UMR 5208\\
Universit\'{e} Claude Bernard Lyon 1\\
43, boulevard du 11 novembre 1918,\\
69622 Villeurbanne cedex, France}

\bigskip

\auteur{Philippe Malbos}{malbos@math.univ-lyon1.fr}
{Universit\'{e} de Lyon,\\
Institut Camille Jordan, CNRS UMR 5208\\
Universit\'{e} Claude Bernard Lyon 1\\
43, boulevard du 11 novembre 1918,\\
69622 Villeurbanne cedex, France}
\end{footnotesize}

\vspace{3cm}

\begin{small}---\;\;\today\;\;-\;\;\hhmm\;\;---\end{small} \hfill
\end{document}